\documentclass[10pt, twosides, a4paper]{article}
\usepackage{theorem,makeidx,float,calc,verbatim}
\usepackage{latexsym,amscd}
\usepackage{amsmath,amsfonts,amssymb,mathrsfs}
\usepackage{enumerate,qsymbols,euscript,graphicx}

\input cyracc.def

\input xy
\xyoption{all}
\newdir{ >}{{}*!/-5pt/@{>}}

\newtheorem{theorem}{Theorem}[section]
\newtheorem{lemma}[theorem]{Lemma}
\newtheorem{proposition}[theorem]{Proposition}
\newtheorem{corollary}[theorem]{Corollary}
\theorembodyfont{\upshape}
\newtheorem{definition}[theorem]{Definition}
\newtheorem{remark}[theorem]{Remark}
\newtheorem{example}[theorem]{Example}

\newenvironment{proof}{%
\noindent{\it Proof.}\hskip 10 pt%
}{%
{\quad} \hfill$\Box$\\%
}

\def\og{\leavevmode\raise.3ex\hbox{$\scriptscriptstyle\langle\!\langle$~}}
\def\fg{\leavevmode\raise.3ex\hbox{~$\!\scriptscriptstyle\,\rangle\!\rangle$}}
\def\endzm{\\ \hspace*{\fill} $\square$\bigskip}

\DeclareMathOperator{\ann}{ann}
\DeclareMathOperator{\card}{card}
\DeclareMathOperator{\CH}{CH}

\DeclareMathOperator{\des}{des}
\renewcommand{\det}{\mathrm{det}}
\DeclareMathOperator{\diag}{diag}
\DeclareMathOperator{\Div}{div}

\DeclareMathOperator{\esp}{\mathbb{E}}

\DeclareMathOperator{\GL}{GL}
\DeclareMathOperator{\HN}{HN}

\DeclareMathOperator{\Image}{Im}
\DeclareMathOperator{\indic}{1\!\!1}
\DeclareMathOperator{\Ker}{Ker}
\DeclareMathOperator{\lbr}{{[\![}}
\DeclareMathOperator{\len}{len}

\DeclareMathOperator{\Pic}{Pic}

\DeclareMathOperator{\rang}{rk}
\DeclareMathOperator{\rbr}{{]\!]}}
\DeclareMathOperator{\Spec}{Spec}
\DeclareMathOperator{\supp}{supp}

\normalsize
 \oddsidemargin
= 10 pt \evensidemargin = 32 pt

\begin{document}

\title{Convergence of Harder-Narasimhan polygons}
\author{Huayi {\sc Chen}\thanks{CMLS, Ecole Polytechnique, Palaiseau 91120, France. (huayi.chen@polytechnique.org)}}
\maketitle

\begin{abstract}
We establish in this article convergence results of normalized Harder-Narasimhan polygons both in geometric and in arithmetic frameworks by introducing the Harder-Narasimhan filtration indexed by $\mathbb R$ and the associated Borel probability measure.
\end{abstract}

\section{Introduction}

\hskip\parindent Let $X$ be a projective
variety of dimension $\ge 1$ over a field $k$
and $L$ be an ample line bundle on $X$. The
Hilbert-Samuel theorem describing the
asymptotic behaviour of $\rang
H^0(X,L^{\otimes D})$ ($D\rightarrow\infty$)
is an important result in commutative algebra
and in algebraic geometry, which is largely
studied since Hilbert's article
\cite{Hilbert90}. Although numerous variants
and generalizations of this theorem have been
developed, many proofs have a common feature
--- the technic of unscrewing
(``d\'evissage'' in French). Let us recall a
variant of Hilbert-Samuel theorem in relative
geometric framework. Suppose that $k$ is a
field and $C$ is a non-singular projective
curve over $\Spec k$. We denote by $K=k(C)$
the field of rational functions on $C$. Let
$\pi:X\rightarrow C$ be a projective and flat
$k$-morphism and $L$ be an invertible
$\mathcal O_X$-module which is ample
relatively to $\pi$. We denote by $d$ the
relative dimension of $X$ over $C$. The
Riemann-Roch theorem implies that
\[\deg(\pi_*(L^{\otimes D}))=
\frac{c_1(L)^{d+1}}{(d+1)!}D^{d+1}+O(D^d)\qquad
(D\rightarrow\infty) .\] Combining with the
classical Hilbert-Samuel theorem
\[\rang(\pi_*(L^{\otimes D}))=\rang
H^0(X_K,L_K^{\otimes
D})=\frac{c_1(L_K)^d}{d!}D^d+O(D^{d-1}),\] we
obtain the asymptotic formula
\begin{equation}\label{Equ:HS geometric}\lim_{D\rightarrow\infty}\frac{
\mu(\pi_*(L^{\otimes
D}))}{D}=\frac{c_1(L)^{d+1}}
{(d+1)c_1(L_K)^d},\end{equation} where the
{\it slope} $\mu$ of a non-zero locally free
$\mathcal O_C$-module of finite type (in
other words, non-zero vector bundle on $C$)
is by definition the quotient of its degree
by its rank. For a non-zero vector bundle $E$
on $C$, there exists invariant which is much
shaper than the slope. Namely, Harder and
Narasimhan have proved in \cite{Harder-Nara}
that there exists a non-zero subbundle
$E_{\des}$ whose slope is maximal among the
slopes of non-zero subbundles of $E$ and
which contains all non-zero subbundles of $E$
having the maximal slope. The slope of
$E_{\des}$ is denoted by $\mu_{\max}(E)$,
called the {\it maximal slope} of $E$. We say
that $E$ is {\it semistable} if and only if
$E=E_{\des}$, or equivalently
$\mu(E)=\mu_{\max}(E)$. By induction we
obtain a sequence
\[0=E_0\subsetneq E_1\subsetneq E_2\subsetneq\cdots
\subsetneq E_n=E\] of saturated subbundles of
$E$ such that
$(E_i/E_{i-1})=(E/E_{i-1})_{\des}$ for any
$1\le i\le n$. This sequence is called the
{\it Harder-Narasimhan flag} of $E$. Clearly
each sub-quotient $E_i/E_{i-1}$ is semistable
and we have
$\mu(E_1/E_0)>\mu(E_2/E_1)>\cdots>\mu(E_n/E_{n-1})$.
The last slope $\mu(E_n/E_{n-1})$ is called
the {\it minimal slope} of $E$, denoted by
$\mu_{\min}(E)$. Note that the
Harder-Narasimhan flag of $E_i$ is just
$0=E_0\subsetneq
E_1\subsetneq\cdots\subsetneq E_i$ and
therefore $\mu_{\min}(E_i)=\mu(E_i/E_{i-1})$.

Recall that the {\it Harder-Narasimhan
polygon} of $E$ is by definition the concave
function $\widetilde P_E$ on the interval
$[0,\rang E]$ whose graph is the convex hull
of points $(\rang F,\deg(F))$, where $F$ runs
over all subbundles of $E$. Therefore, the
function $\widetilde P_E$ takes zero value at
origin; it is piecewise linear and its slope
on the interval $[\rang E_{i-1},\rang E_i]$
is $\mu(E_i/E_{i-1})$. Let $P_E$ be the
function defined on $[0,1]$ whose graph is
similar to that of $\widetilde P_E$, namely
$P_E(t)=\widetilde P_E(t\rang E)/\rang E$,
called the {\it normalized Harder-Narasimhan
polygon of $E$}. Notice that $P_E(1)=\mu(E)$.
Therefore \eqref{Equ:HS geometric} can be
reformulated as
\begin{equation}\lim_{D\rightarrow\infty}\frac{P_{\pi_*(L^{\otimes
D})}(1)}{D}=\frac{c_1(L)^{d+1}}{(d+1)c_1(L_K)^d}.\end{equation}
It is then quite natural to study the
convergence at other points in $[0,1]$. Here
the major difficulty is that, unlike the
degree function $\widetilde P_E(1)\rang E$,
for other points $r\in ]0,\rang E[ $, the
function $E\rightarrow \widetilde P_E(r)\rang
E$ need not be additive with respect to short
exact sequences. Therefore the unscrewing
technic doesn't work.

The original idea of this article is to use
Borel probability measures on $\mathbb R$ to
study Harder-Narasimhan polygons.
\begin{figure}[h]
\caption{Derivative of the polygon and the
corresponding distribution function}
\label{Fig:Derivative of HN}
\begin{center}
\begin{tabular}{cc}\hspace{15 pt}
\begin{picture}(150,200)
\put(-8,93){$\scriptstyle 0$}
\put(10,100){\circle*{1}}
\put(6,93){$\scriptstyle t_{1}$}
\put(22,100){\circle*{1}}
\put(18,93){$\scriptstyle t_2$}
\put(37,100){\circle*{1}}
\put(33,93){$\scriptstyle t_3$}
\put(55,93){$\cdots$}
\put(80,100){\circle*{1}}
\put(75,93){$\scriptstyle t_{n2}$}
\put(82,93){\line(1,0){1}}
\put(93,100){\circle*{1}}
\put(88.5,93){$\scriptstyle t_{n1}$}
\put(95.5,93){\line(1,0){1}}
\put(100,100){\circle*{1}}
\put(100,93){$\scriptstyle 1$}
\put(0,34){\circle*{1}}
\put(-13,34){$\scriptstyle\mu_n$}
\put(0,70){\circle*{1}}
\put(-13,70){$\scriptstyle\mu_{n1}$}
\put(-4,70){\line(1,0){1}}
\put(0,190){\circle*{1}}
\put(-13,190){$\scriptstyle \mu_1$}
\put(0,160){\circle*{1}}
\put(-13,160){$\scriptstyle\mu_2$}
\put(-9,80){$\scriptstyle\vdots$}
\put(-9,105){$\scriptstyle\vdots$}
\put(0,120){\circle*{1}}
\put(-13,120){$\scriptstyle\mu_3$}
\put(0,100){\vector(1,0){110}}
\put(0,0){\vector(0,1){200}}
\put(0,190){\line(1,0){10}}
\put(10,160){\line(1,0){12}}
\put(22,120){\line(1,0){15}}
\put(80,70){\line(1,0){13}}
\put(93,34){\line(1,0){7}}

\linethickness{0.7 pt}
\put(10,190){\line(0,-1){30}}
\put(22,160){\line(0,-1){40}}
\put(37,120){\line(0,-1){7}}
\put(80,70){\line(0,1){7}}
\put(93,70){\line(0,-1){36}}
\put(100,34){\line(0,-1){34}}

\put(41,102){$\scriptstyle\ddots$}
\put(68,77){$\scriptstyle\ddots$}
\end{picture}
& \hspace{25 pt}\begin{picture}(200,100)
\put(0,0){\vector(1,0){200}}
\put(100,0){\vector(0,1){110}}
\put(34,0){\circle*{1}}
\put(70,0){\circle*{1}}
\put(120,0){\circle*{1}}
\put(160,0){\circle*{1}}
\put(190,0){\circle*{1}}
\put(100,10){\circle*{1}}
\put(100,22){\circle*{1}}
\put(100,37){\circle*{1}}
\put(100,80){\circle*{1}}
\put(100,93){\circle*{1}}
\put(100,100){\circle*{1}}
\put(93,100){$\scriptstyle 1$}
\put(102,91){$\scriptstyle t_{n1}$}
\put(109,91){\line(1,0){1}}
\put(102,78){$\scriptstyle t_{n2}$}
\put(109,78){\line(1,0){1}}
\put(93,35){$\scriptstyle t_{3}$}
\put(93,20){$\scriptstyle t_{2}$}
\put(93,8){$\scriptstyle t_{1}$}
\put(98,-7){$\scriptstyle 0$}
\put(32,-7){$\scriptstyle \mu_{n}$}
\put(68,-7){$\scriptstyle \mu_{n1}$}
\put(76.5,-7){\line(1,0){1}}
\put(83,-7){$\scriptstyle\cdots$}
\put(105,-7){$\scriptstyle\cdots$}
\put(118,-7){$\scriptstyle \mu_{3}$}
\put(158,-7){$\scriptstyle \mu_{2}$}
\put(188,-7){$\scriptstyle \mu_{1}$}
\put(95,55){$\scriptstyle\vdots$}
\put(80,68){$\scriptstyle\ddots$}
\put(102,40){$\scriptstyle\ddots$}
\put(14,102){$\scriptstyle E_n$}
\put(47,95){$\scriptstyle E_{n1}$}
\put(57,95){\line(1,0){1}}
\put(72,82){$\scriptstyle E_{n2}$}
\put(82,82){\line(1,0){1}}
\put(113,39){$\scriptstyle E_{3}$}
\put(137,24){$\scriptstyle E_{2}$}
\put(174,12){$\scriptstyle E_{1}$}
\put(191,2){$\scriptstyle E_{0}$}

\put(34,100){\line(0,-1){7}}
\put(70,93){\line(0,-1){13}}
\put(120,37){\line(0,-1){15}}
\put(160,22){\line(0,-1){12}}
\put(190,10){\line(0,-1){10}}
\linethickness{0.7 pt}
\put(0,100){\line(1,0){34}}
\put(34,93){\line(1,0){36}}
\put(70,80){\line(1,0){7}}
\put(120,37){\line(-1,0){7}}
\put(120,22){\line(1,0){40}}
\put(160,10){\line(1,0){30}}
\end{picture}
\end{tabular}
\end{center}
\end{figure}
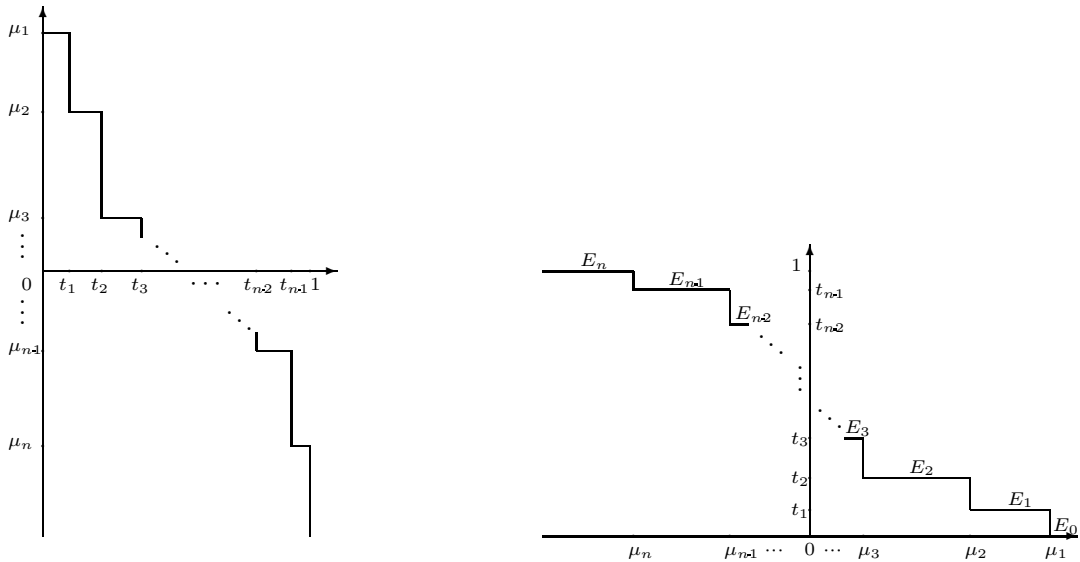
In Figure \ref{Fig:Derivative of HN}, the
left graph presents the first order
derivative of the normalized
Harder-Narasimhan polygon of $E$, where
$\mu_i=\mu(E_{i}/E_{i-1})$ for $1\le i\le n$
and $t_i=\rang E_i/\rang E$ for $0\le i\le
n$. It is a step function on $[0,1]$. The
right graph presents a decreasing step
function on $\mathbb R$ valued in $[0,1]$
whose {\it quasi-inverse} corresponds to the
left graph. Furthermore, this function is the
difference between the constant function $1$
and a probability distribution function and
therefore corresponds to a Borel probability
measure
$\nu_E=\displaystyle\sum_{i=1}^n(t_i-t_{i-1})\delta_{\mu_i}$,
where $\delta_x$ is the Dirac measure at the
point $x$. If we place suitably the
subbundles in the Harder-Narasimhan flag of
$E$ on the right graph, we obtain a
decreasing $\mathbb R$-filtration of the
vector bundle $E$, which induces naturally by
restricting to the generic fiber a decreasing
$\mathbb R$-filtration $\mathcal F^{\HN}$ of
the vector space $E_K$, called the {\it
Harder-Narasimhan filtration} of $E_K$. As we
shall show later, the filtration $\mathcal
F^{\HN}$ can be calculated explicitly from
the vector bundle $E$, namely
\[\mathcal F_r^{\HN} E_K=\sum_{
\begin{subarray}{c}
0\neq F\subset E\\
\mu_{\min}(F)\ge r
\end{subarray}}F_K.\]
From this filtration, one can recover easily
the probability measure
\[\nu_E=\frac{1}{\rang E_K}
\sum_{r\in\mathbb R}\Big(\rang(\mathcal
F_r^{\HN}E_K)-\lim_{\varepsilon\rightarrow
0+}\rang(\mathcal
F_{r+\varepsilon}^{\HN}E_K)\Big)\delta_r.\]
Furthermore, the function presented in the
right graph is just $r\mapsto\rang(\mathcal
F_r^{\HN}E_K)$. By passing to quasi-inverse
(turning over the graph), we retrieve the
first order derivative of the normalized
Harder-Narasimhan polygon. This procedure is
quite general and it works for an arbitrary
(suitably) filtered finite dimensional vector
space, where the word ``suitably'' means that
the filtration is separated, exhaustive and
left continuous, which we shall explain later
in this article. Actually, we have natural
mappings
\[\begin{array}{ccccc}\Big\{\parbox[c]{3.9 cm}{(suitably) filtered finite dimensional vector spaces}\Big\}
&\longrightarrow&\Bigg\{\parbox[c]{4.4
cm}{Borel probability measures on $\mathbb R$
which are linear combinations of Dirac
measures }\Bigg\}&\longleftrightarrow&
\Big\{\parbox[c]{1.5 cm} {polygones on
$[0,1]$}\Big\},\vspace{0.2 cm}\\
V&\longmapsto&\nu_V&\longmapsto&P_V
\end{array}\] the last
mapping being a bijection. If a probability
measure $\nu$ corresponds to the polygon $P$,
then we can verify that, for any real number
$\varepsilon>0$, the probability measure
corresponding to $\varepsilon P$ is the
direct image $T_{\varepsilon}\nu$ of $\nu$ by
the dilation mapping $x\mapsto\varepsilon x$.

Let us go back to the convergence of
polygons. To verify that a sequence of
polygons converges uniformly, it suffice to
prove that the corresponding sequence of
measures converges vaguely to a probability
measure. We state the main theorem of this
article.

\begin{theorem}\label{Thm:main theorem}
Let $f:\mathbb Z_{\ge 0}\rightarrow\mathbb
R_{\ge 0}$ be a function such that
$\displaystyle\lim_{n\rightarrow+\infty}{f(n)}/{n}=0$
and $B=\bigoplus_{n\ge 0}B_n$ be an integral
graded $K$-algebra of finite type over $K$.
Suppose that
\begin{enumerate}[i)]
\item for sufficiently large integer $n$,
the vector space $B_n\neq 0$,
\item for any positive integer $n$,
$B_n$ is equipped with an $\mathbb
R$-filtration $\mathcal F$ which is
separated, exhaustive and left continuous,
such that $B$ is an $f$-quasi-filtered graded
$K$-algebra,
\item $\sup\Big(\supp\nu_{B_n}\Big)=O(n)$.
\end{enumerate}
For any integer $n>0$, denote by
$\nu_n=T_{\frac 1n}\nu_{B_n}$. Then the
supports of $\nu_n$ are uniformly bounded and
the sequence of measures $(\nu_n)_{n\ge 1}$
converges vaguely to a Borel probability
measure on $\mathbb R$. Therefore, the
sequence of polygons $(\frac 1n P_{B_n})$
converges uniformly to a concave function on
$[0,1]$.
\end{theorem}
To apply the above theorem to the convergence
of
$\Big(\displaystyle\frac{1}{D}P_{\pi_*(L^{\otimes
D})}\Big)_{D\ge 1}$, we point out that the
graded algebra $\bigoplus_{D\ge 0
}H^0(X_K,L_K^{\otimes D})$, equipped with
Harder-Narasimhan filtrations, verifies the
conditions in Theorem \ref{Thm:main theorem}
for a suitable constant function $f$. The
verification of this fact is easy. However,
the proof of the theorem requires quite
subtle technical arguments on almost
super-additive sequences and on combinatorics
of monomials, which will be presented in
Section \ref{Sec:Almost super-additive
sequence} and in Section \ref{Sec:Convergence
for symmetric algebras} respectively. The
idea is to prove that the sequence of
measures $(\nu_n)_{n\ge 1}$ is ``vaguely
super-additive'', and then apply a variant of
Fekete's lemma to conclude the vague
convergence.

We are now able to state our geometric
convergence theorem.
\begin{theorem}
With the notations above, the sequence of
polygons
$\Big(\displaystyle\frac{1}{D}P_{\pi_*(L^{\otimes
D})}\Big)_{D\ge 1}$  converges uniformly to a
concave function on $[0,1]$.
\end{theorem}

The analogue of the formula \eqref{Equ:HS
geometric} in Arakelov geometry was firstly
proved by Gillet and Soul\'e
\cite{Gillet-Soule}, using their arithmetic
Riemann-Roch theorem. Since then, this
subject has been rediscussed by many authors
such as Abbes and Bouche \cite{Abbes-Bouche},
Zhang \cite{Zhang95}, Rumely, Lau and Varley
\cite{Rumely_Lau_Varley} and Randriambololona
\cite{Randriam06}.

Let $K$ be a number field and $\mathcal O_K$
be its algebraic integer ring. We denote by
$\Sigma_\infty$ the set of all embedding of
$K$ in $\mathbb C$. If $\mathscr X$ is a
projective arithmetic variety (i.e. scheme of
finite type, projective and flat) over
$\Spec\mathcal O_K$, we call {\it Hermitian
vector bundle} on $\mathscr X$ any pair
$\overline{\mathscr E}=(\mathscr
E,(\|\cdot\|_\sigma)_{\sigma\in\Sigma_\infty})$,
where $\mathscr E$ is a locally free
$\mathcal O_{\mathscr X}$-module of finit
type, and $\|\cdot\|_\sigma$ is a continuous
Hermitian metric on $\mathscr
E_\sigma(\mathbb C)$ such that the collection
$(\|\cdot\|_\sigma)_{\sigma\in\Sigma_\infty}$
is invariant under the action of the
complex conjugation. We call {\it Hermitian
line bundle} on $\mathscr X$ any Hermitian
vector bundle $\overline{\mathscr L}$ on
$\mathscr X$ such that $\mathscr L$ is of
rank one. Notice that a Hermitian vector
bundle on $\Spec\mathcal O_K$ is nothing but
the pair $\overline
E=(E,(\|\cdot\|_{\sigma})_{\sigma\in\Sigma_\infty})$,
where $E$ is a projective $\mathcal
O_K$-module of finite type, and
$\|\cdot\|_\sigma$ is a Hermitian metric on
$E_{\sigma}:=E\otimes_{\mathcal
O_K,\sigma}\mathbb C$ such that the
collection
$(\|\cdot\|_\sigma)_{\sigma\in\Sigma_\infty}$
is invariant by the complex conjugation. If
$\overline E$ is a Hermitian vector bundle of
rank $r$ on $\Spec\mathcal O_K$, we define
its (normalized) {\it Arakelov degree} to be
\[\widehat{\deg}_n(\overline E):=
\frac{1}{[K:\mathbb
Q]}\Big(\log\#\big(E/(\mathcal
O_Ks_1+\cdots+\mathcal O_Ks_r)\big)-\frac
12\sum_{\sigma\in\Sigma_\infty}\log\det(\left<s_i,s_j\right>_\sigma)\Big),\]
where $(s_1,\cdots,s_r)\in E^r$ is an
arbitrary element in $E^r$ which defines a
base of $E_K$ over $K$ (see
\cite{BostBour96}, \cite{Bost2001},
\cite{Chambert}, \cite{Bost2004} and
\cite{Bost_Kunnemann} for details). This is
an analogue in Arakelov geometry of the
classical notion of degree of a vector bundle
on a smooth projective curve. Recall that the
{\it slope} of a non-zero Hermitian vector
bundle $\overline E$ on $\Spec\mathcal O_K$
is by definition the quotient
$\widehat{\mu}(\overline
E):=\widehat{\deg}_n(\overline E)/\rang(E)$.
We denote by $\widehat{\mu}_{\max}(\overline
E)$ the maximal value of slopes of non-zero
Hermitian subbundle of $\overline E$, and by
$\widehat{\mu}_{\min}(\overline E)$ the
minimal value of slopes of non-zero Hermitian
quotient bundle of $\overline E$. Stuhler
\cite{Stuhler76} and Grayson \cite{Grayson76}
have proved that there exists a non-zero
Hermitian subbundle $\overline E_{\des}$ of
$\overline E$ such that
$\widehat{\mu}(\overline
E_{\des})=\widehat{\mu}_{\max}(\overline E)$
and that $E_{\des}$ contains all Hermitian
subbundle of $\overline E$ having the maximal
slope. We obtain therefore a flag
\[0=E_0\subsetneq E_1\subsetneq E_2\subsetneq\cdots\subsetneq
E_n=E\] of $E$ such that
$\overline{E_i}/\overline E_{i-1}=(\overline
E/\overline E_{i-1})_{\des}$ for any integer
$1\le i\le n$, and that
\[\widehat{\mu}_{\max}(\overline E)=
\widehat{\mu}(\overline E_1/\overline E_0)
>\widehat{\mu}(\overline E_2/\overline E_1)>
\cdots>\widehat{\mu}(\overline E_n/\overline
E_{n-1})=\widehat{\mu}_{\min}(\overline E).\]
The {\it Harder-Narasimhan polygon} of
$\overline E$ is by definition the concave
function $\widetilde P_{\overline E}$ defined
on $[0,\rang E]$ whose graph is the convex
hull of points of the form $(\rang
F,\widehat{\deg}_n(\overline F))$, where
$\overline F$ runs over all non-zero
sub-$\mathcal O_K$-modules of $E$ equipped
with induced metrics. The {\it normalized
Harder-Narasimhan polygon} of $\overline E$ is the
concave function $P_{\overline E}$ defined on
$[0,1]$ such that $P_{\overline
E}(t)=\widetilde P_{\overline E}(t\rang
E)/\rang E$. Notice that we have
$P_{\overline E}(1)=\widehat{\mu}(\overline
E)$. The measure theory approach in geometric
case works without any modification in
arithmetic case. Namely, to any non-zero
Hermitian vector bundle $\overline E$ on
$\Spec\mathcal O_K$, we associate a
decreasing filtration $\mathcal F^{\HN}$ of
$E_K$, called the {\it Harder-Narasimhan
filtration}, such that
\[\mathcal F_r^{\HN}E_K=\sum_{\begin{subarray}{c}
0\neq F\subset E\\
\widehat{\mu}_{\min}(\overline F)\ge r
\end{subarray}}F_K.\]
This filtration induces a Borel probability
measure $\nu_{\overline E}$ on $\mathbb R$
such that $\nu_{\overline
E}([r,+\infty[)=\rang(\mathcal
F_r^{\HN}E_K)/\rang E$. Finally the
normalized Harder-Narasimhan polygon
$P_{\overline E}$ is uniquely determined by
$\nu_{\overline E}$.

Using Theorem \ref{Thm:main theorem}, we
obtain the following arithmetic convergence
theorem.
\begin{theorem}
Let $\pi:\mathscr X\rightarrow\Spec\mathcal
O_K$ be a projective arithmetic variety and
$\overline{\mathscr L}$ be a Hermitian line
bundle on $\mathscr X$ such that the graded
algebra $\bigoplus_{D\ge 0}H^0(\mathscr
X_K,\mathscr L_K^{\otimes D})$ is of finite
type over $K$, and that $H^0(\mathscr
X_K,\mathscr L_K^{\otimes D})\neq 0$ for
$D\gg 0$. Then the sequence of polygons
$(\frac{1}{D}P_{\pi_*(\overline{\mathscr
L}^{\otimes D})})_{D\ge 1}$ converges
uniformly to a concave function on $[0,1]$.
\end{theorem}

Contrary to the geometric case, the
verification of the fact that the algebra
$\bigoplus_{D\ge 0}H^0(\mathscr X_K,\mathscr
L_K^{\otimes D})$ equipped with
Harder-Narasimhan filtrations is an
$f$-quasi-filtered graded algebra for a
function $f$ of logarithmic increasing speed
at infinity is subtle, which depends on the
author's recent work \cite{Chen_pm} on an
upper bound of the maximal slope of the
tensor product of several Hermitian vector
bundles.

The article is organized as follows. In the
second section, we introduce the notion of
$\mathbb R$-filtrations of a vector space
over a field and its various properties. We
also explain how to associate to each
filtered vector space of finite rank a Borel
measure on $\mathbb R$, which is a
probability measure if the vector space is
non-zero. The third section is devoted to
a generalization of Fekete's lemma on
sub-additive sequences, which is useful in
sequel. We present the main object of this
article
--- quasi-filtered graded algebras in the
fourth section. Then in the fifth section we
establish the vague convergence of measures
associated to a quasi-filtered symmetric
algebra. In the sixth section we explain how
to construct the polygon associated to a
Borel probability measure which is a linear
combination of Dirac measures. We show that
the vague convergence of probability measures
implies the uniform convergence of associated
polygons. Combining the results obtained in
previous sections, we establish in the
seventh section the uniform convergence of
polygons associated to a general
quasi-filtered graded algebra. In the eighth
and the ninth sections we apply the general
result in the seventh section to relative
geometric framework and to Arakelov geometric
framework respectively to obtain the
corresponding convergence of
Harder-Narasimhan polygons. Finally in the
tenth section, we propose another approach,
inspired by Faltings and W\"ustholz \cite{Fal-Wusth}, to
calculate explicitly the limit of the
polygons. We conclude by providing an
explicit example where the limit of the
polygons is a non-trivial quadratic curve on
$[0,1]$. In the appendix, we develop a
variant of $f$-quasi-filtered graded algebra
--- $f$-pseudo-filtered graded algebra, where
we require less algebraic conditions. With a
stronger condition on the increment of $f$,
we also obtain the convergence of polygons.
Although this approach has not been used in
this article, it may have applications
elsewhere and therefore we include it as
well. \vspace{2mm}

{\bf Acknowledgement.} The results presented
in this article is part of the author's
doctoral thesis supervised by J.-B. Bost. The
author would like to thank him for having
proposed the author to study the convergence
of Harder-Narasimhan polygons, for his
encouragement and for discussions.

\section{Filtrations of vector spaces}

\hskip\parindent We present some basic
definition and properties of filtrations of
vector spaces. Although the notion of
filtrations in a general category has been
discussed in \cite{Chen_hn}, we would like to
introduce it in an explicit way for the
particular case of vector spaces.

We fix in this section a field $K$. A
(decreasing) $\mathbb R$-{\it filtration} of
a vector space $V$ is by definition a
collection $\mathcal F=(\mathcal
F(r))_{r\in\mathbb R}$ of $K$-vector
subspaces of $V$ such that $\mathcal
F(r)\supset\mathcal F(r')$ if $r\le r'$. We
shall use the expression $\mathcal F_rV$ to
denote $\mathcal F(r)$, or simply $V_r$ if
there is no ambiguity on the filtration
$\mathcal F$. An $\mathbb R$-filtration
$\mathcal F$ is said to be {\it separated} if
$\displaystyle V_{+\infty}:{=}\displaystyle
\bigcap_{r\in\mathbb R}V_r=\{0\}$, and to be
{\it exhaustive} if $\displaystyle
V_{-\infty}:{=}
\displaystyle\bigcup_{r\in\mathbb R}V_r=V$.

Let $V$ be a vector space over $K$, $\mathcal
F$ be an $\mathbb R$-filtration of $V$. For
any element $x\in V$, we call {\it index} of
$x$ {\it relatively to} $\mathcal F$ the
element $\sup\{r\in\mathbb R\;|\;x\in\mathcal
F_r V\}$ in $\mathbb R\cup\{\pm\infty\}$ (by
convention $\sup\emptyset=-\infty$), denoted
by $\lambda_{\mathcal F}(x)$, or simply
$\lambda(x)$ if there is no ambiguity on
$\mathcal F$. The mapping $\lambda_{\mathcal
F }:V\rightarrow\mathbb R\cup\{\pm\infty\}$
is called the {\it index function} of
$\mathcal F$.

Let $x$ be an element in $V$. The set
$\{r\in\mathbb R\;|\;x\in\mathcal F_rV\}$ is
non-empty if and only if
$\lambda(x)>-\infty$. In this case, it is
either of the form $]-\infty,\lambda(x)[$ or
of the form $]-\infty,\lambda(x)]$. The
following properties of the function
$\lambda$ are easy to verify:
\begin{enumerate}[1)]
\item $\lambda(x)=+\infty$ if and only if $x\in V_{+\infty}$,
\item $\lambda(x)=-\infty$ if and only if $x\in V\setminus V_{-\infty}$,
\item $\lambda(x)>r$ if and only if $x\in\displaystyle\bigcup_{s>r}V_s$,
\item $\lambda(x)\ge r$ if and only if $x\in\displaystyle\bigcap_{s<r}V_s$.
\end{enumerate}

We say that an $\mathbb R$-filtration
$\mathcal F$ of $V$ is {\it left continuous}
if and only if for any $r\in\mathbb R$,
$\mathcal F_rV=\bigcap_{s<r} \mathcal F_sV$.
If $\mathcal F$ is an arbitrary filtration of
$V$, then the filtration $\mathcal
F^l=(\bigcap_{s<r}\mathcal F_sV)_{r\in\mathbb
R}$ is a left continuous filtration of $V$.

For any element $x\in V$ and any $r\in\mathbb
R\cup\{+\infty\}$, the fact that
$x\in\mathcal F_rV$ implies $\lambda(x)\ge
r$. The converse is true when $\mathcal F$ is
left continuous.

\begin{proposition}\label{Pro:proprietes de la fonction lambda}
Let $V$ be a vector space over $K$ and
$\mathcal F$ be a filtration of $V$. The
following assertions hold:
\begin{enumerate}[1)]
\item if $a\in K^\times$ and if $x\in
V$, then $\lambda(ax)=\lambda(x)$,
\item if $x$ and $y$ are two elements of $V$, then
$\lambda(x+y)\ge\min(\lambda(x),\lambda(y))$,
\item if $x$ and $y$ are two elements of
$V$ such that $\lambda(x)\neq\lambda(y)$,
then $x+y\neq 0$, and
$\lambda(x+y)=\min(\lambda(x),\lambda(y))$,
\item if the rank of $V$ is finite, then the image of
$\lambda$ is a finite subset of
$\mathbb R\cup\{\pm\infty\}$ whose cardinal is bounded from above by $\rang_KV+1$.
\end{enumerate}
\end{proposition}
\begin{proof} 1) For any $a\in K^\times$, $x\in\mathcal F_rV$
if and only if $ax\in\mathcal F_rV$. So
$\displaystyle\{r\in\mathbb R\;|\;
x\in\mathcal F_rV\}=\{r\in\mathbb R\;|\;
ax\in\mathcal F_rV\}$, which implies that
$\lambda(x)=\lambda(ax)$.

2) In fact, $\{t\;|\;x+y\in\mathcal
F_tV\}\supset\{r\;|\;x\in\mathcal F_rV\}
\cap\{s\;|\;y\in\mathcal F_sV\}$. Therefore
$\sup\{t\;|\;x+y\in\mathcal
F_tV\}\ge\min(\sup\{r\;|\;x\in\mathcal
F_rV\}, \sup\{s\;|\;y\in\mathcal F_sV\})$.

3) If $x+y=0$, then $x=-y$. So
$\lambda(x)=\lambda(y)$ by 1), which leads to
a contradiction. Hence $x+y\neq 0$. We may
suppose that $\lambda(x)<\lambda(y)$. For any
$r\in]\lambda(x),\lambda(y)[$, we have
$y\in\mathcal F_rV$ but $x\not\in\mathcal
F_rV$. Therefore $x+y\not\in\mathcal F_rV$,
in other words, $\lambda(x+y)\le r$. Since
$r$ is arbitrary, we obtain
$\lambda(x+y)\le\lambda(x)$. Combining with
2), we get the equality.

4) Suppose that $x_1,\cdots,x_n$ are non-zero elements in
$V$ such that
$\lambda(x_1)<\lambda(x_2)<\cdots<\lambda(x_n)<+\infty$.
After 1) and 3), for any $(a_i)_{1\le
i\le n }\in K^n\setminus\{0\}$,
\[\lambda(a_1x_1+\cdots+a_nx_n)=\min\{\lambda(x_i)\;|\;a_i\neq 0\}<+\infty,\]
which implies that $a_1x_1+\cdots+a_nx_n\neq
0$. Therefore, $x_1,\cdots,x_n$ are linearly independent. So
$n\le\rang_K V$.
\end{proof}

Using the index function $\lambda$, we give
some numerical characterizations for
filtrations of vector spaces.

\begin{proposition}\label{Pro:crietere numerique de separe et exhaustive}
Let $V$ be a vector space over $K$ equipped
with an $\mathbb R$-filtration $\mathcal F$.
Then
\begin{enumerate}[1)]
\item the filtration $\mathcal F$ is separated if and only if for any $x\in
V\setminus\{0\}$, $\lambda(x)<+\infty$,
\item the filtration $\mathcal F$ is  exhaustive if and only if for any $x\in V$, $\lambda(x)>-\infty$.
\end{enumerate}
\end{proposition}
\begin{proof} 1) If the filtration is separated, then for any non-zero element $x$ of $V$, there exists $r\in\mathbb R$ such that
$x\not\in\mathcal F_rV$, so $\lambda(x)\le
r$. Conversely for any non-zero element $x\in V$ such that
$\lambda(x)<+\infty$, if
$\lambda(x)\in\mathbb R$, then $x\not\in
\mathcal F_{\lambda(x)+1}V$, otherwise
$\lambda(x)=-\infty$ and by definition
$x\not\in\mathcal F_{r}V$ for every
$r\in\mathbb R$.

2) If the filtration is exhaustive, then for any element $x$ of $V$, there exists
$r\in\mathbb R$ such that $x\in\mathcal F_rV$.
Hence $\lambda(x)\ge r$.
Conversely for any element $x\in V$ such that $\lambda(x)>-\infty$, either we have
$\lambda(x)\in\mathbb R$, and therefore $x\in\mathcal
F_{\lambda(x)-1}V$, or we have
$\lambda(x)=+\infty$ and $x\in\mathcal F_rV$
for every $r\in\mathbb R$.
\end{proof}

\begin{proposition}\label{Pro:invariance de lambda par la filtration continue associee}
Let $V$ be a vector space over $K$ and
$\mathcal F$ be a filtration of $V$.
\begin{enumerate}[1)]
\item For any element $x$ of $V$, we have
$\lambda_{\mathcal F}(x)=\lambda_{\mathcal
F^l}(x)$.
\item If $\mathcal F$ is separated (resp.
exhaustive), then also is $\mathcal F^l$.
\end{enumerate}
\end{proposition}
\begin{proof}
1) Since $\mathcal F_rV\subset\mathcal
F^l_rV$, we have $\lambda_{\mathcal
F}(x)\le\lambda_{{\mathcal F}^l}(x)$. On the
other hand, if $x\in\mathcal F^l_rV$, then
for any $s<r$, we have $x\in\mathcal F_sV$,
so $\lambda_{{\mathcal F}}(x)\ge r$. Hence
$\lambda_{\mathcal F}(x)\ge\lambda_{{\mathcal
F}^l}(x)$.

2) It's an easy consequence of 1) and
Proposition \ref{Pro:crietere numerique de
separe et exhaustive}.
\end{proof}

Consider now two vector spaces $V$ and $W$
over $K$. Let $\mathcal F$ be an $\mathbb
R$-filtration of $V$ and $\mathcal G$ be an
$\mathbb R$-filtration of $W$. We say that a
linear mapping $f:V\rightarrow W$ is {\it
compatible} with the filtrations $(\mathcal
F,\mathcal G)$ if for any $r\in\mathbb R$,
$f(\mathcal F_rV)\subset\mathcal G_rW$.

We introduce some functorial construction of
filtrations. Let $f:V\rightarrow W$ be a
$K$-linear mapping of vector spaces over $K$.
If $\mathcal G$ is an $\mathbb R$-filtration
of $W$, then the {\it inverse image} of
$\mathcal G$ by $f$ is by definition the
filtration $f^*\mathcal G$ of $V$ such that
$(f^*\mathcal G)_rV=f^{-1}(\mathcal G_rW)$.
Clearly, if $\mathcal G$ is left continuous,
then also is $f^*\mathcal G$. If $\mathcal F$
is an $\mathbb R$-filtration of $V$, the {\it
weak direct image} of $\mathcal F$ by $f$ is
by definition the filtration
$f_{\flat}\mathcal F$ of $W$ such that
$(f_{\flat}\mathcal F)_rW=f(\mathcal F_rV)$,
and the {\it strong direct image} of
$\mathcal F$ by $f$ is by definition the
filtration $f_*\mathcal F=(f_{\flat}\mathcal
F)^l$. Clearly the homomorphism $f$ is
compatible to filtrations $(f^*\mathcal
G,\mathcal G)$, $(\mathcal
F,f_{\flat}\mathcal F)$ and $(\mathcal
F,f_*\mathcal F)$.

\begin{proposition}
If a $K$-linear mapping $f:V\rightarrow W$ is
compatible with the filtrations $(\mathcal
F,\mathcal G)$, then for any $x\in V$, one
has $\lambda(f(x))\ge\lambda(x)$. The
converse is true if $\mathcal G$ is left
continuous.
\end{proposition}
\begin{proof}
``$\Longrightarrow$'': By definition we know
that $\{r\in\mathbb R\;|\;x\in\mathcal
F_rV\}\subset\{r\in\mathbb
R\;|\;f(x)\in\mathcal G_rW\}$ for any $x\in
V$, therefore $\lambda(x)\le\lambda(f(x))$.

``$\Longleftarrow$'': For any $r\in\mathbb R$
and any $x\in\mathcal F_rV$, we have
$\lambda(x)\ge r$, and hence
$\lambda(f(x))\ge r$. Therefore,
$f(x)\in\mathcal G_rW$ since the filtration
$\mathcal G$ is left continuous.
\end{proof}

\begin{proposition}\label{Pro:exhaustivite de filtration induite quotiente}
Let $f:V'\rightarrow V$ be an injective
homomorphism and $\pi:V\rightarrow V''$ be a
surjective homomorphism of vector spaces over
$K$. Suppose that $\mathcal F$ is an $\mathbb
R$-filtration of $V$. Then:
\begin{enumerate}[1)]
\item if $\mathcal F$ is separated, also is $f^*\mathcal
F$;
\item if $\mathcal F$ is separated and if
the rank of $V$ is finite, the filtration
$\pi_{\flat}\mathcal F$ is also separated;
\item if $\mathcal F$
is exhaustive, the filtrations
$f^*\mathcal F$, $\pi_{\flat}\mathcal F$ and
$\pi_*\mathcal F$ are all exhaustive.
\end{enumerate}
\end{proposition}
\begin{proof}
1) As $\mathcal F$ is separated,
$\displaystyle\bigcap_{r\in\mathbb R}\mathcal
F_rV=\{0\}$. Since $f$ is injectif, we have
\[\bigcap_{r\in\mathbb R}(f^*\mathcal F)_rV'=\bigcap_{r\in\mathbb R}f^{-1}(\mathcal F_rV)=f^{-1}\Big(\bigcap_{r\in\mathbb R}\mathcal F_rV\Big)=f^{-1}(\{0\})=\{0\}.\]
Therefore $f^*\mathcal F$ is also separated.

2) If $\rang E<+\infty$, then
$\lambda_{\mathcal F}$ takes only a finite
number of values. Let $r_0=\sup\Big(
\lambda_{\mathcal F
}(E)\setminus\{\pm\infty\}\Big)<+\infty$. For
any real number $r>r_0$ and any $x\in
\mathcal F_rV$ we have $\lambda_{\mathcal
F}(x)\ge r>r_0$, so $\lambda_{\mathcal F
}(x)=+\infty$, i.e., $x=0$ since the
filtration $\mathcal F$ is separated.
Therefore, $\mathcal F_rV=0$ and
$(\pi_{\flat}\mathcal F)_rV''=\pi(\mathcal
F_r V)=0$.

3) Since the filtration $\mathcal F$ is
exhaustive, we have
$\displaystyle\bigcup_{r\in\mathbb
R}\mathcal F_rV=V$. Therefore,
\begin{gather*}
\bigcup_{r\in\mathbb R}(f^*\mathcal F)_rV'=
\bigcup_{r\in\mathbb R}(V'\cap \mathcal F_rV)
=V'\cap\Big(\bigcup_{r\in\mathbb R}\mathcal F_rV\Big)
=V'\cap V=V',\\
\bigcup_{r\in\mathbb R}(\pi_{\flat}\mathcal
F )_rV''=\bigcup_{r\in\mathbb R}\pi(\mathcal
F_rV)=\pi\Big( \bigcup_{r\in\mathbb
R}\mathcal F_rV\Big)=\pi(V)=V''.
\end{gather*}
So the filtrations $f^*\mathcal F$ and
$\pi_{\flat}\mathcal F$ are exhaustive.
Finally, after Proposition
\ref{Pro:invariance de lambda par la filtration continue associee} 2), $\pi_*\mathcal F=(\pi_{\flat}\mathcal
F)^l$ is exhaustive.
\end{proof}

The following proposition gives index
description of functorial constructions of
filtrations.

\begin{proposition}\label{Pro:critere numerique de
filtration induit quotient} Let $V$ and $W$
be two finite dimensional vector spaces over
$K$, $\mathcal F$ be an $\mathbb R$-filtration of $V$, $\mathcal G$ be an
$\mathbb R$-filtration of $W$ and
$\varphi:V\rightarrow W$ be a $K$-linear
mapping.
\begin{enumerate}[1)]
\item Suppose that $\varphi$ is injective.
If $\mathcal F=\varphi^*\mathcal G$, then for
any $x\in V$, one has $\lambda_{\mathcal
F}(x)=\lambda_{\mathcal G}(\varphi(x))$. The
converse is true if both filtrations
$\mathcal F$ and $\mathcal G$ are left
continuous.
\item Suppose that $\varphi$ is surjective.
If $\mathcal G=\varphi_*\mathcal F$, then for
any $y\in W$, $\displaystyle\lambda_{\mathcal
G}(y)=\sup_{x\in\varphi^{-1}(y)}\lambda_{\mathcal
F}(x)$. The converse is true if both
filtrations $\mathcal F$ and $\mathcal G$ are
left continuous.
\end{enumerate}
\end{proposition}
\begin{proof} 1) ``$\Longrightarrow$'':
Since $\mathcal F =\varphi^*\mathcal G$, a
non-zero element $x\in V$ lies in $V_\lambda$
if and only if $\varphi(x)\in W_\lambda$,
hence $\lambda(x)=\sup\{r\in\mathbb
R\;|\;x\in W_r\}=\sup\{r\in \mathbb
R\;|\;\varphi(x)\in
V_r\}=\lambda(\varphi(x))$.

``$\Longleftarrow$'': If $x\in V_r$, then
$\lambda(\varphi(x))\ge\lambda(x)\ge r$. So
$\displaystyle\varphi(x)\in W_r$
since the filtration $\mathcal G$ is left
continuous. On the other hand, if $0\neq
x\in\varphi^{-1}(W_r)$, then
$\lambda(x)=\lambda(\varphi(x))\ge r$, so
$\displaystyle x\in V_r$
since the filtration $\mathcal F$ of $V$ is
left continuous. Therefore
$V_r=\varphi^{-1}(W_r)$.

2) ``$\Longrightarrow$'': If $x\in V_r$, then
$\varphi(x)\in W_r$, so
$\lambda(\varphi(x))\ge\lambda(x)$. Hence for
any $y\in W\setminus\{0\}$,
$\displaystyle\lambda(y)\ge\sup_{x\in\varphi^{-1}(y)}\lambda(x)$.
On the other hand, $y\in W_r$ implies that
$V_s\cap\varphi^{-1}(y)\neq\emptyset$ for any
$s<r$. Therefore $\displaystyle
r\le\sup_{x\in\varphi^{-1}(y)}\lambda(x)$,
and hence
$\displaystyle\lambda(y)=\sup\{r\in\mathbb
R\;|\;y\in W_r\}\le\sup_{
x\in\varphi^{-1}(y)}\lambda(x)$.

``$\Longleftarrow$'': For any non-zero
element $y$ of $W$, if $y\in W_r$, then
$\lambda(y)\ge r$, so
$\displaystyle\sup_{x\in\varphi^{-1}(y)}\lambda(x)\ge
r$. Therefore, for any $s<r$, there exists
$x\in\varphi^{-1}(y)$ such that
$\lambda(x)\ge s$. Since the filtration
$\mathcal F$ is left continuous, we have
$x\in V_s$. This implies $\displaystyle
y\in\bigcap_{s<r}\varphi(V_s)$.

On the other hand, if $y$ is a non-zero
element in $\varphi(V_s)$, then there exists
$x\in V_s$ such that $y=\varphi(x)$. So
$\lambda(y)\ge\lambda(x)\ge s$. This implies
that $y\in W_s$ since the filtration
$\mathcal G$ is left continuous. Therefore,
$\displaystyle\bigcap_{s<r}\varphi(V_s)\subset\bigcap_{s<r}
W_s=W_r$.
\end{proof}

In the following, we use Borel measures on
$\mathbb R$ to study $\mathbb R $-filtrations
of vector spaces. For any finite dimensional
vector space over $K$, equipped with a
separated and exhaustive filtration, we shall
associate a Borel probability measure on
$\mathbb R$ to each base of the vector space,
which is a linear combination of Dirac
measures. Furthermore, there exists a
``maximal base'' whose associated measure
captures full ``numerical'' information of
the filtration. This technic will play an
import role in the sequel.

If $\nu_1$ and $\nu_2$ are two bounded Borel
measures on $\mathbb R$, we say that $\nu_1$
is {\it on the right} of $\nu_2$ and we write
$\nu_1\succ\nu_2$ or $\nu_2\prec\nu_1$ if for
any increasing and bounded function $f$, we
have
\[\int_{\mathbb R}f\mathrm{d}\nu_1\ge
\int_{\mathbb R}f\mathrm{d}\nu_2,\] which is
also equivalent to say that for any
$r\in\mathbb R$, $\displaystyle\int_{\mathbb
R}\indic_{[r,+\infty[}\mathrm{d}\nu_1
\ge\int_{\mathbb
R}\indic_{[r,+\infty[}\mathrm{d}\nu_2$. We
say that $\nu_1$ is {\it strictly on the
right} of $\nu_2$ if $\nu_1\succ\nu_2$ but
$\nu_2\neq\nu_1$.

\begin{definition}\label{Def:mesure associee a une base}
Let $V$ be a vector space of rank
$0<n<+\infty$ over $K$, equipped with a
separated and exhaustive filtration $\mathcal
F$. If $\mathbf{e}=(e_i)_{1\le i\le n}$ is a
base of $V$, we define a Borel probability
measure on $\mathbb R$
\[\nu_{\mathcal F,\mathbf{e}}:=
\frac{1}{n}\sum_{i=1}^n\delta_{\lambda(e_i)},\]
called the {\it probability associated to }
$\mathcal F$ {\it relatively to}
$\mathbf{e}$. If there is no ambiguity on the
filtration, we write also $\nu_{\mathbf{e}}$
instead of $\nu_{\mathcal F,\mathbf{e}}$.
Notice that Proposition \ref{Pro:invariance
de lambda par la filtration continue
associee} implies that $\nu_{\mathcal
F,\mathbf{e}}=\nu_{\mathcal F^l,\mathbf{e}}$.
We say that a base $\mathbf{e}$ of $V$ is
{\it maximal} if for any base $\mathbf{e'}$
of $V$, we have
$\nu_{\mathbf{e}}\succ\nu_{\mathbf{e'}}$.
Clearly a base $\mathbf{e}$ is maximal for
the filtration $\mathcal F$ if and only if it
is maximal for the filtration $\mathcal F^l$.
\end{definition}

\begin{proposition}\label{Pro:critere de base maximale}
Suppose that the filtration $\mathcal F$ of
$V$ is left continuous. Then a base
$\mathbf{e}=(e_i)_{1\le i\le n}$ of $V$ is
{\it maximal} if and only if $\card(\mathbf{e}\cap V_r)=\rang
V_r$ for any real
number $r$.
\end{proposition}
\begin{proof} ``$\Longrightarrow$'':
Since $\mathbf{e}$ is a base of $V$, we have
$\card(\mathbf{e}\cap V_r)\le\rang V_r$ for
any $r\in\mathbb R$. Suppose that there
exists a real number $r$ such that
$\card(\mathbf{e}\cap V_r)<\rang V_r$. Let
$V_r'$ be the sub-vector space of $V_r$
generated by $\mathbf{e}\cap V_r$. We have
clearly $\rang V_r'<\rang V_r$. Hence there
exists $e'\in V_r\setminus V_r'$. Since
$\mathbf{e}$ is a base of $V$, there exists
$(a_i)_{1\le i\le n}\in K^n$ such that
$e'=a_1e_1+\cdots+a_ne_n$. As $e'\not\in
V_r'$, there exists an integer $1\le i\le n$
such that $e_i\not\in V_r$ and that $a_i\neq
0$. Therefore
\[\mathbf{e}'=(e_1,\cdots,e_{i-1},e',e_{i+1},\cdots,e_n)\]
is a base of $V$. Furthermore, as $e_i\not\in
V_r$, we have $\lambda(e_i)<r$ since the
filtration is left continuous. On the other
hand, since $e'\in V_r$, we have
$\lambda(e')\ge r$. Let $g$ be an increasing
function such that
$g(\lambda(e_i))<g(\lambda(e'))$. Then we
have
\[\int_{\mathbb R} g\mathrm{d}\nu_{\mathbf{e}'}-
\int_{\mathbb R} g\mathrm{d}\nu_{\mathbf{e}}
=\frac{1}{n}\Big(g(\lambda(e'))-g(\lambda(e_i))\Big)>0,\]
which is absurd since $\mathbf{e}$ is
maximal.

``$\Longleftarrow$'': For any real number $r$
and any base $\mathbf{e}'$ of $V$, we have
$\displaystyle\card(\mathbf{e}'\cap
V_r)=n\int_{\mathbb R}\indic_{[r,+\infty[}
\mathrm{d}\nu_{\mathbf{e}'}$. Hence for any
real number $r$, we have
$\displaystyle\int_{\mathbb
R}\indic_{[r,+\infty[}
\mathrm{d}\nu_{\mathbf{e'}} \le\int_{\mathbb
R}\indic_{[r,+\infty[}\mathrm{d}\nu_{\mathbf{e}}$.
Therefore,
$\nu_{\mathbf{e}'}\prec\nu_{\mathbf{e}}$.
\end{proof}

\begin{proposition}\label{Pro:construction de base maximale a partir
une base quelconique} For any base
$\mathbf{e}=(e_1,\cdots,e_n)^T$ of $V$, there
exists an upper triangulated matrix $A\in
M_{n\times n}(K)$ with
$\diag(A)=(1,\cdots,1)$ such that
$A\mathbf{e}$ is a maximal base of $V$.
\end{proposition}
\begin{proof} We may suppose that the
filtration $\mathcal F$ is left continuous:
it suffices to replace it by $\mathcal F^l$.
We shall prove the proposition by induction
on the rank $n$ of $V$. If $n=1$, then
\[V_r=\begin{cases}
V,&r\le\lambda(e_1),\\
0,&r>\lambda(e_1).
\end{cases}\]
Hence $\card(V_r\cap\{e_1\})=\rang V_r$. In
other words, $\mathbf{e}$ is already a
maximal base.

Suppose that $n>1$. Let $W$ be the quotient
space $V/Ke_n$, equipped with the strong
direct image filtration. Then
$\widetilde{\mathbf{e}}=([e_1],\cdots,[e_{n-1}])^T$
is a base of $W$, where $[e_i]$ is the
canonical image of $e_i$ in $W$ ($1\le i\le
n-1$). By the hypothesis of induction, there
exists $\widetilde A\in
M_{(n-1)\times(n-1)}(K)$ with
$\diag(\widetilde A )=(1,\cdots,1)$ such that
$\overrightarrow{\alpha}=(\alpha_1,\cdots,\alpha_{n-1}):=\widetilde
A\widetilde{\mathbf{e}}$ is a maximal base.

Let $\pi:V\rightarrow W$ be the canonical
projection. For any $1\le i\le n-1$, choose
$e_i'\in \pi^{-1}(\alpha_i)$ such that
$\displaystyle\lambda(e_i')=\max_{x\in\pi^{-1}(\alpha_i)}\lambda(x)$.
This is always possible since the function
$\lambda$ takes only a finite number of
values. Let
$\mathbf{e}'=(e_1',\cdots,e_{n-1}',e_n)^T$.
Notice that $\mathbf{e}'$ can be written as
$A\mathbf{e}$, where
\[A=\left(\begin{array}{cc}\widetilde A&*\\
0&1\end{array}\right)\] is an upper
triangulated matrix with diagonal
$\diag(A)=(1,\cdots,1)$. Since
$\overrightarrow{\alpha}$ is a maximal base,
$\card(W_r\cap\overrightarrow{\alpha})=\rang
W_r$ for any $r\in\mathbb R$. In addition,
$e_i'\in V_r$ implies that
$\alpha_i=\pi(e_i')\in W_r$. Hence
\[\card(\mathbf{e}'\cap V_r)\ge
\begin{cases}
\rang W_r\ge\rang\pi(V_r)=\rang V_r,&e_n\not\in V_r,\\
\rang W_r+1\ge\rang\pi(V_r)+1=\rang
V_r,&e_n\in V_r.
\end{cases}\]
So we always have
$\card(V_r\cap\mathbf{e}')=\rang V_r$, and
hence $\mathbf{e}'$ is a maximal base.
\end{proof}

Proposition \ref{Pro:construction de base
maximale a partir une base quelconique} can
also be proved in the following way: the set
$X$ of complete flags of $V$ is equipped with
a transitive action of $\GL_n(K)$ and
identifies with the homogeneous space
$\GL_n(K)/B$, where $B$ is the subgroup of
upper triangulated matrices. The proposition
is then a consequence of Bruhat's
decomposition
for invertible
matrices.

\begin{remark}
Proposition \ref{Pro:construction de base
maximale a partir une base quelconique}
implies actually that there always exists a
maximal base of $V$.
\end{remark}

\begin{definition}\label{Def:mesure associee a une filtration}
If $\mathbf{e}$ is a maximal base of $V$, the
measure $\nu_{\mathcal F, \mathbf{e}}$ is
called the (probability) {\it measure}
associated to $\mathcal F$.

It is clear that the measure associated to
$\mathcal F$ doesn't depend on the choice of
the maximal base $\mathbf{e}$, we shall denote
it by $\nu_{\mathcal F,V}$ (or simply $\nu_V$
if there is no ambiguity on $\mathcal F$). If
$V$ is the zero space, then $\nu_V$ is by
convention the zero measure.
\end{definition}

Let $V$ be a finite dimensional vector space
over $K$. A left continuous $\mathbb
R$-filtration $\mathcal F$ of $V$ is
equivalent to the data of a flag
$V^{(0)}\subsetneq V^{(1)}\subsetneq\cdots
\subsetneq V^{(n)}$ together with a strictly
decreasing real number sequence $(a_i)_{1\le
i\le n}$, which describes the jumping points.
We have
\[ \mathcal F_rV=\begin{cases}V^{(0)}
&\text{if }r\in]a_1,+\infty[,\\
V^{(i)}&\text{if }r\in ]a_{i+1},a_{i}],\quad 1\le i<n,\\
V^{(n)}&\text{if }r\in]-\infty,a_n].
\end{cases}\]
The filtration $\mathcal F$ is separated
(resp. exhaustive) if and only if
$V^{(0)}=\{0\}$ (resp. $ V^{(n)}=V$). When
$\mathcal F$ is separated and exhaustive, the
measure associated to $\mathcal F$ equals to
\[\sum_{i=1}^n\frac{\rang V^{(i)}-\rang V^{(i-1)}}{\rang V}\delta_{a_i}.\]
Therefore, if $V$ is non-zero, then for any
$x\in\mathbb R$, we have the equality
\[1-\frac{\rang V_x}{\rang V}=
\nu_V\big(]-\infty,x[\big).\] The probability
distribution function of $\nu_{V}$ is
therefore
\[F(x)=1-\lim_{y\rightarrow x+}
\frac{\rang V_y}{\rang V}.\]

\begin{proposition}\label{Pro:relation de mesure associe pour suite exacte}
Let $\xymatrix{0\ar[r]&V'\ar[r]^\varphi&V
\ar[r]^\psi&V''\ar[r]&0}$ be a short exact
sequence of finite dimensional vector spaces
over $K$ equipped with left continuous
$\mathbb R$-filtrations. Suppose that the
following conditions are verified:
\begin{enumerate}[i)]
\item the space $V$ is non-zero
and the filtration $\mathcal F$ of $V$ is
separated and exhaustive;
\item the filtration of $V'$ is the inverse
image $\varphi^*\mathcal F$;
\item the filtration of $V''$ is the strong
direct image $\psi_*\mathcal F$.
\end{enumerate}
Then $\displaystyle\nu_{V}=\frac{\rang
V'}{\rang V} \nu_{V'}+\frac{\rang V''}{\rang
V}\nu_{V''}$.
\end{proposition}
\begin{proof} If $W$ is a finite
dimensional vector space over $K$ equipped
with an $\mathbb R$-filtration, the
filtration of $W$ is separated and exhaustive
if and only if the function
$\lambda:W\setminus\{0\}\rightarrow\mathbb
R\cup\{\pm\infty\}$ takes values in a bounded
interval in $\mathbb R$ (see Proposition
\ref{Pro:proprietes de la fonction lambda} 4)
and Proposition \ref{Pro:crietere numerique
de separe et exhaustive}). Therefore, after
Proposition \ref{Pro:exhaustivite de
filtration induite quotiente}, if $\mathcal
F$ is separated and exhaustive, then also are
$\varphi^*\mathcal F$ and $\psi_*\mathcal F$.
So the measures $\nu_{V'}$ and $\nu_{V''}$
are well defined.

Let $\mathbf{e}'=(e_i')_{1\le i\le n}$ (resp.
$\mathbf{e}''=(e_j'')_{1\le j\le m}$) be a
maximal base of $V'$ (resp. $V''$). Let
\[\mathbf{e}=(\varphi(e_1'),\cdots,
\varphi(e_n'),e_{n+1},\cdots,e_{n+m})\] be a
base of $V$ such that, for any integer $1\le
j\le m$, $\psi(e_{n+j})=e_j''$ and
$\lambda(e_{n+j})=\lambda(e_j'')$ (this is
always possible after Proposition
\ref{Pro:proprietes de la fonction lambda} 4)
and Proposition \ref{Pro:critere numerique de
filtration induit quotient} 2)). By
definition we know that
\[\nu_{\mathbf{e}}=\frac{\rang V'}{\rang V}\nu_{\mathbf{e}'}+\frac{\rang V''}{\rang V}
\nu_{\mathbf{e}''}.\] It suffices then to
verify that $\mathbf{e}$ is a maximal base.

Let $r$ be a real number. First we have
\begin{equation}
\label{Equ:suite exacte courte de lambda 1}
\card(\{\varphi(e_1'),\cdots\varphi(e_n')\}
\cap V_r)=\card(\mathbf{e}'\cap V_r')=\rang
V_r'.\end{equation} On the other hand, since
$\lambda(e_j'')=\lambda(e_{n+j})$,
$e_{j}''\in V_r''$ if and only if $e_{n+j}\in V_r$. Therefore
\begin{equation}\label{Equ:suite exacte courte de lambda 2}
\card(\{e_{n+1},\cdots,e_{n+m}\}\cap V
_r)=\card(\mathbf{e}''\cap
V_r'')=\rang V_r''.\end{equation}
The sum of the inequalities (\ref{Equ:suite
exacte courte de lambda 1}) and
(\ref{Equ:suite exacte courte de lambda 2})
gives $\card(\mathbf{e}\cap
V_r)=\rang(V_r')+ \rang(V_{r}'')=\rang
(V_r)$, so $\mathbf{e}$ is a maximal base.
\end{proof}

\section{Almost super-additive sequence}
\label{Sec:Almost super-additive sequence}
\hskip\parindent In this section we discuss
a generalization of Fekete's lemma (see
\cite{Fekete23} page 233 for a particular
case) asserting that, for any sub-additive
sequence $(a_n)_{n\ge 1}$ of real numbers
(that's to say, $a_{n+m}\le a_n+a_m$ for any
$(m,n)\in\mathbb Z_{>0}^2$), the limit
$\displaystyle\lim_{n\rightarrow+\infty}a_n/n$
exists in $\mathbb R\cup\{-\infty\}$. We
shall show that the convergence of the
sequence $(a_n/n)_{n\ge 1}$ is still valid if
the sequence $(a_n)_{n\ge 1}$ is sub-additive
up to a small error term. These technical
results are crucial to prove the convergence
theorems stated in the section of
introduction.

\begin{proposition}\label{Pro:convergence un peu faible que le lemme precedent}
Let $(a_n)_{n\ge 1}$ be a sequence in
$\mathbb R_{\ge 0}$ and $f:\mathbb
Z_{>0}\rightarrow\mathbb R$ be a function
such that
$\displaystyle\lim_{n\rightarrow\infty}{f(n)}/{n}=0$.
If there exists an integer $n_0>0$ such that,
for any integer $l\ge 2$ and any $(n_i)_{1\le
i\le l}\in\mathbb Z_{\ge n_0}^l$, we have
$a_{n_1+\cdots+n_l}\le
a_{n_1}+\cdots+a_{n_l}+f(n_1)+\cdots+f(n_l)$,
then the sequence $(a_n/n)_{n\ge 1}$ has a
limit in $\mathbb R_{\ge 0}$.
\end{proposition}
\begin{proof} If $n$, $p$ and $n\le l<2n$
are three integers $\ge n_0$, we have
\[\begin{split}\frac{a_{pn+l}}{pn+l}&\le
\frac{pa_n+a_l}{pn+l}+\frac{pf(n)+f(l)}{pn+l}\le\frac{a_n}{n}+\frac{a_l}{pn}+\frac{pf(n)+f(l)}{pn+l}\\
&\le\frac{a_n}{n}+\frac{a_l}{pn}+
\frac{|f(n)|}{n}+\frac{|f(l)|}{pn}.\end{split}\]
Since
$\displaystyle\lim_{p\rightarrow\infty}\frac{\underset{n\le
i< 2n}{\max}a_i}{pn}+\frac{\underset{n\le i<
2n}{\max}|f(i)|}{pn}=0$, we obtain, for any
integer $n>0$, that
\begin{equation}\label{Equ:limsup am sur mmajoree}
\limsup_{m\rightarrow\infty}\frac{a_m}{m}\le
\frac{a_n}{n}+\frac{|f(n)|}{n},\end{equation}
hence
\[\displaystyle\limsup_{m\rightarrow\infty}\frac{a_m}{m}
\le\liminf_{n\rightarrow\infty}\left(\frac{a_n}{n}+\frac{|f(n)|}{n}\right)\le\liminf_{n\rightarrow\infty}\frac{a_n}{n}+\limsup_{n\rightarrow\infty}\frac{|f(n)|}{n}=\liminf_{n\rightarrow\infty}\frac{a_n}{n}.\]
Therefore, the sequence $(a_n/n)_{n\ge 1}$
has a limit, which is clearly $\ge 0$, and is
finite after \eqref{Equ:limsup am sur
mmajoree}.
\end{proof}

\begin{corollary}\label{Cor:limite de an sur n forme forte}
Let $(a_n)_{n\ge 1}$ be a sequence of real
numbers and $f:\mathbb
Z_{>0}\rightarrow\mathbb R$ be a function
such that
$\displaystyle\lim_{n\rightarrow\infty}{f(n)}/{n}=0$.
If the following two conditions are verified:
\begin{enumerate}[1)]
\item there exists an integer $n_0>0$ such
that, for any integer $l\ge 2$ and any
$(n_i)_{1\le i\le l}\in\mathbb Z_{\ge
n_0}^l$, we have $a_{n_1+\cdots+n_l}\ge
a_{n_1}+\cdots+a_{n_l}-f(n_1)-\cdots-f(n_l)$,
\item there exists a constant
$\alpha>0$ such that $a_n\le \alpha n$ for
any integer $n\ge 1$,
\end{enumerate}
then the sequence $(a_n/n)_{n\ge 1}$ has a
limit in $\mathbb R$.
\end{corollary}
\begin{proof} Consider the sequence $(b_n=\alpha
n-a_n)_{n\ge 1}$ of positive real numbers. If
$n_1,\cdots,n_l$ are integers $\ge n_0$ and
$n=n_1+\cdots+n_l$, then
\[\begin{split}&\quad\; b_n=\alpha n-a_n=
\alpha\sum_{i=1}^ln_i-a_n \le
\alpha\sum_{i=1}^ln_i-\sum_{i=1}^l
\Big(a_{n_i}-f(n_i)\Big)\\
&=\sum_{i=1}^l \Big(\alpha
n_i-a_{n_i}+f(n_i)\Big)
=b_{n_1}+\cdots+b_{n_l}+f(n_1)+\cdots+f(n_l).
\end{split}\]
After Proposition \ref{Pro:convergence un peu
faible que le lemme precedent}, the sequence
$(b_n/n)_{n\ge 1}$ has a limit in $\mathbb
R$. As ${b_n}/{n}=\alpha-{a_n}/{n}$, the
sequence $(a_n/n)_{n\ge 1}$ also has a limit
in $\mathbb R$.
\end{proof}

\begin{corollary}\label{Cor:suite croissance lineaire la version renforcee}
Let $(a_n)_{n\ge 1}$ be a sequence of real
numbers and $c_1$, $c_2$ be two positive
constants such that
\begin{enumerate}[1)]
\item $a_{m+n}\ge a_m+a_n-c_1$ for any pair $(m,n)$ of sufficiently large
integers,
\item $a_n\le c_2n$ for any integer $n\ge 1$,
\end{enumerate}
then the sequence $(a_n/n)_{n\ge 1}$ has a
limit in $\mathbb R$.
\end{corollary}
\begin{proof}
Let $f$ be the constant function taking value
$c_1$. By induction we obtain the following inequality for any
finite sequence  $(n_i)_{1\le i\le l}$ of
sufficiently large integers:
\[a_{n_1+\cdots+n_l}\ge a_{n_1}+\cdots+a_{n_l}-f(n_1)-\cdots-f(n_l),\]
After Corollary \ref{Cor:limite de an sur n
forme forte}, the sequence $(a_n/n)_{n\ge 1
}$ converges in $\mathbb R$.
\end{proof}

\section{Quasi-filtered graded algebras}
\label{Sec:Quasi-filtered graded algebras}
\hskip\parindent In this section we introduce
the notion of {\it quasi-filtered graded
algebras}. Such algebras are fundamental
objects in this article. We are particularly
interested in the convergence of measures
associated to a quasi-filtered graded algebra
(Sections \ref{Sec:Convergence for symmetric
algebras} to \ref{Sec:Convergence of polygons
of a quasi-filtered graded algebra}). Later
we shall show that the graded algebras that
we have mentioned in the section of
introduction, equipped with Harder-Narasimhan
filtrations, are quasi-filtered graded
algebras. The results presented in this
section is therefore a formalism which is
useful to study the Harder-Narasimhan
filtrations of graded algebras.

Let $K$ be a filed. Recall that a $\mathbb
Z_{\ge 0}$-{\it graded $K$-algebra} is a
direct sum $B=\bigoplus_{n\ge 0}B_n$ of
vector spaces over $K$ indexed by $\mathbb
Z_{\ge 0}$ equipped with a commutative
unitary $K$-algebra structure such that
$B_nB_m\subset B_{n+m}$ for any
$(m,n)\in\mathbb Z_{\ge 0}^2$. We call {\it
homogeneous element of degree $n$} any
element in $B_n$. Clearly the unit element of
$B$ is homogeneous of degree $0$. In the
following, we shall use the expression
``graded $K$-algebra'' to denote a $\mathbb
Z_{\ge 0}$-graded $K$-algebra. If $B$ is a
graded $K$-algebra, we call {\it graded}
$B$-{\it module} any $B$-module $M$ equipped
with a decomposition
$M=\bigoplus_{n\in\mathbb Z}M_n$ into direct
sum of vector subspaces over $K$ such that
$B_nM_m\subset M_{n+m}$ for any
$(n,m)\in\mathbb Z_{\ge 0 }\times\mathbb Z$.
The elements in $M_m$ are called {\it
homogeneous element of degree $m$} of $M$. If
$x$ is a non-zero homogeneous element of $M$,
we use $\mathrm{d}^{\circ}_M(x)$ or
$\mathrm{d}^\circ(x)$ to denote the
homogeneous degree of $x$. For reference on
graded algebras and graded modules, one can
consult \cite{Bourbaki85}.

\begin{definition}
Let $B=\bigoplus_{n\ge 0}B_n$ be a graded
$K$-algebra and $f:\mathbb Z_{\ge 0
}\rightarrow\mathbb R_{\ge 0}$ be a function.
We say that the $K$-algebra $B$ is {\it
$f$-quasi-filtered} if each vector space
$B_n$ is equipped with an $\mathbb
R$-filtration $(B_{n,s})_{s\in\mathbb R}$
satisfying the following
condition:\begin{quote}{\it there exists an
integer $n_0\ge 0$ such that, for any integer
$r>0$, any $(n_i)_{1\le i\le r}\in\mathbb
Z_{\ge n_0 }^r$ and any $(s_i)_{1\le i\le
r}\in\mathbb R^r$, we have
\[\prod_{i=1}^rB_{n_i,s_i}\subset
B_{N,S}\qquad \text{where}\qquad
N=\sum_{i=1}^rn_i,\quad
S=\sum_{i=1}^r\big(s_i-f(n_i)\big).\]}\end{quote}
If $B$ is an $f$-quasi-filtered graded
$K$-algebra. We say that a graded $B$-module
$M=\bigoplus_{n\in\mathbb Z}M_n$ is {\it
$f$-quasi-filtered} if for any integer $n$,
$M_n$ is equipped with an $\mathbb
R$-filtration $(M_{n,s})_{s\in\mathbb R}$ satisfying the following
condition: \begin{quote}{\it there exists an
integer $n_0\ge 0$ such that, for any integer
$r>0$, any $(n_i)_{1\le i\le r+1}\in \mathbb
Z_{\ge n_0 }^{r+1}$ and any $(s_i)_{1\le i\le
r+1}\in\mathbb R^{r+1}$, we have
\[\Big(\prod_{i=1}^rB_{n_i,s_i}\Big)
M_{n_{r+1},s_{r+1}}\subset
M_{N,S}\qquad\text{where}\qquad
N=\sum_{i=1}^{r+1}n_i,\quad
S=\sum_{i=1}^{r+1}\big(s_i-f(n_i)).
\]
}\end{quote} In particular, if $f\equiv 0$,
we say that $B$ is a {\it filtered graded}
$K$-algebra, and $M$ is a {\it filtered
graded} $B$-module.
\end{definition}

We now give some numerical criteria for a
graded algebra (or graded module) equipped
with $\mathbb R$-filtrations to be
quasi-filtered.

\begin{proposition}\label{Pro:critere numerique de structure d'algebre filre}
Let $B$ be a graded $K$-algebra and
$f:\mathbb Z_{\ge 0}\rightarrow\mathbb R_{\ge
0}$ be a function. Suppose that for each
$n\in\mathbb Z_{\ge 0}$, $B_n$ is equipped
with an exhaustive and left continuous
$\mathbb R$-filtration. Then the following
conditions are equivalent:
\begin{enumerate}[1)]
\item the graded algebra $B$ is
$f$-quasi-filtered,
\item there exists an integer
$n_0>0$ such that, for any integer $r\ge 2$
and any non-zero homogeneous elements
$a_1,\cdots,a_r$ of degree $\ge n_0$ of $B$,
if we write
$a=\displaystyle\prod_{i=1}^ra_i$, then
\begin{equation}\lambda(a)\ge
\sum_{i=1}^r\Big(\lambda(a_i)-f({\mathrm
d}^{\circ}(a_i))\Big).\label{Equ:critere
numerique de quasifiltree}
\end{equation}
\end{enumerate}
\end{proposition}
\begin{proof}
The filtrations being exhaustive, the sum on
the right side of (\ref{Equ:critere numerique
de quasifiltree}) is well defined and takes
value in $\mathbb R\cup\{+\infty\}$.

``1)$\Longrightarrow$2)'': Since the
filtrations are left continuous, we have
$a_i\in\mathcal
F_{\lambda(a_i)}B_{\mathrm{d}^\circ(a_i)}$.
Let
\[d=\sum_{i=1}^r\mathrm{d}^\circ(a_i)\qquad\text{and}\qquad
\eta=\sum_{i=1}^r\Big(\lambda(a_i)-f(\mathrm{d}^\circ(a_i))\Big).\]
Since $B$ is $f$-quasi-filtered, we obtain
$a\in\mathcal F_\eta B_d$, so
$\lambda(a)\ge\eta$.

``2)$\Longrightarrow$1)'': Suppose that
$a_1,\cdots,a_r$ are homogeneous elements of
degrees $\ge n_0$ of $B$. For any integer
$1\le i\le r$ let
$d_i=\mathrm{d}^\circ(a_i)$. Let
$a=\displaystyle\prod_{i=1}^ra_i$. If for any
integer $1\le i\le r$, we have
$a_i\in\mathcal F_{t_i}B_{d_i}$, then we have
$\lambda(a_i)\ge t_i$. Therefore,
$\displaystyle\lambda(a)\ge\sum_{i=1}^r\Big(t_i-f(d_i)\Big)$.
Hence $a\in\mathcal
F_{t_1+\cdots+t_r-f(d_1)-\cdots-f(d_r)}B_{d_1+\cdots+d_r}$.
\end{proof}

Using the numerical criterion established
above, we obtain the following corollary.

\begin{corollary}\label{Cor:sous algebre algebre quotient}
Let $f:\mathbb Z_{\ge 0}\rightarrow\mathbb
R_{\ge 0}$ be a function and $B$ be an
$f$-quasi-filtered graded $K$-algebra.
Suppose that for any integer $n\ge 0$, the
filtration of $B_n$ is exhaustive and left
continuous.
\begin{enumerate}[1)]
\item Let $A$ be a sub-$K$-algebra of $B$ generated by homogeneous elements,
equipped with induced graduation. If for each
$n\in\mathbb Z_{\ge 0}$, the vector space
$A_n$ is equipped with the inverse image
filtration, then $A$ is an $f$-quasi-filtered
graded $K$-algebra.
\item Let $I$ be a homogeneous ideal of $B$
and let $C=B/I$, equipped with the quotient
graduation. If for each $n\in\mathbb Z_{\ge
0}$, the vector space $C_n$ is equipped with
the strong direct image filtration, then $C$
is an $f$-quasi-filtered graded $K$-algebra.
\end{enumerate}
\end{corollary}
\begin{proof}
1) After Proposition \ref{Pro:exhaustivite de
filtration induite quotiente}, the
filtrations of $A_n$ are exhaustive.
Furthermore, they are left continuous. If $a$
is a homogeneous element of $A$, then
$\mathrm{d}_A^\circ(a)=\mathrm{d}_B^\circ(a)$.
On the other hand, since the filtrations of
$A_n$ ($n\ge 0$) are inverse images
filtrations, we obtain
$\lambda_A(a)=\lambda_B(a)$. So for any
integer $r\ge 2$ and any family $(a_i)_{1\le
i\le r}$ of homogeneous elements of degree
$\ge n_0$ in $A$ with
$a=\displaystyle\prod_{i=1}^ra_i$, we have
\[\lambda_A(a)=\lambda_B(a)\ge
\sum_{i=1}^r\Big(\lambda_B(a_i)-f(\mathrm{d}_B^\circ
(a_i))\Big)=\sum_{i=1}^r\Big(\lambda_A(a_i)-
f(\mathrm{d}_A^\circ(a_i))\Big).\] So the
graded algebra $A$ is $f$-quasi-filtered.

2) After Proposition \ref{Pro:exhaustivite de
filtration induite quotiente}, the
filtrations of homogeneous components of $C$
are exhaustive. Let $\pi:B\rightarrow C$ be
the canonical homomorphism. Suppose that
$(a_i)_{1\le i\le r}$ is a family of
homogeneous elements of degree $\ge n_0$ in
$C$. For any $1\le i\le r$, let $d_i=
\mathrm{d}^\circ(a_i)$ and
$t_i=\lambda_C(a_i)$. After Proposition
\ref{Pro:critere numerique de filtration
induit quotient} 2), for any $1\le i\le r$,
there exists a sequence
$(\alpha_j^{(i)})_{j\ge 1}$ in $B_{d_i}$ such
that $\pi(\alpha_j^{(i)})=a_i$ for any $j\ge
1$ and that the sequence
$(\lambda_B(\alpha_j^{(i)}))_{j\ge 1}$ is
increasing and converge to $t_i$. Let
$a=\displaystyle\prod_{i=1}^ra_i$ and for any
$j\ge 1$, let
$\alpha_j=\displaystyle\prod_{i=1}^r\alpha_j^{(i)}$.
Clearly we have $a=\pi(\alpha_j)$ for any
$j\ge 1$. Therefore,
$\lambda_C(a)\ge\lambda_B(\alpha_j)$. On the
other hand,
$\displaystyle\lambda_B(\alpha_j)\ge\sum_{i=1}^r\Big(
\lambda_B(\alpha_j^{(i)})-f(d_i)\Big)$. Hence
$\displaystyle\lambda_C(a)\ge\sum_{i=1}^r\Big(\lambda_B(\alpha_j^{(i)})-f(d_i)\Big)$.
By passing to the limit when $j\rightarrow
+\infty$ we obtain
$\displaystyle\lambda_C(a)\ge\sum_{i=1}^r(t_i-f(d_i))$.
\end{proof}

The following assertions give numerical
criteria for quasi-filtered graded modules,
the proofs are similar.

\begin{proposition}
Let $f:\mathbb Z_{>0}\rightarrow\mathbb
R_{\ge 0}$ be a function, $B$ be an
$f$-quasi-filtered graded $K$-algebra and $M$
be a graded $B$-module. Suppose that for any
integer $n$, $M_n$ is equipped with an
exhaustive and left continuous $\mathbb
R$-filtration. Suppose in addition that for
any integer $n\ge 0$, the filtration of $B_n$
is exhaustive and left continuous. Then the
following conditions are equivalent:
\begin{enumerate}[1)]
\item the graded $B$-module $M$ is
$f$-quasi-filtered;
\item there exists an integer $n_0\ge 0$ such that,
for any integer $r\ge 1$, any family
$(a_i)_{1\le i\le r}$ of non-zero homogeneous
elements of degree $\ge n_0$ of $B$ and any
non-zero homogeneous element $x$ of degree
$\ge n_0$ of $M$, if we write $y=(a_1\cdots
a_r)x$, then
\[\lambda(y)\ge\sum_{i=1}^r\Big(\lambda(a_i)-
f(\mathrm{d}^\circ(a_i))\Big)+\lambda(x)-f(\mathrm{d}^\circ(x)).\]
\end{enumerate}
\end{proposition}

\begin{corollary}\label{Cor:sousmodule module quotient}
Let $f:\mathbb Z_{>0}\rightarrow\mathbb
R_{\ge 0}$ be a function, $B$ be an
$f$-quasi-filtered graded $K$-algebra and $M$
be an $f$-quasi-filtered graded $B$-module.
Suppose that for any integer $n\ge 0$, the
filtrations of $B_n$ and of $M_n$ are
exhaustive and left continuous.
\begin{enumerate}[1)]
\item Let $M'$ be a graded sub-$B$-module.
If each $M_n'$ is equipped with the inverse
image filtration, then $M'$ is an
$f$-quasi-filtered graded $B$-module.
\item Let $M'$ be a homogeneous
sub-$B$-module of $M$ and let $M''=M/M'$. If
each $M_n''$ is equipped with the strong
direct image filtration, then $M''$ is an
$f$-quasi-filtered graded $B$-module.
\end{enumerate}
\end{corollary}

\begin{corollary}\label{Cor:fonctorialite de filtrations qotiet etc}
Let $f:\mathbb Z_{>0}\rightarrow\mathbb
R_{\ge 0}$ be a function, $B$ be an
$f$-quasi-filtered graded $K$-algebra, and
$M$ be an $f$-quasi-filtered graded
$B$-module. Suppose that for any positive
integer (resp. any integer) $n$, the
filtration of $B_n$ (resp. $M_n$) is
exhaustive and left continuous.
\begin{enumerate}[1)]
\item Let $A$ be a sub-$K$-algebra of $B$ generated by
homogeneous elements, equipped with the
induced graduation. If each vector space
$A_n$ is equipped with the inverse image
filtration, then $M$ is an $f$-quasi-filtered
graded $A$-module.
\item Let $I$ be a homogeneous ideal of $B$
contained in $\ann(M)$ and $C=B/I$ which is
equipped with the quotient graduation. If
each $C_n$ is equipped with the strong direct
image filtration, then $M$ is an
$f$-quasi-filtered graded $C$-module.
\end{enumerate}
\end{corollary}

\section{Convergence for symmetric algebras}
\label{Sec:Convergence for symmetric
algebras} \hskip\parindent We now consider
the symmetric algebra of a finite dimensional
non-zero vector space, which is equipped with
certain suitable filtrations. Each
homogeneous component of the symmetric
algebra contains a special base which
consists of monomials. By introducing a
combinatoric equality on monomials (Theorem
\ref{Thm:theorme cle sur equiprobability}),
we establish a convergence result (Corollary
\ref{Cor:suradditivite de l'integrale}) for
quasi-filtered symmetric algebras. We shall
show later in Section \ref{Sec:Convergence of
polygons of a quasi-filtered graded algebra}
that the general convergence can be deduced
from this result in the special case of
quasi-filtered symmetric algebras.

For any pair of integers $(n,d)$ such that
$n\ge 0$ and $d\ge 1$, let $\Delta_n^{(d)}$
be the subset of $\mathbb Z_{\ge 0}^d$
consisting of all decompositions of $n$ into
sum of $d$ positive integers. We introduce
the lexicographic order on $\Delta_n^{(d)}$:
$(a_1,\cdots,a_d)\ge (b_1,\cdots,b_d)$ if and
only if there exists an integer $1\le i\le d$
such that $a_j=b_j$ for any $1\le j\le i$ and
that $a_{i+1}>b_{i+1}$ if $i<d$. The set
$\Delta_n^{(d)}$ is totally ordered. On the
other hand, for any integer $r\ge 2$ and any
${\mathbf{n}}=(n_i)_{1\le i\le r}\in\mathbb
Z_{\ge 0}^r$, we have a mapping from
$\Delta_{n_1}^{(d)}\times\cdots\times\Delta_{n_r}^{(d)}$
to $\Delta_{n_1+\cdots+n_r}^{(d)}$ which
sends $(\alpha_i)_{1\le i\le r}$ to
$\alpha_1+\cdots+\alpha_r$ (the addition
being that in $\mathbb Z^d$). This mapping is
not injective in general but is always
surjective. Moreover, if $(\alpha_i)_{1\le
i\le r}$ and $(\beta_i)_{1\le i\le r}$ are
two elements of
$\Delta_{n_1}^{(d)}\times\cdots\Delta_{n_r}^{(d)}$
such that  $\alpha_i\ge\beta_i$ for any $1\le
i\le r$, then
$\alpha_1+\cdots+\alpha_r\ge\beta_1+\cdots+\beta_r$.

For any $n\in \mathbb Z_{\ge 0}$, we denote
by $\Gamma_n^{(d)}$ the subset of $\mathbb
Z_{\ge 0}^{d-1}$ consisting of elements
$(a_i)_{1\le i\le d-1}$ such that $0\le
a_1+\cdots+a_{d-1}\le n$. We have a natural
mapping
$p_n^{(d)}:\Delta_n^{(d)}\rightarrow\Gamma_n^{(d)}$
defined by the projection on the first $d-1$
factors. The mapping $p_n^{(d)}$ is in fact a
bijection and its inverse is the mapping
which sends $(a_i)_{1\le i\le d-1}$ to
$(a_1,\cdots,a_{d-1},n-a_1-\cdots-a_{d-1})$.
For any $\mathbf{n}= (n_i)_{1\le i\le r
}\in\mathbb Z_{\ge 0}^r$, we have the
following commutative diagram
\begin{equation}\label{Equ:commutatif Gamma et Delta}\xymatrix{\relax
\Delta_{n_1}^{(d)}\times
\cdots\times\Delta_{n_r}^{(d)}
\ar[d]_{p_{n_1}^{(d)}\times\cdots\times
p_{n_r}^{(d)}
}\ar[r]^-{+}&\Delta_{|\mathbf{n}|}^{(d)}\ar[d]^{p_{N}^{(d)}}\\
\Gamma_{n_1}^{(d)}\times\cdots\times
\Gamma_{n_r}^{(d)}\ar[r]_-{+}&
\Gamma_{|\mathbf{n}|}^{(d)}}\end{equation}
where $|\mathbf{n}|=n_1+\cdots+n_r$ and the
operators ``$+$'' are defined by the addition
structures in the monoids $\mathbb Z_{\ge
0}^d$ and $\mathbb Z_{\ge 0}^{d-1}$
respectively.

\begin{theorem}\label{Thm:theorme cle sur equiprobability}
Let $r\ge 2$ and $d\ge 1$ be two integers.
For any $\mathbf{n}=(n_i)_{1\le i\le r
}\in\mathbb Z_{\ge 0}^r$, there exists a
probability measure $\rho_{\mathbf{n}}$ on
$\Delta_{n_1}^{(d)}\times\cdots\times
\Delta_{n_r}^{(d)}$ such that the direct
image of $\rho_{\mathbf{n}}$ by each of the
$r$ projections on
$\Delta_{n_1}^{(d)},\cdots,\Delta_{n_r}^{(d)}$
is equidistributed, and also is its direct
image on $\Delta_{|\mathbf{n}|}^{(d)}$ by the
operator ``$+$''.
\end{theorem}
\begin{proof}
The theorem is trivial when $d=1$ because in
this case, for any $k\in\mathbb Z_{\ge 0}$,
$\Delta_k^{(1)}$ is the one point set
$\{k\}$. In the following, we suppose $d\ge
2$. By \eqref{Equ:commutatif Gamma et Delta},
it suffices to construct a probability
measure $\rho_{\mathbf{n}}$ on
$\Gamma_{n_1}^{(d)}\times\cdots\times
\Gamma_{n_r}^{(d)}$ such that the direct
image of $\rho_{\mathbf{n}}$ by each of the
$r$ projections on
$\Gamma_{n_1}^{(d)},\cdots,\Gamma_{n_r}^d$ is
an equidistributed measure, and also is the
direct image on $\Gamma_{|\mathbf{n}|}^{(d)}$
by the operator ``$+$''.

For any $\alpha=(a_i)_{1\le i\le
d-1}\in\mathbb Z_{\ge 0}^{d-1}$, we define
$|\alpha|=a_1+\cdots+a_{d-1}$. The set
$\Gamma_n^{(d)}$ can be written in the form
$\Gamma_n^{(d)}=\{\alpha\in\mathbb Z_{\ge
0}^{d-1}\;\Big|\; |\alpha|\le n\}$. If
$\alpha=(a_i)_{1\le i\le d-1}$ is an element
of $\mathbb Z_{\ge 0}^{d-1}$, we write
$\alpha!=a_1!\times\cdots\times a_{d-1}!$.

Consider the algebra of formal series in $rd$
variables $R=\mathbb Z\lbr
\mathbf{t},\mathbf{X}\rbr$, where
$\mathbf{t}=(t_1,\cdots,t_{r})$,
$\mathbf{X}=(X_{i,j})_{\begin{subarray}{c}1\le
i\le r,\\
1\le j\le {d-1}
\end{subarray}}$. If
$\alpha=(a_1,\cdots,a_{d-1})\in\mathbb Z_{\ge
0}^{d-1}$ and if $1\le i\le r$, we denote by
$X_i^{\alpha}$ the product
$X_{i,1}^{a_1}\times\cdots\times
X_{i,d-1}^{a_{d-1}}$. If
$\mathbf{n}=(n_i)_{1\le i\le r}$ is an
element in $\mathbb Z_{\ge 0}^r$, we denote
by $\mathbf{t}^{\mathbf n}$ the product
$t_1^{n_1}\times\cdots\times t_r^{n_r}$. Let
$H(\mathbf{t},\mathbf{X})$ be the formal
series
\[\sum_{\mathbf{n}=(n_i)\in\mathbb Z_{\ge 0}^r}\mathbf{t}^{\mathbf{n}}
\sum_{\begin{subarray}{c}(\alpha_i)\in(\mathbb Z_{\ge 0}^{d-1})^r\\
|\alpha_i|\le n_i
\end{subarray}}\frac{(\alpha_1+\cdots+\alpha_r)!}{\alpha_1!
\cdots\alpha_r!}\frac{(n_1+\cdots+n_r-|\alpha_1+\cdots+\alpha_r|)!}
{(n_1-|\alpha_1|)!\cdots(n_r-|\alpha_r|)!}\prod_{j=1}^rX_j^{\alpha_j},\]
the coefficients of which are positive
integers. If we perform the change of indices
$m_i=n_i-|\alpha_i|$ and permute the
summations by defining
$(\beta_1,\cdots,\beta_{d-1})=\alpha_1+\cdots
+\alpha_r$ and $m=m_1+\cdots+m_r$, we obtain
the following equality in $\mathbb Z\lbr
\mathbf{t},\mathbf{X}\rbr$:
\[\begin{split}
H(\mathbf{t},\mathbf{X})&=\sum_{(\alpha_i)\in(\mathbb
Z_{\ge
0}^{d-1})^r}\frac{(\alpha_1+\cdots+\alpha_r)!}{\alpha_1!\cdots\alpha_r!}
\prod_{j=1}^rt_j^{|\alpha_j|}X_j^{\alpha_j}
\sum_{\mathbf{m}=(m_i)\in\mathbb Z_{\ge 0}^r
}\frac{(m_1+\cdots+m_r)!}{m_1!\cdots m_r!
}\mathbf{t}^{\mathbf{m}}\\
&=\sum_{(\beta_i)\in\mathbb Z_{\ge 0}^{d-1}}
\prod_{i=1}^{d-1}(t_1X_{1,i}+\cdots+t_rX_{r,i})^{\beta_i}
\sum_{m\in\mathbb Z_{\ge 0}}(t_1+\cdots+t_r)^m\\
&=(1-(t_1+\cdots+t_r))^{-1}\prod_{i=1}^{d-1}(1-(t_1X_{1,i}+\cdots+t_rX_{r,i}))^{-1}.
\end{split}\]
This calculation also implies (cf.
\cite{Homander90} chap. II \S2.4) that the
Reinhardt's absolute convergence domain of
$H(\mathbf{t},\mathbf{X})$ in $\mathbb
C^{rd}$ is defined by the condition
\[\sum_{j=1}^r|t_j|<1\text{ and }\sum_{j=1}^r|t_j|
|X_{j,i}|<1.\] This observation enables us to
substitute certain variables $X_i$ by the
vector $\indic=(\underset{d-1\text{
copies}}{\underbrace{1,\cdots,1}})$ without
examining convergence problems. By the change
of variables $m_i=n_i-|\alpha_i|$ for $2\le
i\le r$,
we obtain {
\begin{equation*}%
\begin{split} &\quad\; H(\mathbf{t},\mathbf{X})|_{X_2=\cdots=X_r=\indic}\\
&=\sum_{n_1\ge 0}t_1^{n_1}\sum_{|\alpha_1|\le
n_1}X_1^{\alpha_1}
\sum_{\begin{subarray}{c}(\alpha_i)_{i=2}^r\in(\mathbb
Z_{\ge
0}^{d-1})^{r-1}\\(m_i)_{i=2}^r\in\mathbb
Z_{\ge 0}^{r-1}\end{subarray}}
\frac{(\alpha_1+\cdots+\alpha_r)!}{\alpha_1!\cdots\alpha_r!}
\frac{(n_1+m_2+\cdots+m_r-|\alpha_1|)!}
{(n_1-|\alpha_1|)!\,m_2!\cdots
m_r!}\prod_{j=2}^rt_j^{m_j+|\alpha_j|}
\\
&=\sum_{n_1\ge 0}t_1^{n_1}\sum_{|\alpha_1|\le
n_1}X_1^{\alpha_1}\sum_{(\alpha_i)_{i=2}^r\in(\mathbb
Z_{\ge
0}^{d-1})^{r-1}}\frac{(\alpha_1+\cdots+\alpha_r)!}{\alpha_1!\cdots\alpha_r!}\prod_{j=2}^rt_j^{|\alpha_j|}\\
&\qquad\quad\sum_{(m_i)_{i=2}^r\in\mathbb
Z_{\ge
0}^{r-1}}\frac{(n_1+m_2+\cdots+m_r-|\alpha_1|)!}{(n_1-|\alpha_1|)!m_2!\cdots
m_r!}\prod_{j=2}^rt_j^{m_j}.\end{split}\end{equation*}
For any $a\in\mathbb Z_{\ge 0}$, we have
\[\sum_{b\ge 0}\frac{(a+b)!}{a!b!}t^b=(1-t)^{-a-1},\]
hence we get
\[\begin{split}
&\quad\;\sum_{(\alpha_i)_{i=2}^r\in(\mathbb
Z_{\ge
0}^{d-1})^{r-1}}\frac{(\alpha_1+\cdots+\alpha_r)!}{\alpha_1!\cdots\alpha_r!}\prod_{j=2}^rt_j^{|\alpha_j|}\\
&=\sum_{(\alpha_i)_{i=2}^r\in(\mathbb Z_{\ge
0}^{d-1})^{r-1}}\frac{(\alpha_1+\cdots+\alpha_r)!}{\alpha_1!(\alpha_2+\cdots+\alpha_r)!}
\frac{(\alpha_2+\cdots+\alpha_r)!}{\alpha_2!\cdots\alpha_r!}\prod_{j=2}^rt_j^{|\alpha_j|}\\
&=\sum_{\alpha\in\mathbb Z_{\ge 0
}^{d-1}}\frac{(\alpha_1+\alpha)!}{\alpha_1!\alpha!}\sum_{\begin{subarray}{c}
(\alpha_i)_{i=2}^r\in(\mathbb Z_{\ge
0}^{d-1})^{r-1}\\
\alpha_2+\cdots+\alpha_r=\alpha
\end{subarray}}\frac{(\alpha_2+\cdots+\alpha_r)!}{\alpha_2!\cdots\alpha_r!}\prod_{j=2}^rt_j^{|\alpha_j|}\\
&=\sum_{\alpha\in\mathbb Z_{\ge 0
}^{d-1}}\frac{(\alpha_1+\alpha)!}{\alpha_1!\alpha!}(t_2+\cdots+t_r)^{|\alpha|}=
(1-(t_2+\cdots+t_r))^{-|\alpha_1|-d+1},
\end{split}\]
and
\[\begin{split}
&\quad\;\sum_{(m_i)_{i=2}^r\in\mathbb Z_{\ge
0}^{r-1}}\frac{(n_1+m_2+\cdots+m_r-|\alpha_1|)!}{(n_1-|\alpha_1|)!m_2!\cdots
m_r!}\prod_{j=2}^rt_j^{m_j}\\
&=\sum_{(m_i)_{i=2}^r\in\mathbb Z_{\ge
0}^{r-1}}\frac{(n_1+m_2+\cdots+m_r-|\alpha_1|)!}{
(n_1-|\alpha_1|)!(m_2+\cdots+m_r)!}\frac{(m_2+\cdots+m_r)!}
{m_2!\cdots m_r!}\prod_{j=2}^rt_j^{m_j}\\
&=\sum_{M\ge 0}\frac{(n_1-|\alpha_1|+M)!}{
(n_1-|\alpha_1|)!M!}\sum_{\begin{subarray}{c}
(m_i)_{i=2}^r\in\mathbb Z_{\ge 0}^{r-1}\\
m_2+\cdots+m_r=M
\end{subarray}}\frac{(m_2+\cdots+m_r)!}
{m_2!\cdots m_r!}\prod_{j=2}^rt_j^{m_j}\\
&=\sum_{M\ge 0}\frac{(n_1-|\alpha_1|+M)!}{
(n_1-|\alpha_1|)!M!}(t_2+\cdots+t_r)^M=(1-(t_2+
\cdots+t_r))^{-n_1+|\alpha_1|-1}.
\end{split}\]
Therefore
\begin{equation}\label{Equ:Y egale a 1}\begin{split}
H(\mathbf{t},\mathbf{X})|_{X_2=\cdots=X_r=\indic}&=\sum_{n_1\ge
0}t_1^{n_1}(1-(t_2+\cdots+t_r))^{-n_1-d}\sum_{|\alpha_1|\le
n_1}X_1^{\alpha_1} \\
&=\sum_{\mathbf{n}=(n_i)\in\mathbb Z_{\ge
0}^r}\mathbf{t}^{\mathbf{n}}\frac{(n_1+\cdots+n_r+d-1)!}{(n_1+d-1)!n_2!\cdots
n_r!}\sum_{|\alpha_1|\le n_1}X_1^{\alpha_1}.
\end{split}\end{equation}}
Similarly, for any $1\le j\le r$, we have
\begin{equation}\label{Equ:X egale a 1}
\begin{split}&\quad\;H(\mathbf{t},\mathbf{X})|_{
X_1=\cdots=X_{j-1}=X_{j+1}=\cdots=X_r=\indic}
\\&=\sum_{\mathbf{n}=(n_i)\in\mathbb Z_{\ge
0}^r}\mathbf{t}^{\mathbf{n}}\frac{(n_1+\cdots+n_r+d-1)!}{n_1!\cdots
n_{j-1}!(n_j+d-1)!n_{j+1}!\cdots
n_r!}\sum_{|\alpha_j|\le
n_j}X_j^{\alpha_j}.\end{split}\end{equation}
On the other hand,
\[
\begin{split}&\quad\;H(\mathbf{t},\mathbf{X})|_{X_1=\cdots=X_r=Y}
\\&=\sum_{\mathbf{n}
=(n_i)\in\mathbb Z_{\ge 0}^r}
\mathbf{t}^{\mathbf
n}\sum_{\begin{subarray}{c}(\alpha_i)\in(\mathbb
Z_{\ge 0}^{d-1})^r\\
|\alpha_i|\le n_i
\end{subarray}}\frac{(\alpha_1+\cdots+\alpha_r)!}
{\alpha_1!\cdots\alpha_r!}
\frac{(n_1+\cdots+n_r-|\alpha_1+\cdots+
\alpha_r|)!}{(n_1-|\alpha_1|)!\cdots(n_r-|\alpha_r|)!}
Y^{\alpha_1+\cdots+\alpha_r}.\end{split}\] By
the change of variables $m_i=n_i-|\alpha_i|$
for any $1\le i\le r$, we obtain
\[\begin{split}&\quad\;H(\mathbf{t},\mathbf{X})|_{X_1=\cdots=X_r=Y}
\\ &=\sum_{(\alpha_i)\in(\mathbb
Z_{\ge
0}^{d-1})^r}\bigg(\frac{(\alpha_1+\cdots+\alpha_r)!}
{\alpha_1!\cdots\alpha_r!}\prod_{j=1}^rt_j^{|\alpha_j|}
\bigg)
Y^{\alpha_1+\cdots+\alpha_r}\sum_{\mathbf{m}=(m_i)\in\mathbb
Z_{\ge
0}^{r}}\frac{(m_1+\cdots+m_r)!}{m_1!\cdots
m_r!}\mathbf{t}^{\mathbf{m}}
\\
&=\sum_{N\ge
0}(t_1+\cdots+t_r)^N\sum_{\gamma\in\mathbb
Z_{\ge 0}^{d-1}}
Y^\gamma(t_1+\cdots+t_r)^{|\gamma|}
=\sum_{M\ge
0}(t_1+\cdots+t_r)^M\sum_{|\gamma|\le
M}Y^\gamma,\end{split}\] where we have
performed the change of variables
$\gamma=\alpha_1+\cdots+\alpha_r$ and
$M=N+|\gamma|$ in the last equality.
Therefore, we have
\begin{equation}
\label{Equ:X egale a Y}
H(\mathbf{t},\mathbf{X})|_{X_1=\cdots=X_r=Y}
=\sum_{\mathbf{n}=(n_i)\in\mathbb Z_{\ge
0}^{r}}\frac{(n_1+\cdots+n_r)!}{n_1!\cdots
n_r!}t_1^{n_1}\cdots
t_r^{n_r}\sum_{|\gamma|\le
n_1+\cdots+n_r}Y^\gamma.
\end{equation}
Finally,
\begin{equation}\label{Equ:mass totale egale 1}
\begin{split}
&\;\quad H(\mathbf{t},(\indic,\cdots\indic))
=(1-(t_1+\cdots+t_r))^{-d}\\&=\sum_{N\ge
0}\frac{(N+d-1)!}{N!(d-1)!}\sum_{\begin{subarray}{c}
\mathbf{n}=(n_i)\in\mathbb Z_{\ge
0}^r\\
n_1+\cdots+n_r=N
\end{subarray}}\frac{N!}{n_1!\cdots
n_r!}\mathbf{t}^{\mathbf{n}}\\
&=\sum_{\mathbf{n}=(n_i)\in\mathbb Z_{\ge
0}^r }\frac{(n_1+\cdots+n_r+d-1)!}{
n_1!\cdots
n_r!(d-1)!}\mathbf{t}^{\mathbf{n}}.\end{split}\end{equation}
For any $\mathbf{n}=(n_i)\in\mathbb Z_{\ge
0}^{r}$, let
\[\rho_{\mathbf{n}}=\frac{(d-1)!n_1!\cdots n_r!}
{(n_1+\cdots+n_r+d-1)!}\sum_{\begin{subarray}{c}(\alpha_i)\in
(\mathbb Z_{\ge 0}^{d-1})^r\\ |\alpha_i|\le
n_i\end{subarray}}
\left(\frac{(\alpha_1+\cdots+\alpha_r)!}{\alpha_1!\cdots\alpha_r!}
\frac{(n_1+\cdots+n_r-|\alpha_1+\cdots+\alpha_r|)!}
{(n_1-|\alpha_1|)!\cdots(n_r-|\alpha_r|)!}\right)\\
\delta_{(\alpha_1,\cdots,\alpha_r)}.
\]
The definition of $H(\mathbf{t},\mathbf{X}) $
and the equalities \eqref{Equ:X egale a 1},
\eqref{Equ:X egale a Y} and \eqref{Equ:mass
totale egale 1} implies that
$\rho_{\mathbf{n}}$ verifies the required
conditions.
\end{proof}

We introduce some operators on the space of
Borel measures on $\mathbb R$ which we shall
use later. We denote by $C_c(\mathbb R)$ the
space of continuous functions with compact
support on $\mathbb R$. Recall that a {\it
Radon measure} on $\mathbb R$ is nothing but
a positive linear form on $C_c(\mathbb R)$.
Note that all bounded Borel measures on
$\mathbb R$ are Radon measures. We denote by
$\mathscr M_+$ the convex cone of Radon
measures on $\mathbb R$ (in the space of all
linear forms on $C_c(\mathbb R)$) and by
$\mathscr M_1$ the sub-space of Borel
probability measures on $\mathbb R$. Note
that $\mathscr M_1$ is a convex subset of
$\mathscr M_+$.

If $c$ is a real number, we denote by
$\varphi_c:\mathbb R\rightarrow\mathbb R$ the
mapping which sends $x$ to $x+c$. It induces
an automorphism of convex cone
$\tau_c:\mathscr M_+\rightarrow\mathscr M_+$
which sends $\nu\in\mathscr M_+$ to the
direct image of $\nu$ by $\varphi_c$. Thus we
define an action of $\mathbb R$ on $\mathscr
M_+$ which keeps $\mathscr M_1$ invariant,
and which preserves the order $\succ$ between
Borel measures.

If $\varepsilon$ is a strictly positive real
number, we denote by
$\gamma_{\varepsilon}:\mathbb
R\rightarrow\mathbb R$ the dilation mapping
which sends $x\in\mathbb R$ to $\varepsilon
x$. This mapping induces by direct image an
automorphism of the convex cone
$T_{\varepsilon}:\mathscr
M_+\rightarrow\mathscr M_+$ which keeps
$\mathscr M_1$ invariant and also preserves
the order $\succ$.

We now consider a vector space $V$ of finite
dimension $d$ over a field $K$. For any
integer $n\ge 0$, let $B_n=S^nV$ be the $n^{\mathrm{th}}$ symmetric power of $V$, equipped
with a separated, exhaustive and left
continuous $\mathbb R$-filtration. We shall
use Theorem \ref{Thm:theorme cle sur
equiprobability} to establish the almost
super-additivity of the measures associated
to $B_n$ ($n\ge 1$) under the condition that
the graded algebra $B=\bigoplus_{n\ge 0}B_n$
is quasi-filtered.

Choose a base $\mathbf{e}=(e_i)_{1\le i\le
d}$ of $V$. We then have a mapping
$\varphi_n:\Delta_n^{(d)}\rightarrow B_n$
which sends
$\alpha=(\alpha_1,\cdots,\alpha_d)$ to
$e^\alpha:=e_1^{\alpha_1}\cdots
e_d^{\alpha_d}$. The image of
$\Delta_n^{(d)}$ by $\varphi_n$ is a base of
$B_n$. There exists, for each $n\in\mathbb
N$, a maximal base
$\mathbf{u}^{(n)}=(u_{\alpha})_{\alpha\in\Delta_n^{(d)}}$
of $B_n$ such that (see Proposition
\ref{Pro:construction de base maximale a
partir une base quelconique} {\it infra})
\begin{equation}\label{Equ:forme de ualpha}u_{\alpha}\in
e^\alpha+\sum_{\beta<\alpha}Ke^\beta.\end{equation}
If ${\mathbf{n}}=(n_i)_{1\le i\le r
}\in\mathbb Z_{\ge 0}^r$ and
$N=n_1+\cdots+n_r$, for any
$\gamma\in\Delta_{N}^{(d)}$, let
$u^{({\mathbf n})}_\gamma$ be an element in
\[\left\{\left.\prod_{i=1}^ru_{\alpha_i}\;\right|\;\alpha_i\in\Delta_{n_i}^{(d)},\;\sum_{i=1}^r\alpha_i=\gamma\right\}.\]
such that
\[\lambda(u^{(\mathbf{n})}_\gamma)=\max_{\begin{subarray}{c}
\alpha_i\in\Delta_{n_i}^{(d)}\\
\alpha_1+\cdots+\alpha_r=\gamma
\end{subarray}}
\lambda(u_{\alpha_1}\cdots u_{\alpha_r}).\]
From \eqref{Equ:forme de ualpha}, we deduce
\[u_\gamma^{(\mathbf{n})}\in e^\gamma+\sum_{\delta<\gamma}Ke^\delta.\]
Hence
$\mathbf{u}^{(\mathbf{n})}:=(u_{\gamma}^{(\mathbf{n}
)})_{\gamma\in\Delta_{N}^{(d)}}$ is a base of
$B_{N}$.

\begin{proposition}
\label{Pro:suradditivite de integrale} Let
$f:\mathbb Z_{\ge 0}\rightarrow\mathbb R_{\ge
0}$ be a function, $c$ be a positive real
number and $g:\mathbb R\rightarrow\mathbb R$
be a concave increasing $c$-Lipschitz
function. Suppose that the graded algebra
$B=\bigoplus_{n\ge 0}B_n$ is
$f$-quasi-filtered. If for any integer $n\ge
0$, denote by
\[I_n=\int_\mathbb Rg\,\mathrm{d}\Big(T_{\frac 1n}\nu_{B_n}\Big),\]
then for any integer $r\ge 2$ and any
$\mathbf{n}=(n_i)\in\mathbb Z_{\ge n_0}^r$,
by writing $N=n_1+\cdots+n_r$, we have
\[NI_{N}\ge \sum_{i=1}^r \Big(n_iI_{n_i}-cf(n_i)\Big).\]
\end{proposition}
\begin{proof} For any integer $n\ge 0$,
denote by $\xi_n$ the equidistributed measure
on $\Delta_n^{(d)}$, by $\rho_{\mathbf{n}}$ a
measure on
$\Delta_{n_1}^{(d)}\times\cdots\times\Delta_{n_r}^{(d)}$
satisfying the conditions of Theorem
\ref{Thm:theorme cle sur equiprobability},
and by $\mathbf{u}^{(\mathbf{n})}$ the base
of $B_N$ constructed as above. Then by
Definition \ref{Def:mesure associee a une
filtration},
\[\begin{split}I_{N}&\ge\int_\mathbb
Rg\,{\mathrm{d}}\Big(T_{\frac
1N}\nu_{\mathbf{u}^{(\mathbf{n})}}\Big)
=\int_{\Delta_{N}^{(d)}}g
\left(\frac{1}{N}\lambda(u_{\gamma}^{(\mathbf{n})})\right)
\mathrm{d}\xi_{N}(\gamma)\\
&=\int_{\Delta_{n_1}^{(d)}\times\cdots\times\Delta_{n_r}^{(d)}}g
\left(\frac{1}{N}\lambda(u_{\alpha_1+\cdots+\alpha_r}^{(\mathbf{n})})
\right)
\mathrm{d}\rho_{\mathbf{n}}(\alpha_1,\cdots,\alpha_r).\end{split}\]
Since $g$ is an increasing function, by
definition of $u_{\gamma}^{(\mathbf{n})}$, we
have
\[
I_N\ge\int_{\Delta_{n_1}^{(d)}
\times\cdots\times\Delta_{n_r}^{(d)}}
g\left(\frac{1}{N}\lambda(u_{\alpha_1}\cdots
u_{\alpha_r}) \right)
\mathrm{d}\rho_{\mathbf{n}}(\alpha_1,\cdots,\alpha_n).\]
Since $B$ is an $f$-quasi-filtered graded
algebra and since $g$ is increasing, we
obtain
\[
I_N\ge\int_{\Delta_{n_1}^{(d)}\times\cdots\times\Delta_{n_r}^{(d)}}
g\left(\frac{1}{N}\sum_{i=1}^r\Big(\lambda(u_{\alpha_i})
-f(n_i)\Big) \right)
\mathrm{d}\rho_{\mathbf{n}}(\alpha_1,\cdots,\alpha_r).\]
Since the function $g$ is $c$-Lipschitz, then
\[I_N\ge\int_{\Delta_{n_1}^{(d)}\times\cdots\times\Delta_{n_r}^{(d)}}
\left[g\left(\frac{1}{N}\sum_{i=1}^r\lambda(u_{\alpha_i})
\right)-\frac{c}{N}\sum_{i=1}^rf(n_i)
\right]\mathrm{d}\rho_{\mathbf{n}}(\alpha_1,\cdots,\alpha_r).\]
Then the concavity of $g$ implies that
\[
I_N\ge\int_{\Delta_{n_1}^{(d)}\times\cdots\times\Delta_{n_r}^{(d)}}
\left[\sum_{i=1}^r\frac{n_i}{N}
g\left(\frac{\lambda(u_{\alpha_i})}{n_i}\right)\right]
\mathrm{d}\rho_{\mathbf{n}}(\alpha_1,\cdots,\alpha_r)-\frac{c}{N}
\sum_{i=1}^rf(n_i)\] Finally, since the
direct image of $\rho_{\mathbf{n}}$ by the
$r$ projections are equidistributed, we
obtain that
\[I_N\ge\sum_{i=1}^r\frac{n_i}{N}I_{n_i}-\frac{c}{N}\sum_{i=1}^rf(n_i).
\]
\end{proof}

\begin{corollary}\label{Cor:suradditivite de l'integrale}
With the notations of Proposition
\ref{Pro:suradditivite de integrale}, if the
sequence $(I_n)_{n\ge 0}$ is bounded from
above (for example if $g$ is bounded from
above, or if there exists $a\in\mathbb R$
such that $\supp
(\nu_{B_n})\subset]-\infty,na]$ for any
sufficiently large integer $n$) and if
$\displaystyle\lim_{n\rightarrow+\infty}f(n)/n=0$,
then the sequence $(I_n)_{n\ge 0}$ has a
limit when $n\rightarrow+\infty$.
\end{corollary}
\begin{proof}
It is a consequence of
Proposition \ref{Pro:suradditivite de
integrale} and Corollary \ref{Cor:limite
de an sur n forme forte}.
\end{proof}

\section{Polygon associated to a Borel measure}

\hskip\parindent We explain in this section
how to associate to a Borel probability
measure on $\mathbb R$ a concave function on
$[0,1]$ which takes zero value at the origin.
Furthermore, if the measure is a linear
combination of Dirac measures, then the
associated concave function is piecewise
linear, therefore is a {\it polygon} on
$[0,1]$.

If $f:\mathbb R\rightarrow [0,1]$ is a right
continuous decreasing function such that
\[\lim_{x\rightarrow
-\infty}f(x)=1,\quad\text{and}\quad
\lim_{x\rightarrow+\infty}f(x)=0,\] we define
the {\it quasi-inverse} of $f$ the function
$f^*:]0,1[\rightarrow\mathbb R$ which sends
any $t\in ]0,1[$ to $\sup\{x\;|\; f(x)>t\}$.
The following properties of $f^*$ are easy to
verify.

\begin{proposition}\label{Pro:proprietes de quasi-inverse}
Let  $f:\mathbb R\rightarrow [0,1]$ be a
right continuous decreasing function such
that $\displaystyle\lim_{x\rightarrow -\infty}f(x)=1$ and
$\displaystyle\lim_{x\rightarrow+\infty}f(x)=0$. Then
\begin{enumerate}[1)]
\item for any $t\in ]0,1[$ and any
$y\in\mathbb R$, $f(y)>t$ if and only if
$y<f^*(t)$;
\item $f^*$ is a right continuous decreasing function;
\item $\displaystyle\sup_{t\in]0,1[}f^*(t)=\inf\{x\in\mathbb
R\;|\;f(x)=0\}$ and
$\displaystyle\inf_{t\in]0,1[}f^*(t)=\sup\{x\in\mathbb
R\;|\;f(x)=1\}$ (by convention
$\inf\emptyset=+\infty$ and
$\sup\emptyset=-\infty$).
\end{enumerate}
\end{proposition}

\begin{proposition}\label{Pro:le polygone associe a une filtration}
Let $\nu$ be a Borel probability measure on
$\mathbb R$ which is a linear combination of
Dirac measures. If we denote by $f:\mathbb
R\rightarrow [0,1]$ the function
$f(x)=\nu\big(]x,+\infty[\big)$, then the
function on $[0,1]$ defined by $\displaystyle
P(\nu)(t){:=}\int_0^tf^*(s)\mathrm{d}s$ is a
polygon on $[0,1]$.
\end{proposition}
\begin{proof}
Since $\nu$ is a linear combination of Dirac
measures, the function $f$ is decreasing,
right continuous, and piecewise constant.
Furthermore, $f(x)=0$ (resp. $f(x)$=1) when
$x$ is sufficiently positive (resp.
negative). Therefore, $f^*$ is decreasing,
right continuous, piecewise constant and
bounded. As $P(\nu)$ is the primitive
function of $f^*$, which takes zero value at
the origin, we obtain that $P(\nu)$ is a
polygon.
\end{proof}

Actually, $P(\nu)$ is just the {\it Legendre
transformation} of the concave function
$x\mapsto\int_{0}^xf(y)\mathrm{d}y$ (see
\cite{Homander94} II.2.2), which is called
the {\it polygon} associated to the Borel
probability measure $\nu$.

We can calculate explicitly $P(\nu)$. Suppose
that $\nu$ is of the form
$\displaystyle\sum_{i=1}^n
(t_i-t_{i-1})\delta_{a_i}$, where
$a_1>\cdots>a_n$, and $0=t_0<\cdots<t_n=1$.
Then $\displaystyle
f(x)=\indic_{]-\infty,a_n[}(x)+\sum_{i=1}^{n-1}t_i\indic_{[a_{i+1},a_i[}(x),
$ and hence $\displaystyle
f^*(t)=a_0\indic_{]0,t_1[}(t)+\sum_{i=2}^{n}a_i\indic_{[t_{i-1},t_i[}(t)
$. Therefore,
\[\displaystyle
P(\nu)(t)=\sum_{i=1}^{j-1}a_i(t_i-t_{i-1})
+a_j(t-t_{j-1}),\quad t\in
[t_{j-1},t_{j}],\quad 1\le j\le n.\]

If $V$ is a non-zero vector space of finite
rank over $K$ and $\mathcal F$ is a separated
and exhaustive filtration of $V$. We call
{\it polygon} associated to the filtration
$\mathcal F$ the polygon $P(\nu_{\mathcal
F,V})$ on $[0,1]$, denoted by $P_{\mathcal
F,V}$ (or simply $P_V$ if there is no
ambiguity on the filtration).

Suppose in addition that the filtration
$\mathcal F$ is left continuous. Then
$\mathcal F$ corresponds to a flag
\[0=V^{(0)}\subsetneq V^{(1)}\subsetneq\cdots\subsetneq V^{(n)}=V\]
and a strictly decreasing sequence
$(a_i)_{1\le i\le n}$. We have shown that its
associated probability measure is
$\displaystyle\nu_{\mathcal
F,V}=\sum_{i=1}^n\Big(\frac{\rang
V^{(i)}}{\rang V}-\frac{\rang
V^{(i-1)}}{\rang V}\Big)\delta_{a_i}$.
Therefore we have, for any integer $1\le j\le
n$ and any $t\in
\displaystyle\Big[\frac{\rang
V^{(j-1)}}{\rang V},\frac{\rang
V^{(j)}}{\rang V}\Big]$,
\[P_{\mathcal F,V}(t)=\sum_{i=1}^{j-1}
a_i\Big(\frac{\rang V^{(i)}}{\rang
V}-\frac{\rang V^{(i-1)}}{\rang
V}\Big)+a_j\Big(t-\frac{\rang
V^{(j-1)}}{\rang V}\Big).\]

For a general Borel probability measure
$\nu$, similarly to Proposition \ref{Pro:le
polygone associe a une filtration}, we can
also define a concave function $P(\nu)$. If
$\nu_1$ and $\nu_2$ are two Borel probability
measures on $\mathbb R$ such that
$\nu_1\succ\nu_2$, then $P(\nu_1)\ge
P(\nu_2)$. Furthermore, for any real number
$a$, $P(\tau_a\nu)(t)=P(\nu)(t)+at$ and for
any strictly positive number $\varepsilon$,
$P(T_{\varepsilon}\nu)=\varepsilon P(\nu)$.

In the following, we explain that the vague
convergence of Borel probability measures
implies the uniform convergence of associated
polygons. With this observation, to prove the
convergence of polygons, it suffice to
establish the vague convergence of
corresponding probability measures. We begin
by presenting some properties of Borel
probability measures.

\begin{lemma}\label{Lem:majoration de dilatation et translation}
For any function $f\in C_c(\mathbb R)$, we
have
\[\lim_{\varepsilon\rightarrow 0}
\|f\circ\gamma_{1+\varepsilon}-f\|_{\sup}=0,
\quad \lim_{\varepsilon\rightarrow
0}\|f\circ\varphi_{\varepsilon}-f\|_{\sup}=0.\]
\end{lemma}
\begin{proof}
Suppose that $\supp(f)\subset[-K,K]$ ($K>0$).
For any number $-1/2<\varepsilon<1/2$,
\[\|f\circ\gamma_{1+\varepsilon}-f\|_{\sup}=\sup_{-2K\le x\le 2K}|f(x+\varepsilon x)-f(x)|.\]
Since $f$ is uniformly continuous,
$\displaystyle\lim_{\varepsilon\rightarrow
0}\sup_{-2K\le x\le 2K}|f(x+\varepsilon
x)-f(x)|=0$, so we have
$\displaystyle\lim_{\varepsilon\rightarrow
0}\|
f\circ\gamma_{1+\varepsilon}-f\|_{\sup}=0$.
The other assertion is just the definition of
uniform continuity of $f$.
\end{proof}

\begin{definition}
If $(\nu_n)_{n\ge 1}$ is a sequence of Radon
measures on $\mathbb R$ and if $\nu$ is a
Radon measure on $\mathbb R$, we say that
$(\nu_n)_{n\ge 1}$ {\it converges vaguely} to
$\nu$ if for any function $f\in C_c(\mathbb
R)$, the sequence $(\int_{\mathbb
R}f\mathrm{d}\nu_n)_{n\ge 1}$ converges to
$\int_{\mathbb R}f\mathrm{d}\nu$.
\end{definition}

\begin{proposition}\label{Pro:convergence vague par dilatation et translation}
Let $(\nu_n)_{n\ge 1}$ be a sequence of Radon
measures on $\mathbb R$, $\nu$ be a Radon
measure on $\mathbb R$, and $(a_n)_{n\ge 1}$
be a sequence of real numbers in
$]-1,+\infty[$ which converges to $0$.
Suppose that the total masses of
$(\nu_n)_{n\ge 1}$ are uniformly bounded.
Then the following conditions are
equivalents:
\begin{enumerate}[1)]
\item the sequence $(\nu_n)_{n\ge 1}$ converges vaguely to $\nu$;
\item the sequence $(T_{1+a_n}\nu_n)_{n\ge 1}$
converges vaguely to $\nu$;
\item the sequence $(\tau_{a_n}\nu_n
)_{n\ge 1}$ converges vaguely to $\nu$.
\end{enumerate}
\end{proposition}
\begin{proof}
Since $\tau_{a_n}^{-1}=\tau_{-a_n}$ and
$T_{1+a_n}^{-1}=T_{(1+a_n)^{-1}}=T_{1-\frac{a_n}{1+a_n}}$,
it suffices to verify ``1)$\Longrightarrow$
2)'' and ``1)$\Longrightarrow$3)'', which are
immediate consequences of Lemma
\ref{Lem:majoration de dilatation et
translation}.
\end{proof}

\begin{lemma}\label{Lem:convergence vague vers une proba}
Let $(\nu_n)_{n\ge 1}$ be a sequence of Borel
probability measures on $\mathbb R$ which
converges vaguely to a measure $\nu$. If the
supports of $(\nu_n)_{n\ge 1}$ are uniformly
bounded, then $\nu$ is also a probability
measure.
\end{lemma}
\begin{proof} Suppose
$\displaystyle\bigcup_{n\ge
1}\supp(\nu_n)\subset[m,M]$. If $\varphi$ is
a continuous function with compact support
which takes values in $[0,1]$ and such that
$\varphi|_{[m,M]}=1$. We have
$\displaystyle\int_{\mathbb R}\varphi
\mathrm{d}\nu=
\lim_{n\rightarrow\infty}\int_{\mathbb R}
\varphi \mathrm{d}\nu_n= 1$. Since $\varphi$
is arbitrary, we obtain $\nu(\mathbb R)=1$.
\end{proof}

\begin{proposition}\label{Pro:convergence de fonction de repartition}
Let $(\nu_n)_{n\ge 1}$ be a sequence of Borel
probability measures on $\mathbb R$ which
converges vaguely to a measure $\nu$. Suppose
that the supports of $\nu_n$ are uniformly
bounded. Let $F_n$ (resp. $F$) be the
distribution function of $\nu_n$ (resp.
$\nu$). Then there exists a numerable subset
$Z$ of $\mathbb R$ such that, for any point
$x\in\mathbb R\setminus Z$, the sequence
$(F_n(x))_{n\ge 1}$ converges to $F(x)$.
\end{proposition}
\begin{proof} After Lemma \ref{Lem:convergence vague vers une proba},
$\nu$ is a probability measure. Let $Z$ be
the set of $x\in\mathbb R$ such that
$\nu(\{x\})\neq 0$. Since $\nu$ is of total
mass $1$, the set $Z$ is numerable. If $r$ is
a point in $\mathbb R\setminus Z$, the set of
discontinuous points of the function
$\indic_{]-\infty,r]}(x)$ is $\{r\}$, which
is $\mu$-negligible. After \cite{Bourbaki65}
IV.5 Proposition 22, the sequence
$(F_n(x))_{n\ge 1}$ converges to $F(x)$.
\end{proof}

\begin{proposition}
Let $(f_n)_{n\ge 1}$ be a sequence of right
continuous decreasing functions valued in
$[0,1]$ such that
\begin{enumerate}[i)]
\item
$\displaystyle \sup_{n\ge 1}\inf\{x\in\mathbb
R\;|\;f_n(x)=0\}<+\infty$,
$\displaystyle\inf_{n\ge 1}\sup\{x\in\mathbb
R\;|\;f_n(x)=1\}>-\infty$;
\item there exists a numerable subset $Z$ of
$\mathbb R$ such that, for any $x\in\mathbb
R\setminus Z$, the sequence $(f_n(x))_{n\ge
1}$ converges.
\end{enumerate}
Let $f:\mathbb R\rightarrow[0,1]$ be a right
continuous function such that
$f(x)=\displaystyle\lim_{n\rightarrow+\infty}f_n(x)$
for any $x\in\mathbb R\setminus Z$. Then
\begin{enumerate}[1)]
\item the function $f$ is decreasing;
\item if we write $\displaystyle
A:=\liminf_{n\rightarrow+\infty}\inf\{x\in\mathbb
R\;|\;f_n(x)=0\}$, $\displaystyle B:=
\limsup_{n\rightarrow+\infty}\sup\{x\in\mathbb
R\;|\;f_n(x)=1\}$, then
$f|_{]A,+\infty[}\equiv 0$,
$f|_{]-\infty,B[}\equiv 1$;
\item there exists a numerable subset
$Z'$ of $]0,1[$ such that $(f_n^*(t))_{n\ge 1
}$ converges to $f^*(t)$ for any
$t\in]0,1[\setminus Z'$;
\item the function sequence
$(\int_0^tf_n^*(s)\mathrm{d}s)_{n\ge 0}$
converges uniformly to
$\int_0^tf^*(s)\mathrm{d}s$.
\end{enumerate}
\end{proposition}
\begin{proof} 1) and 2) are easy to verify.

3) After the condition i), the function
$f_n^*$ is well defined for any $n\ge 1$.
After 2), the function $f^*$ is well defined.
If $t$ is a number in $]0,1[$ and if
$y=f^*(t)$, then there exists a a strictly
increasing sequence $(x_m)_{m\ge 1
}\subset\mathbb R\setminus Z$ which converges
to $y$. Since $x_m<y$, we have $f(x_m)>t$.
Since $x_m\in\mathbb R\setminus Z$, there
exists $N(m)\in\mathbb Z_{\ge 0}$ such that
$f_n(x_m)>t$ (i.e., $x_m<f_n^*(t)$) for any
$n>N(m)$, which implies that
$\displaystyle\liminf_{n\rightarrow+\infty}f_n^*(t)\ge
f^*(t)$.

For any integer $n\ge 1$, let $Z_n'$ be the
set of all $t\in]0,1[$ such that
$f_n^{-1}(\{t\})$ has an interior point.
Clearly $Z_n'$ is a numerable set. Let $Z''$
be the set of $t\in]0,1[$ such that
$f^{-1}(\{t\})$ has an interior point. Let
$Z'$ be the union of all $Z_n'$ and $Z''$. It
is also a numerable subset of $]0,1[$. Let
$t$ be a point in $]0,1[\setminus Z'$ and
$y=f^*(t)$. We take a strictly decreasing
sequence $(x_m)_{m\ge 1 }\subset\mathbb
R\setminus Z$ which converges to $y$. Since
$y\not\in Z''$, we have $f(x_m)<t$.
Therefore, there exists $N(m)\in\mathbb
Z_{>0}$ such that, for any $n>N(m)$,
$f_n(x_m)<t$ and {\it a fortiori} $x_m\ge
f_n^*(t)$. We then have
$\displaystyle\limsup_{n\rightarrow+\infty}f_n^*(t)\le f^*(t)$.

4) After Proposition \ref{Pro:proprietes de
quasi-inverse} 3), the condition i) implies
that the functions $f_n^*$ are uniformly
bounded. On the other hand, $f_n^*-f$
converges almost everywhere to the zero
function. After the Lebesgue's dominate
convergence theorem, we obtain that
\[\left|\int_0^tf_n^*(s)\mathrm{d}s-
\int_0^tf^*(s)\mathrm{d}s\right|\le
\int_0^1|f_n^*(s)-f^*(s)|\mathrm{d}s\]
converges to $0$ when $n\rightarrow+\infty$.
\end{proof}

\section{Convergence of polygons of a quasi-filtered
graded algebra}\label{Sec:Convergence of
polygons of a quasi-filtered graded algebra}

\hskip\parindent We establish in this section
the convergence of polygons of a
quasi-filtered graded algebra. By using the
results obtained in Section
\ref{Sec:Quasi-filtered graded algebras}, we
show that the measures associated to a
quasi-filtered graded algebra converge
vaguely to a Borel probability measure on
$\mathbb R$, and therefore, the associated
polygons converge uniformly to a concave
function on $[0,1]$.

We first recall some facts about Poincar\'e
series of a graded module, which we shall use
later. Let $A$ be an Artinian ring, $B$ be a
graded $A$-algebra of finite type and
genenrated by $B_1$, and $M$ be a non-zero
graded $B$-module of finite type. For any
$n\in\mathbb Z$, $M_n$ is an $A$-module of
finite type, therefore of finite length. We
denote by $H_M$ the Poincar\'e series
associated to $M$, i.e.,
$H_M(X)=\sum_{n\in\mathbb
Z}\len_A(M_n)X^n\in\mathbb Z\lbr X\rbr$. The
theory of Poincar\'e series affirms (cf.
\cite{Bourbaki83}) that there exists an
integer $r\ge 0$ such that $H_M(X)$ can be
written in the form
\begin{equation}\label{Equ:serie de Poincare}
H_M(X)=a_r(X)(1-X)^{-r}+a_{r-1}(X)(1-X)^{-r+1}+\cdots+a_0(X),
\end{equation}
where $a_0,\cdots,a_{r}$ are elements in $\mathbb Z[X,X^{-1}]$, $a_r$ being non-zero and having positive coefficients if $M$ is non-zero. Moreover, the values $r$ and $a_r(1)$ don't depend on the choice of
$(a_0,\cdots,a_{r})$. In fact, $r$ identifies with the dimension of $M$. We write
$c(M)=a_r(1)$. Clearly we have
\[\len_A(M_n)=\frac{c(M)}{(r-1)!}n^{r-1}+o(n^{r-1})\]
when $n\rightarrow+\infty$, and there exists a polynom $Q_M$ with coefficients in
$\mathbb Q$ such that $Q_M(n)=\len_A(M_n)$ for
sufficiently large integer $n$. If $M$ is the zero $B$-module, by convention we define $\dim(M)=-\infty$
and $c(M)=0$.

If
$\xymatrix{0\ar[r]&M'\ar[r]&M\ar[r]&M''\ar[r]&0}$
is a short exact sequence of graded
$B$-modules of finite type, we have
$H_{M}=H_{M'}+H_{M''}$. Therefore, $ \dim
M=\max(\dim M',\dim M'')$ and
\begin{equation}\label{Equ:comparaison des indices}
c(M)=c(M')\indic_{\{\dim M'\ge\dim
M''\}}+c(M'')\indic_{\{\dim M''\ge\dim M'\}}.
\end{equation}

\begin{definition}\label{Def:la condition de convergence vague}
Let $K$ be a field, $B$ be a graded
$K$-algebra of finite type which is generated
by $B_1$ and $M$ be a graded $B$-module of
finite type and of dimension $d>0$. Suppose
that for each integer $n\ge 0$, $M_n$ is
equipped with a separated, exhaustive and
left continuous $\mathbb R$-filtration. We
say that $M$ satisfies the {\it vague
convergence condition} and we write
$\mathbf{CV}(M)$ if the sequence of Radon
measures $(T_{\frac 1n}\nu_{M_n})_{n\ge 1}$
converges  vaguely. Finally, if $N$ is a
graded $B$-module which is of dimension $0$
or is zero, then by convention $N$ satisfies
the vague convergence condition (in fact, for
any sufficiently large integer $n$, we have
$N_n=0$, so $T_{\frac 1n}\nu_{N_n}$ is the
zero measure).
\end{definition}

Although not explicitly stated, in Section
\ref{Sec:Quasi-filtered graded algebras}, we
have essentially proved the vague convergence
of measures associated to a quasi-filtered
symmetric algebra. We now state this result
as follows.

\begin{proposition}\label{Pro:vague convergnce de mesures}
Let $f:\mathbb Z_{\ge 0}\rightarrow\mathbb
R_{\ge 0}$ be a function such that
$\displaystyle\lim_{n\rightarrow+\infty}f(n)/n=0$
and $V$ be a vector space of dimension
$0<d<+\infty$ over $K$. For any integer $n\ge
0$ let $B_n=S^nV$. Suppose that each vector
space $B_n$ is equipped with a separated,
exhaustive and left continuous $\mathbb
R$-filtration such that the graded algebra
$B=\bigoplus_{n\ge 0}B_n$ is
$f$-quasi-filtered. Then $B$ satisfies the
vague convergence condition.
\end{proposition}
\begin{proof} For any integer $n\ge 1$,
denote by $\nu_n=T_{\frac 1n}\nu_{B_n}$. Let
$G$ be the set of Borel functions $g$ on
$\mathbb R$ such that, for any $n\in\mathbb
Z_{\ge 0}$, $g$ is integrable with respect to
$\nu_{n}$ and such that $(\int
g\mathrm{d}\nu_n)_{n\ge 0}$ converges.
Corollary \ref{Cor:suradditivite de
l'integrale} implies that $G$ contains all
increasing, bounded from above, concave and
Lipschitz functions. Cleary $G$ is a vector
space over $\mathbb R$. Suppose that $f$ is a
function in $C_0^\infty(\mathbb R)$. Let
$I=[a,b]$ be an interval containing the
support of $f$. Notice that $f'$ and $f''$
are also smooth functions and the supports of
$f'$ and of $f''$ are contained in $I$.
Therefore, $f'$ and $f''$ are bounded
functions. Let $C=\|f'\|_{\sup}$ and
$C'=\|f''\|_{\sup}/2$. Let $h$ be the
function \[ h(x)=\begin{cases}
C'(b-a)(2x-a-b)+C(x-b),&x\le
a,\\
-C'(b-x)^2+C(x-b),&a< x\le b,\\
0,&x> b.
\end{cases}
\] It's a concave, increasing and
$(2C'(b-a)+C)$-Lipschitz function which is
bounded from above by $0$. Hence $h\in G$. On
the other hand, $h+f$ is also a concave
function since $h''=-2C'$ on $I$. It is also
increasing because $h'(x)\ge 0$ on $\mathbb
R$ and $h'(x)\ge C$ on $I$. Furthermore, it
is $(2C'(b-a)+2C)$-Lipschitz and bounded from
above by $\|f\|_{\sup}$. Therefore, we have
$h+f\in G$. We then deduce $f\in G$. Finally,
since $C_0^\infty(\mathbb R)$ is dense in the
normed space $(C_c(\mathbb
R),\|\cdot\|_{\sup})$, we obtain $C_c(\mathbb
R)\subset G$.

Let $S:C_c(\mathbb R)\rightarrow\mathbb R$ be
the opeartor which associates to each
continous function $g$ with compact support
the limit of the sequence $(\int
g\mathrm{d}\nu_n)_{n\ge 1}$. It's a linear
operator. Furthermore, if $g$ is a positive
function in $C_c(\mathbb R)$, then $\int
g\mathrm{d}\nu_n\ge 0$ for any $n\in\mathbb
Z_{\ge 0}$. Therefore, we have $S(g)\ge 0$.
After Riesz's representation theorem, there
exists a unique Radon measure $\nu$ on
$\mathbb R$ such that $S(g)=\int
g\mathrm{d}\nu$. By definition the sequence
$(\nu_n)_{n\ge 1}$ converges vaguely to
$\nu$.
\end{proof}

In the following, we shall establish the
vague convergence (Theorem
\ref{Thm:convergence vague de algebre gradue
filtree}) for a general quasi-filtered graded
algebra of finite type over a field. We begin
by introducing two technical lemmas
(\ref{Lem:suite exacte et la comdtion de
convergence vague} and \ref{Lem:comparaison
des mesures par un iso}), which are useful to
prove Theorem \ref{Thm:convergence vague de
algebre gradue filtree}.

\begin{lemma}\label{Lem:suite exacte et la comdtion de convergence vague}
Let $B$ be a graded $K$-algebra of finite
type which is generated by $B_1$,
\[\xymatrix{0\ar[r]&M'
\ar[r]^\phi&M\ar[r]^\pi&M''\ar[r]&0}\] be a
short exact sequence of graded $B$-modules of
finite type. We denote by $d'=\dim M'$,
$d=\dim M$ and $d''=\dim M''$. Suppose that
for any integer $n\ge 0$ (resp. $n$), $B_n$
(resp. $M_n$) is equipped with a separated,
exhaustive and left continuous $\mathbb
R$-filtration. Suppose furthermore that for
each integer $n\ge 0$, $M_n'$ (resp. $M_n''$)
is equipped with the inverse image filtration
(resp. strong direct image filtration), then
\begin{enumerate}[1)]
\item if $d'>d''$, then $\mathbf{CV}(M')\Longleftrightarrow\mathbf{CV}(M)$;
\item if $d''>d'$, then $\mathbf{CV}(M'')\Longleftrightarrow\mathbf{CV}(M)$;
\item if $d'=d''$, then $\mathbf{CV}(M') \text{ and }\mathbf{CV}(M'')\Longrightarrow\mathbf{CV}(M)$.
\end{enumerate}
\end{lemma}
\begin{proof} Let $\alpha'=c(
M')$, $\alpha=c(M)$, and $\alpha''=c( M'')$.
If $\dim M'=0$, then for sufficiently large
$n$, we have $M_n=M_n''$, so
$\mathbf{CV}(M'')\Longleftrightarrow\mathbf{CV}(M)$.
Hence the proposition is true when $\dim
M'=0$. Similarly it is also true when $\dim
M''=0$. In the following we suppose
$\min(d',d'')\ge 1$. We then have
$d=\max(d',d'')$. For any integer $n\ge 0$,
let
\begin{gather*}
\nu_n'=T_{\frac 1n}\nu_{M_n'},\quad
\nu_n=T_{\frac 1n}\nu_{M_n},\quad
\nu_n''=T_{\frac 1n}\nu_{M_n''}\end{gather*}
and\[ r_n'=\rang M_n',\quad r_n=\rang
M_n,\quad r_{n}''=\rang M_n''.\]
For sufficiently large integer $n$, $r_n'$,
$r_n$ and $r_n''$ are strictly positive.
Moreover, by Proposition \ref{Pro:relation de
mesure associe pour suite exacte},
$\displaystyle
\nu_n=\frac{r'_n}{r_n}\nu'_n+\frac{r''_n}{r_n}\nu''_n$.
The measures $\nu_n'$, $\nu_n$ and $\nu_n''$
are of bounded total masses, and we have the
following estimations:
\[r_n'=\frac{\alpha'}{(d'-1)!}n^{d'-1}+o(n^{d'-1}),
\quad
r_n''=\frac{\alpha''}{(d''-1)!}n^{d''-1}+o(n^{d''-1}),\quad
r_n=r_n'+r_n''.\]

1) If $d'>d''$, then
$\displaystyle\lim_{n\rightarrow+\infty}r_n'/r_n=1$,
$\displaystyle\lim_{n\rightarrow+\infty}r_n''/r_n=0$,
so $(\nu_n)_{n\ge 1}$ converges vaguely if
and only if $(\nu_n')_{n\ge 1}$ converges
vaguely, and if it is the case, they have the
same limit.

2) It is similar to 1).

3) If $d''=d'$, then
$\alpha=\alpha'+\alpha''$, and
\[\lim_{n\rightarrow+\infty}\frac{r_n'}{r_n}=
\frac{\alpha'}{\alpha'+\alpha''}, \quad
\lim_{n\rightarrow+\infty}\frac{r_n''}{r_n}=
\frac{\alpha''}{\alpha'+\alpha''}.\] If
$(\nu_n')_{n\ge 1}$ converges vaguely to
$\nu'$ and if $(\nu_n'')_{n\ge 1}$ converges
vaguely to $\nu''$, then $(\nu_n)_{n\ge 1}$
converges vaguely to
$\displaystyle\frac{\alpha'}{\alpha'+\alpha''}\nu'+
\frac{\alpha''}{\alpha'+\alpha''}\nu''$.
\end{proof}

\begin{lemma}\label{Lem:comparaison des mesures par un iso}
Let $V$ and $V'$ be two vector spaces of
finite rank over $K$, equipped with
separated, exhaustive and left continuous
$\mathbb R$-filtrations,
$\varphi:V\rightarrow V'$ be an isomorphism
of vector spaces over $K$ and $c$ be a real
number. If
$\lambda(x)\le\lambda(\varphi(x))+c$ for any
element $x\in V$, then
$\nu_V\prec\tau_{c}\nu_{V'}$.
\end{lemma}
\begin{proof}
Let $\mathbf{e}=(e_i)_{1\le i\le n}$ be a
maximal base of $V$. Then
$\mathbf{e}'=(\varphi(e_i))_{1\le i\le n}$ is
a base of $V'$. Hence
$\displaystyle\tau_{c}\nu_{V'}\succ\tau_{c}\nu_{\mathbf{e}'}=
\frac
1n\sum_{i=1}^n\delta_{\lambda(\varphi(e_i))+c}
\succ\frac{1}{n}\sum_{i=1}^n\delta_{\lambda(e_i)}=\nu_{E}$.
\end{proof}

We now present the general convergence
theorem. As we have already proved the
special case of symmetric algebras, it is
quite natural to expect that the general case
follows by using the method of unscrewing.
However, as we shall see in Remark
\ref{Rem:contreexample}, the theorem cannot
be generalized to quasi-filtered graded
modules. Even for modules generated by one
homogeneous element, the convergence of
associated probability measures fails in
general. Therefore, the first step of
unscrewing doesn't work. In fact, the major
difference between filtration and grading is
that, in a graded algebra, the homogeneous
degree of the product of two homogeneous
elements equals to the sum of homogeneous
degrees, as for (quasi-)filtrated algebra, we
only give a lower bound for the index of the
product, which doesn't prevent the product
going ``far away'' in the filtration. More
precisely, the graded algebra associated to a
filtered algebra of finite type over $K$ need
not be of finite type over $K$ in general.

The proof of the theorem below uses the
Noether's normalization theorem, which
provides a subalgebra isomorphic to a
symmetric algebra over which the algebra is
finite. It is this finiteness which prevents
the product of two element from going too
``far away''.

\begin{theorem} \label{Thm:convergence vague de algebre gradue filtree}
Let $f:\mathbb Z_{\ge 0}\rightarrow\mathbb
R_{\ge 0}$ be a function such that
$\displaystyle\lim_{n\rightarrow+\infty}{f(n)}/{n}=0$
and $B=\bigoplus_{n\ge 0}B_n$ be an integral
graded $K$-algebra of finite type over $K$,
which is generated by $B_1$ as $K$-algebra.
Suppose that
\begin{enumerate}[i)]
\item $d=\dim B$ is strictly positive,
\item for any positive integer $n$,
$B_n$ is equipped with an $\mathbb
R$-filtration $\mathcal F$ which is
separated, exhaustive and left continuous,
such that $B$ is an $f$-quasi-filtered graded
$K$-algebra,
\item $\displaystyle\limsup_{n\rightarrow+\infty}\sup_{0\neq a\in B_n}\frac{\lambda(a)}{n}<+\infty$.
\end{enumerate}
For any integer $n>0$, we denote by
$\nu_n=T_{\frac 1n}\nu_{B_n}$. Then
\begin{enumerate}[1)]
\item $\displaystyle\lim_{n\rightarrow+\infty}
\min_{a\in B_n}\frac{\lambda(a)}{n}$ exists in $\mathbb R$,
\item the
supports of $\nu_n$ ($n\ge 1$) are uniformly
bounded and the sequence of measures
$(\nu_n)_{\ge 1}$ converges vaguely to a
Borel probability measure on $\mathbb R$.
\end{enumerate}
\end{theorem}
\begin{proof} For any integer $n\ge 1$, let
\[\lambda^{\max}_n=\sup_{0\neq a\in B_n}
\lambda(a)\text{ and
}\lambda^{\min}_n=\min_{a\in
B_n}\lambda(a).\] The support of $\nu_n$ is
contained in
$[\lambda^{\min}_n/n,\lambda^{\max}_n/n]$.
Since $0<d<+\infty$, for any integer $n>0$,
$B_n$ is a non-zero vector space of finite
rank, so $\lambda^{\min}_n\in\mathbb R$ since
the filtration of $B_n$ is exhaustive. On the
other hand, there exists an element $a_n$ in
$B_n$ such that
$\lambda^{\min}_n=\lambda(a_n)$. Let
\[W_n=\{b_1\cdots b_n\;|\;b_i\in
B_1\text{ for any }1\le i\le n\}.\] Since $B$
is generated by $B_1$, $B_n$ is generated as
vector space over $K$ (even as commutative
group) by $W_n$. After Proposition
\ref{Pro:proprietes de la fonction lambda}
2), we may suppose $a_n\in W_n$. Clearly
$a_n$ is non-zero for any integer $n>1$. If
$\mathbf{n}=(n_i)_{1\le i\le r}$ is a
multi-index in $\mathbb Z_{>0}^r$ and if
$N=n_1+\cdots+n_r$, we can write $a_N$ as the
product of $r$ elements $c_1,\cdots,c_r$,
where $c_i\in B_{n_i}\setminus\{0\}$.
Therefore,
\begin{equation}\label{Equ:estimation de lambda min}\lambda^{\min}_N=\lambda(a_N)
\ge\sum_{i=1}^r\Big(\lambda(c_i)-f(n_i)\Big)\ge\sum_{i=1}^r\Big(\lambda^{\min}_{n_i}-f(n_i)\Big).
\end{equation} The condition iii) implies that
$\displaystyle\limsup_{n\rightarrow+\infty}{\lambda^{\min}_n}/{n}<+\infty$,
so the sequence $(\lambda^{\min}_n/n)_{n\ge
1}$ has a limit in $\mathbb R$ (by Corollary
\ref{Cor:limite de an sur n forme forte}) and
therefore is bounded from below. On the other
hand, the condition iii) shows that the
sequence $(\lambda^{\max}_n/n)_{n\ge 1}$ is
bounded from above. Hence the supports of
measures $\nu_{n}$ ($n\ge 1$) are uniformly
bounded.

We now prove the second assertion of the
theorem. After Lemma \ref{Lem:convergence
vague vers une proba}, it suffices to verify
$\mathbf{CV}(B)$.

\vskip 5 pt \noindent{\it Step} $1$: some
simplifications.

First, after possible extension of fields, we
may suppose that $K$ is infinite.

Let $c$ be a real constant. We consider the
filtration $\mathcal F^c$ of $B$ such that
$\mathcal F^c_tB_n=\mathcal F_{t-cn}B_n$. In
other words, for any element $a\in B_n$, we
have the equality $\lambda_{\mathcal
F^c}(a)=\lambda_{\mathcal F}(a)+cn$. If
$(n_i)_{1\le i\le r}\in\mathbb Z_{>0}^r$ is
an multi-index and if for any $i$, $a_i$ is
an element in $B_{n_i}$, in writing
$N=n_1+\cdots+n_r$, $a=a_1\cdots a_r$, we
have
\[\lambda_{\mathcal F^c}(a)=\lambda_{\mathcal F}(a)+cN\ge
\sum_{i=1}^r\Big(\lambda_{\mathcal
F}(a_i)-f(n_i)\Big)+\sum_{i=1}^r cn_i=
\sum_{i=1}^r\Big(\lambda_{\mathcal
F^c}(a_i)-f(n_i)\Big),\] in other words, $B$
is $f$-quasi-filtered for the filtration
$\mathcal F^c$. On the other hand, if we
denote by $\nu_{B_n}^c$ the probability
measure associated to $B_n$ for the
filtration $\mathcal F^c$, we have
$\nu_{B_n}^c=\tau_{cn}\nu_{B_n}$. Therefore,
$T_{\frac 1n}\nu_{B_n}^c=T_{\frac
1n}\tau_{cn}\nu_{B_n}=\tau_{c}T_{\frac
1n}\nu_{B_n}$. Hence $B$ satisfies the vague
convergence condition for the filtration
$\mathcal F$ if and only if it is the case
for the filtration $\mathcal F^{c}$. After
the proof of the first assertion we obtain
$\lambda^{\min}_n=O(n)$. Since $f(n)=o(n)$,
we have $\lambda^{\min}_n-f(n)=O(n)$. In
replacing the filtration $\mathcal F$ by
$\mathcal F^c$, where $c\in\mathbb R_{>0}$ is
sufficiently large, we reduce the problem to
the case where $\lambda^{\min}_n-f(n)\ge 0$
for any $n\ge 1$. In particular, for any
homogeneous element $a$ of $B$ of homogeneous
degree $n$, we have $\lambda(a)-f(n)\ge 0$.

\vskip 5 pt \noindent{\it Step} $2$: Since
$K$ is an infinite field, by Noether's
normalization (cf. \cite{Eise} Theorem 13.3),
there exist $d$ elements $x_1,\cdots,x_d$ in
$B_1$ such that
\begin{enumerate}[1)]
\item the homomorphism from the polynomial algebra
$K[T_1,\cdots,T_d]$ to $B$, which sends $T_i$
to $x_i$, is an isomorphism of graded
$K$-algebras from $K[T_1,\cdots,T_d]$ to its
image,
\item if we denote by $A$ this image,
i.e., the sub-$K$-algebra of $B$ generated by
$x_1,\cdots,x_d$, then $B$ is a graded $
A$-module of finite type.
\end{enumerate}
The sub-$K$-algebra $A$, equipped with the
inverse image filtrations, is an
$f$-quasi-filtered graded $K$-algebra.
Moreover, $B$ is an $f$-quasi-filtered graded
$A$-module. Proposition \ref{Pro:vague
convergnce de mesures} shows that we have
$\mathbf{CV}(A)$.

Let $a$ be a non-zero homogeneous element of
$A$. We equip $Aa$ with the inverse image
filtration. Since $\dim A/Aa<\dim A$, we have
$\mathbf{CV}(Aa)$ after Lemma \ref{Lem:suite
exacte et la comdtion de convergence vague}.
Furthermore, the sequences $(T_{\frac
1n}\nu_{A_n})_{n\ge 1}$ and $(T_{\frac
1n}\nu_{(Aa)_n})_{n\ge 1}$ of probability
measures on $\mathbb R$ converge vaguely to
the same probability measure on $\mathbb R$.

If $x$ is a homogeneous element of degree
$m>0$ in $B$, then there exists a unitary
polynomial $P\in A[X]$ of degree $p\ge 1$
such that $P(x)=0$. We may suppose that $P$
is minimal and is written in the form
\[P(X)=X^p+a_{p-1}X^{p-1}+\cdots+a_0.\] Since
$P$ is minimal and since $B$ is an integral
ring, $a_0$ is non-zero. For any integer
$0\le i<p$, let $\widetilde{a}_i$ be the
component of degree $(p-i)m$ of $a_i$. If we
write
\[\widetilde P(X)=X^p+\widetilde
{a}_{p-1}X^{p-1}+\cdots+\widetilde{a}_0,\]
then we still have $\widetilde{P}(x)=0$ since
$x$ is homogeneous of degree $m$. Therefore
we can suppose that $a_i$ is homogeneous of
degree $(p-i)m$ for any $0\le i<p$. Let
$y=x^{p-1}+a_{p-1}x^{p-2}+\cdots+a_1$, which
is homogeneous of degree $(p-1)m$. Moreover
we have $xy+a_0=0$. If $u$ is a homogeneous
element of degree $n$ of $A$, then
\begin{equation}\label{Equ:quasi-fitlree 1}\lambda(ua_0)=\lambda(uxy)\ge\lambda(ux)
-f(n+m)+\lambda(y)-f((p-1)m)\ge
\lambda(ux)-f(n+m).\end{equation} We then
deduce that
$\lambda(ux)\le\lambda(ua_0)+f(n+m)$. On the
other hand,
\begin{equation}\label{Equ:quasi-fitlree 2}\lambda(ux)\ge\lambda(u)+\lambda(x)-f(m)-f(n)\ge\lambda(u)-f(n).
\end{equation} Let $M=Aa_0$ and $M'=Ax$.
The algebra $B$ being integral, for any
integer $n\ge 1$, the mapping $ux\mapsto
ua_0$ ($u\in A_n$) from $M_{n+m}'$ to
$M_{n+mp}$ is an isomorphism of vector spaces
over $K$. After \eqref{Equ:quasi-fitlree 1}
and Lemma \ref{Lem:comparaison des mesures
par un iso}, we have $\nu_{M_{n+m}'}\prec
\tau_{f(n+m)}\nu_{M_{n+mp}}$. On the other
hand, the mapping $u\mapsto ux$ ($u\in A_n$)
from $A_n$ to $M'_{n+m}$ is an isomorphism of
vector spaces over $K$. After
\eqref{Equ:quasi-fitlree 2} and Lemma
\ref{Lem:comparaison des mesures par un iso},
we obtain
$\nu_{A_n}\prec\tau_{f(n)}\nu_{M'_{n+m}}$, or
equivalently $\tau_{-f(n)}\nu_{A_n}\prec
\nu_{M'_{n+m}}$. So we have the estimation
\[\tau_{-f(n)}\nu_{A_n}\prec
\nu_{M'_{n+m}}\prec\tau_{f(n+m)}\nu_{M_{n+mp}},\]
and hence
\begin{equation*}T_{\frac 1{n+m}}\tau_{-f(n)}\nu_{A_n}\prec
T_{\frac 1{n+m}}\nu_{M'_{n+m}}\prec T_{\frac
1{n+m}}\tau_{f(n+m)}\nu_{M_{n+mp}},
\end{equation*}
or equivalently
\begin{equation}\label{Equ:encadrement de mesures}
\tau_{-f(n)/(n+m)}T_{\frac{n}{n+m}}T_{\frac
1n}\nu_{A_n}\prec T_{\frac
1{n+m}}\nu_{M'_{n+m}}\prec
\tau_{f(n+m)/(n+m)}T_{\frac{n+mp}{n+m}}T_{\frac
1{n+mp}}\nu_{M_{n+mp}}.
\end{equation}
As proved above, the sequences $(T_{\frac
1n}\nu_{A_n})_{n\ge 1}$ and $(T_{\frac
1n}\nu_{M_n})_{n\ge 1}$ converge vaguely to
the same limit $\nu$. After Proposition
\ref{Pro:convergence vague par dilatation et
translation} and the estimation
(\ref{Equ:encadrement de mesures}), we
conclude that the sequence $(T_{\frac
1{n}}\nu_{M'_n})_{n\ge 1}$ converges vaguely
to $\nu$.

\vskip 10 pt\noindent{\it Step }$3$: Since
$B$ is a finite algebra over $A$, the algebra
$L\otimes_AB$ is of finite rank over $L$,
where $L$ is the quotient field of $A$. The
$A$-module $B$ is generated by homogeneous
elements, hence there exist homogeneous
elements $x_1,\cdots,x_s$ of $B$ forming a
base of $L\otimes_AB$ over $L$. If we write
$H=Ax_1+\cdots+Ax_s$, then $H$ is a free
sub-$A$-module of base $(x_1,\cdots,x_s)$ of
$B$. Let $H'=B/H$. We have an exact sequence:
\[\xymatrix{0\ar[r]&H\ar[r]^\psi&B\ar[r]^\pi&H'\ar[r]&0}.\]
Since $1\otimes\psi:L\otimes_AH\rightarrow
L\otimes_AB$ is an isomorphism, we have
$L\otimes_AH'=0$, so $H'$ is a torsion
$A$-module. Then $\dim_AH'<\dim
A=\dim_AH=\dim_AB$. After the step $2$ we
have $\mathbf{CV}(Ax_i)$ for any $1\le i\le
s$. After Lemma \ref{Lem:suite exacte et la
comdtion de convergence vague}, we obtain
$\mathbf{CV}(H)$ and hence $\mathbf{CV}(B)$.
\end{proof}

\begin{remark}\label{Rem:convergence vague de
algebre gradue filtree} In Theorem
\ref{Thm:convergence vague de algebre gradue
filtree}, if we suppose that the vector space
$B_n$ is non-zero for sufficiently large $n$
(this condition is notably satisfied when
$B_1\neq 0$), then the condition that $B$ is
generated by $B_1$ is not necessary. In fact,
after \cite{EGAII} II, 2.1.6, there exists an
integer $d>0$ such that
$B^{(d)}=\bigoplus_{n\ge 0}B_{nd}$ is a
$B_0$-algebra generated by $B^{(d)}_1=B_d$.
Moreover, if we denote by $g:\mathbb Z_{\ge
0}\rightarrow\mathbb R_{\ge 0}$ the mapping
such that $g(n)=f(nd)$, then $B^{(d)}$ is a
$g$-quasi-filtered $K$-algebra. After Theorem
\ref{Thm:convergence vague de algebre gradue
filtree}, the algebra $B^{(d)}$ satisfies the
vague convergence condition. Hence by an
argument similar to the second step of the
proof of Theorem \ref{Thm:convergence vague
de algebre gradue filtree}, for any non-zero
homogeneous element $x$ of $B$, $B^{(d)}x$
satisfies the vague convergence condition,
and the sequence of probability measures
associated to $B^{(d)}x$ converges to the
limit of that associated to $B^{(d)}$. We
suppose that $B_n\neq 0$ for any $n\ge m_0$.
Then for any integer $m_0\le k<m_0+d$, the
$B^{(d)}$-module $B^{(d,k)}=\bigoplus_{n\ge
0}B_{nd+k}$ is non-zero. By an argument
similar to the third step of the proof of
Theorem \ref{Thm:convergence vague de algebre
gradue filtree} using the fact that $B^{(d)}$
is an integral ring, we conclude that
$B^{(d,k)}$ satisfies the vague convergence
condition, and that the limit of the sequence
of probability measures associated to
$B^{(d,k)}$ coincides with that of
probability measures associated to $B^{(d)}$.
Finally, combining all these measure
sequences, Proposition \ref{Pro:convergence
vague par dilatation et translation} shows
that the sequence of probability measures
associated to $B$ converges vaguely.
\end{remark}

\begin{remark}\label{Rem:contreexample}
\begin{enumerate}[1)]
\item Theorem
\ref{Thm:convergence vague de algebre gradue
filtree} is not true in general for a
quasi-filtered graded module. In fact, let
$B$ be the algebra $K[X]$ of polynomials in
one variable, equipped with the usual
graduation and with the filtration $\mathcal
F$ such that
\[\mathcal F_t
B_n=\begin{cases} B_n,&t\le 0,\\
0,&t>0.
\end{cases}\]
Clearly $B$ is a quasi-filtered graded
$K$-algebra. Let $M$ be a free graded
$B$-module generated by one homogeneous
element of degree $0$. If $\varphi:\mathbb
Z_{\ge 0}\rightarrow\mathbb R$ is an
increasing function, we can define a
filtration $\mathcal F^\varphi$ of $M$ such
that
\[\mathcal F_t^\varphi M_n=\begin{cases}
M_n,&t\le\varphi(n),\\
0,&t>\varphi(n).
\end{cases}\]
Then $M$ is a quasi-filtered graded
$B$-module, and for any integer $n\ge 0$,
$\nu_{M_n}=\delta_{\varphi(n)}$. Notice that
the condition $\mathbf{CV}(M)$ is equivalent
to the assertion that
$\displaystyle\lim_{n\rightarrow+\infty}
\varphi(n)/n$ exists in $\mathbb
R\cup\{+\infty\}$. If $\varphi:\mathbb Z_{\ge
0}\rightarrow\mathbb R$ is an increasing
function such that the sequence
$(\varphi(n)/n)_{n\ge 1}$ has more than one
accumulation point
--- for example, if $\varphi(n)=2^{\lfloor
\log_2n\rfloor}$, then $\mathbf{CV}(M)$ is no
longer satisfied. This counter-example shows
that it is not possible to prove Theorem
\ref{Thm:convergence vague de algebre gradue
filtree} by using the classical version of
unscrewing.
\item Theorem
\ref{Thm:convergence vague de algebre gradue
filtree} is not true in general for a
quasi-filtered graded algebra which is not
integral. In fact, if $B$ is a quasi-filtered
algebra over $K$ and if $M$ is a
quasi-filtered graded $B$-module. We suppose
that $\mathbf{CV}(B)$ is satisfied, but the
condition $\mathbf{CV}(M)$ is {\bf not}
satisfied (after 1), this is always
possible). If we denote by $C$ the nilpotent
extension of $B$ by $M$ (see \cite{Matsumura}
chap. 9 \S25), then $C$ is a filtered graded
algebra over $K$, which is of finite type.
But the condition $\mathbf{CV}(C)$ is not
satisfied.
\end{enumerate}
\end{remark}

\begin{corollary}\label{Cor:convergence de polygone cas general}
With the notations of Theorem
\ref{Thm:convergence vague de algebre gradue
filtree}, if for any $n\in\mathbb N$, we
denote by $P_n$ the polygon associated to the
probability measure $\nu_n$, then the
sequence of polygons $(P_n)_{n\ge 1}$
converges uniformly to a concave function on
$[0,1]$. If $B_n\neq 0$ for sufficiently
large $n$, the same result remains true if we
remove the condition that $B$ is generated as
$K$-algebra by $B_1$.
\end{corollary}

\section{Convergence of Harder-Narasimhan polygons: relative geometric case}

\hskip\parindent Using the results
established in the previous section, notably
Theorem \ref{Thm:convergence vague de algebre
gradue filtree} and Remark
\ref{Rem:convergence vague de algebre gradue
filtree}, we obtain in Theorem
\ref{Thm:convergence for quasifilted algebra
geometric} the convergence of normalized
Harder-Narasimhan polygons for an algebra in
vector bundles on a non-singular projective
curve.

Let $k$ be a field, $C$ be a non-singular
projective curve of genus $g$ over $k$,
$\eta$ be the generic point of $C$ and $K$ be
the field of rational functions on $C$. As
explained in the introduction, we shall
associate to each non-zero vector bundle $E$
on $C$ an $\mathbb R$-filtration of $E_K$
which is separated, exhaustive and left
continuous. Let
\[0=E_0\subsetneq E_1\subsetneq E_2\subsetneq\cdots\subsetneq E_n=E\]
be the Harder-Narasimhan flag of $E$, which
induces a flag
\[0=E_{0,K}\subsetneq E_{1,K}\subsetneq E_{2,K}
\subsetneq\cdots\subsetneq E_{n,K}=E_K\] of
the vector space $E_K$. Furthermore, if we
write $\mu_i=\mu(E_i/E_{i-1})$ for $1\le i\le
n$, then the sequence of rational numbers
$(\mu_i)_{1\le i\le n}$ is strictly
decreasing.

Therefore, we obtain a filtration $\mathcal
F^{\HN}$ of $E_K$ such that
\[\mathcal F_s^{\HN}E_K=\begin{cases}
0,&\text{if }s>\mu_1,\\
E_{i,K},&\text{if }\mu_{i+1}<s\le\mu_i,\quad
1\le i\le n,\\
E_K&\text{if }s\le \mu_n,
\end{cases}\]
called the {\it Harder-Narasimhan filtration}
of $E_K$. Note that the normalized
Harder-Narasimhan polygon of $E$ identifies
with the polygon associated to the
Harder-Narasimhan filtration of $E_K$.

We recall that if $\varphi:F\rightarrow G$ is
a non-zero homomorphism of vector bundles on
$C$, then the inequality
$\mu_{\min}(F)\le\mu(\varphi(F))\le\mu_{\max}(G)$
holds. We obtain therefore the following
proposition.
\begin{proposition}\label{Pro:fonctorial de HN}
Let $\varphi:F\rightarrow E$ be a
homomorphism of vector bundles on $C$. For
any real number $s$, the image
$\varphi_K(F_K)$ is contained in $\mathcal
F_s^{\HN}E_K$ if $\mu_{\min}(F)\ge s$.
\end{proposition}
\begin{proof}
The case where $\varphi=0$ is trivial. We
assume hence $\varphi\neq 0$. First, for any
real number $s\in\mathbb R$, $\mathcal
F_s^{\HN}E_K\in\{E_{0,K},\cdots,E_{n,K}\}$.
Since the vector bundles $E_i$ are saturated
in $E$, $\varphi_K(F_K)\subset E_{i,K}$ if
and only if $\varphi(F)\subset E_i$.
Therefore, if $i$ is the smallest index such
that $\varphi_K(F_K)\subset E_{i,K}$, which
is always $\ge 1$ because $\varphi\neq 0$,
then
$\mu_{\min}(F)\le\mu_{\max}(E_i/E_{i-1})=\mu_i$
since the composed homomorphism
$\xymatrix{F\ar[r]^\varphi&E_i\ar[r]&E_i/E_{i-1}}$
is non-zero. Therefore we have $s\le\mu_i$,
so $\varphi_K(F_K)\subset
E_{i,K}\subset\mathcal F_s^{\HN}E_K$.
\end{proof}

Proposition \ref{Pro:fonctorial de HN}
implies in particular that, for any subbundle
$F\subset E$ such that $\mu_{\min}(F)\ge s$,
$F_K$ is contained in $\mathcal
F^{\HN}_sE_K$. Therefore we have
\[\mathcal F^{\HN}_sE_K=\sum_{\begin{subarray}{c}
0\neq F\subset E\\
\mu_{\min}(F)\ge s
\end{subarray}}F_K.\]

\begin{corollary}\label{Cor:homomorphism compatible
aux filtration de HN} Let
$\varphi:F\rightarrow E$ be a homomorphism of
vector bundles on $C$. For any real number
$s$, the $K$-linear mapping $\varphi_K$ sends
$\mathcal F^{\HN}_sF_K$ into $\mathcal
F^{\HN}_sE_K$. In other words, the
homomorphism $\varphi_K$ is compatible with
Harder-Narasimhan filtrations.
\end{corollary}
\begin{proof}
Let $F_s$ be the saturated subbundle of $F$
such that $F_{s,K}=\mathcal F^{\HN}_sF_K$. By
the definition of Harder-Narasimhan
filtrations, we know that $\mu_{\min}(F_s)\ge
s$ once $\mathcal F_s^{\HN}F_K$ is non-zero.
Therefore, the canonical mapping from
$\mathcal F^{\HN}_sF_K$ to $E_K$ factorizes
through $\mathcal F^{\HN}_sE_K$.
\end{proof}

In the following, we shall introduce some
easy estimations for the maximal and the
minimal slope of the tensor product of vector
bundles on $C$, which will be useful in
Proposition \ref{Pro:toute algebre est quasi
filtree}.

\begin{lemma}\label{Lem:majoration de degre}
Let $E$ be a non-zero vector bundle on $C$.
If $H^0(C,E)$ reduces to zero, then
$\mu_{\max}(E)\le g-1$.
\end{lemma}
\begin{proof}
As $H^{0}(C,E)=0$, for any subbundle $F$ of
$E$, we have $H^0(C,F)=0$. After Riemann-Roch
theorem, we have
$\rang_kH^0(C,F)-\rang_kH^1(C,F)=\deg(F)+\rang(F)(g-1)$.
If $H^0(C,F)=0$, then
$\deg(F)+\rang(F)(1-g)\le 0$, i.e. $\mu(F)\le
g-1$.
\end{proof}

Let
$b(C)=\min\{\deg(H)\;|\;H\in\Pic(C),\;H\text{
is ample}\}$. It is a strictly positif
integer, and the set of values
$\{\deg(H)\;|\;H\in\Pic(C)\}$ is exactly
$b(C)\mathbb Z$. We define $a(C)=b(C)+g$.

\begin{proposition}\label{Pro:grand sous fibre inversible}
For any non-zero vector bundle $E$ on $C$,
there exists a line subbundle $L$ of $E$ such
that $\deg(L)\ge\mu_{\max}(E)-a(C)$.
\end{proposition}
\begin{proof} Let
$M$ be a line bundle of degree $b(C)$ on $C$.
We write
$r=\lceil(g-\mu_{\max}(E))/b(C)\rceil$. Thus
\[\frac{g-\mu_{\max}(E)}{b(C)}\le r<\frac{g-\mu_{\max}(E)+b(C)}{b(C)}.\]
Therefore $\mu_{\max}(E\otimes M^{\otimes
r})= \mu_{\max}(E)+rb(C)\ge g$. After Lemma
\ref{Lem:majoration de degre}, we obtain
$H^0(C,E\otimes M^{\otimes r})\neq 0$. So
there exists an injective homomorphism from
$\mathcal O_C$ to $E\otimes M^{\otimes r}$.
Let $L=M^{\vee\otimes r}$. Then $L$ is a
subbundle of $E$. On the other hand, we have
$\deg(L)=-r\deg(M)=-rb(C)>\mu_{\max}(E)-g-b(C)$.
Since $a(C)=b(C)+g$, we obtain
$\deg(L)\ge\mu_{\max}(E)-a(C)$.
\end{proof}

\begin{proposition}\label{Pro:estimation de mu max
et mu min} If $E_1$ and $E_2$ are two
non-zero vector bundles on $C$, then
\begin{enumerate}[1)]
\item $\mu_{\max}(E_1\otimes E_2)<\mu_{\max}
(E_1)+\mu_{\max}(E_2)+a(C)$;
\item $\mu_{\min}(E_1\otimes E_2)>
\mu_{\min}(E_1)+\mu_{\min}(E_2)-a(C)$.
\end{enumerate}
\end{proposition}
\begin{proof}
1) First we prove that if
$\mu_{\max}(E_1)+\mu_{\max}(E_2)<0$, then
$\mu_{\max}(E_1\otimes E_2)<g$. In fact, if
$\mu_{\max}(E_1\otimes E_2)\ge g$, then
$H^0(C,E_1\otimes E_2)\neq 0$ (see Lemma
\ref{Lem:majoration de degre}). Therefore,
there exists a non-zero homomorphism from
$E_1^\vee$ to $E_2$, which implies that
\[\mu_{\max}(E_2)\ge\mu_{\min}(E_1^\vee)=-\mu_{\max}(E_1),\]
i.e., $\mu_{\max}(E_1)+\mu_{\max}(E_2)\ge 0$.
To prove 1), we take a line bundle $L$ on $C$
such that
$-b(C)\le\mu_{\max}(E_1)+\mu_{\max}(E_2)+\deg(L)<0$.
We then have $\mu_{\max}(E_1\otimes
L)+\mu_{\max}(E_2)<0$ and hence, after the
result established above,
$\mu_{\max}(E_1\otimes L\otimes E_2)<g$.
Therefore,
\[\mu_{\max}(E_1\otimes E_2)<g-\deg(L)\le
\mu_{\max}(E_1)+\mu_{\max}(E_2)+g+b(C)=
\mu_{\max}(E_1)+\mu_{\max}(E_2)+a(C).\]

2) In fact,
\[\begin{split}
&\quad\;\mu_{\min}(E_1\otimes E_2)=-\mu_{\max}((E_1\otimes E_2)^\vee)=-\mu_{\max}(E_1^\vee\otimes E_2^\vee)\\
&>
-\Big(\mu_{\max}(E_1^\vee)+\mu_{\max}(E_2^\vee)+a(C)\Big)=
\mu_{\min}(E_1)+\mu_{\min}(E_2)-a(C).
\end{split}\]
\end{proof}

From Proposition \ref{Pro:estimation de mu
max et mu min}, we obtain by induction that
if $(E_i)_{1\le i\le r}$ is a family of
non-zero vector bundles on $C$, we have the
estimation
\[\mu_{\min}(E_1\otimes\cdots\otimes E_r)\ge
\sum_{i=1}^r\mu_{\min}(E_i)-a(C)(r-1)\ge\sum_{i=1}^r\mu_{\min}
(E_i)-a(C)r.\] Actually, if the field $k$ is
of characteristic $0$, then we have even the
equality $\mu_{\min}(E_1\otimes\cdots\otimes
E_r)=\mu_{\min}(E_1)+\cdots+\mu_{\min}(E_r)$.
This is a consequence of Ramanan and
Ramanathan's result \cite{Ramanan_Ramanathan}
asserting that the tensor product of two
semistable vector bundles on $C$ is
semistable.

\begin{proposition}
\label{Pro:toute algebre est quasi filtree}
Let $f:\mathbb Z_{\ge 0}\rightarrow\mathbb
R_{\ge 0}$ be the constant function which
sends any $n\in\mathbb Z_{\ge 0}$ to $a(C)$.
Let $\mathscr B=\bigoplus_{n\ge 0}\mathscr
B_n$ be a graded quasi-coherent $\mathcal
O_C$-algebra. Suppose that for any integer
$n\ge 0$, $\mathscr B_n$ is a vector bundle
over $C$, and we denote by $B_n=\mathscr
B_{n,K}$. Then $B=\bigoplus_{n\ge 0}B_n$,
equipped with Harder-Narasimhan filtrations,
is an $f$-quasi-filtered graded $K$-algebra.
\end{proposition}
\begin{proof}
For any integer $n\ge 0$ and any real number
$s$, let $\mathscr B_{n,s}$ be the saturated
subbundle of $\mathscr B_n$ such that
$\mathscr B_{n,s,K}=\mathcal F^{\HN}_sB_n$.
Since $\mathscr B$ is an $\mathcal
O_C$-algebra, for any integer $r\ge 2$ and
any element $(n_i)_{1\le i\le r}\in\mathbb
Z_{\ge 0}^{r}$, we have a natural
homomorphism $\varphi$ from $\mathscr
B_{n_1}\otimes\cdots\otimes\mathscr B_{n_r}$
to $\mathscr B_{N}$, where
$N=n_1+\cdots+n_r$. If $(t_i)_{1\le i\le r}$
is a family of real numbers, the homomorphism
$\varphi$ induces by restriction a
homomorphism $\psi$ from $\mathscr
B_{n_1,t_1}\otimes\cdots\otimes\mathscr
B_{n_r,t_r}$ to $\mathscr B_{N}$. By the
definition of Harder-Narasimhan filtration we
obtain that if $\mathscr B_{n_i,t_i}$ is
non-zero, then $\mu_{\min}(\mathscr
B_{n_i,t_i})\ge t_i$. Therefore, by using the
convention $\mu_{\min}(0)=+\infty$, we have
$\mu_{\min}(\mathscr
B_{n_1,t_1}\otimes\cdots\otimes\mathscr
B_{n_r,t_r})\ge t_1+\cdots+t_r-a(C)r$. After
Corollary \ref{Cor:homomorphism compatible
aux filtration de HN}, $\psi_K$ is compatible
with Harder-Narasimhan filtrations, so
$\psi_K$ factorizes through $\mathcal
F_{t_1+\cdots+t_r-a(C)r} B_{N}$. Therefore,
$B$ is a quasi-filtered graded $K$-algebra.
\end{proof}

\begin{theorem}\label{Thm:convergence for quasifilted algebra
geometric} Let $\mathscr B=\bigoplus_{n\ge
0}\mathscr B_n$ be a quasi-coherent graded
$\mathcal O_C$-algebra. Suppose that the
following conditions are verified:
\begin{enumerate}[i)]
\item $\mathscr B_n$
is a vector bundle on $C$ for any integer
$n\ge 0$;
\item there exists a constant
$a>0$ such that $\mu_{\max}(\mathscr B_n)\le
an$ for any integer $n\ge 1$;
\item $\mathscr B_K$ is an integral ring which is of finite type over $K$
and $\mathscr B_n$ is non-zero for
sufficiently large integer $n$.
\end{enumerate}
For any integer $n\ge 1$, we denote by $P_n$
the Harder-Narasimhan polygon of $\mathscr
B_n$. Then
\begin{enumerate}[1)]
\item the sequence of numbers $(\frac 1n\mu_{\min}(\mathscr B_n))_{n\ge 1}$ has a limit in $\mathbb R$.
\item the sequence $(\frac 1nP_n)_{n\ge
1}$ converges uniformly on $[0,1]$.
\end{enumerate}
\end{theorem}
\begin{proof}
Let $f$ be the constant function $\mathbb
Z_{\ge 0}$ with value $a(C)$. After
Proposition \ref{Pro:toute algebre est quasi
filtree}, we obtain that $\mathscr B_K$
equipped with Harder-Narasimhan filtrations
is an $f$-quasi-filtered graded $K$-algebra.
The theorem is then proved by using Theorem \ref{Thm:convergence vague de algebre gradue filtree} (see also Remark \ref{Rem:convergence vague de
algebre gradue filtree}) and Corollary
\ref{Cor:convergence de polygone cas
general}.
\end{proof}

Let $\pi:X\rightarrow C$ be a projective and
flat morphism from an algebraic variety $X$
to $C$ and $L$ be a line bundle on $X$. We
shall apply Theorem \ref{Thm:convergence for
quasifilted algebra geometric} to the special
case where $\mathscr B$ is the direct sum of
the direct images by the morphism $\pi$ of
tensor powers of $L$.

\begin{lemma} \label{Lem:majoration de mu max par une croissance lineaire}
There exists a constant $\varepsilon$ such
that, for any integer $n> 0$,
\[\mu_{\max}(\pi_*(L^{\otimes n})) \le
\varepsilon n.\]
\end{lemma}
\begin{proof}
The variety $X$ is projective over $\Spec k$.
We can hence choose an ample line bundle
$\mathscr L$ on $X$.

Let $d=\dim X$. Observe that
$\pi_*(c_1(\mathscr L)^{d-1})=(\deg_{\mathscr
L_K}X_K)[C]$ in the Chow group $\CH_1(C)$.
Suppose that $M$ is a line bundle on $C$ and
that $\varphi:M\rightarrow\pi_*( \mathscr
L^{\otimes n})$ is an injective homomorphism.
We denote by $\widetilde
\varphi:\pi^*M\rightarrow\mathscr L^{\otimes n}$ the
homomorphism of $\mathcal O_X$-modules
corresponding to $\varphi$ by adjunction,
which identifies with a non-identically zero
section of $\pi^*M^\vee\otimes\mathscr
L^{\otimes n}$, whose divisor
$\Div\widetilde\varphi$ is effective. Then
$\deg_X\Big(c_1(\mathscr
L)^{d-1}[\Div(\widetilde\varphi)]\Big)\ge 0$.
On the other hand,
$[\Div\widetilde\varphi]=-\pi^*c_1(M)+nc_1(L)$
in $\CH^1(X)$. Hence
\[\begin{split}&\quad\;\deg_X\Big(c_1(\mathscr L)^{d-1}[\Div(\widetilde\varphi)]\Big)
=\deg_X\Big((-\pi^*c_1(M)+nc_1(L))c_1(\mathscr L)^{d-1}\Big)\\
&=-\deg_C\Big(c_1(M)\pi_*(c_1(\mathscr
L)^{d-1})\Big)+n\deg_X(c_1(L)c_1(\mathscr
L)^{d-1}).
\end{split}\]
Therefore, $\displaystyle\deg_C(M)\le
n\frac{\deg_X(c_1(L) c_1(\mathscr L)^{d-1})}
{\deg_{\mathscr L_K}X_K}$. Finally, using the
comparison established in Proposition
\ref{Pro:grand sous fibre inversible}, we
deduce the upper bound of
$\mu_{\max}(\pi_*(\mathscr L^{\otimes n}))$
by a linear function on $n$.
\end{proof}

\begin{theorem}\label{Thm:convergence of polygone geometric}
Suppose that $H^0(X_K,L_K^{\otimes n })\neq
0$ for sufficiently large integer $n$ and
that the graded algebra $\bigoplus_{n\ge
0}H^0(X_K,L_K^{\otimes n})$ is of finite type
over $K$ (this condition is satisfied notably
when $L_K$ is ample). For any integer $n\ge
1$, let $P_n$ be the Harder-Narasimhan
polygon of $\pi_*(L^{\otimes n})$. Then the
sequence of numbers $(\frac 1n\mu_{\min}(\pi_*(L^{\otimes n})))_{n\ge 1}$ has a limit in $\mathbb R$ and the sequence of polygons $(\frac 1nP_n)_{n\ge 1}$
converges uniformly on $[0,1]$.
\end{theorem}
\begin{proof}
After Lemma \ref{Lem:majoration de mu max par
une croissance lineaire},
$\mu_{\max}(\pi_*(L^{\otimes n}))=O(n)$
($n\rightarrow+\infty$). Therefore, the
algebra $\mathscr B:=\bigoplus_{n\ge
0}\pi_*(L^{\otimes n})$ verifies the
conditions in Theorem \ref{Thm:convergence
for quasifilted algebra geometric}.
\end{proof}

The convergence of polygons $(\frac 1n P_n)$
suggests that the sequence of (normalized)
maximal slopes $(\frac
1n\mu_{\max}(\pi_*(L^{\otimes n})))_{n\ge 1}$
converges. However, this is
not a formal consequence of Theorem
\ref{Thm:convergence of polygone geometric}.
In Proposition \ref{Pro:convergence geometric
de mu max et mu min}, we shall justify the
convergence of this sequence by using
the same generalization of Fekete's lemma for
almost super-additive sequences.

\begin{lemma}\label{Lem:minoration de mumax de produit tensoriel}
Let $E_1$ and $E_2$ be two vector bundles on
$X$. If $\pi_*(E_1)$ and $\pi_*(E_2)$ are
non-zero, then
\[\mu_{\max}(\pi_*(E_1\otimes E_2))\ge
\mu_{\max}(\pi_*E_1)+\mu_{\max}(\pi_*E_2)-2a(C),\]
where $a(C)$ is the constant in Proposition
\ref{Pro:grand sous fibre inversible}.
\end{lemma}
\begin{proof}
Since $\pi_*(E_1)$ and $\pi_*(E_2)$ are
non-zero, also is $\pi_*(E_1\otimes E_2)$.
After Proposition \ref{Pro:grand sous fibre
inversible}, there exist two line bundles
$M_1$ and $M_2$ on $C$ and two injective
homomorphisms $M_1\rightarrow\pi_*(E_1)$ and
$M_2\rightarrow\pi_*(E_2)$ such that $\deg
M_1\ge\mu_{\max}(\pi_*E_1)-a(C)$ and $\deg
M_2\ge\mu_{\max}(\pi_*E_2)-a(C)$. Since both
$M_1^\vee\otimes\pi_*(E_1)$ and
$M_2^\vee\otimes\pi_*(E_2)$ have global
sections which do not vanish everywhere on
$C$, then both $\pi^*(M_1)^\vee\otimes E_1$
and $\pi^*(M_2)^\vee\otimes E_2$ have global
sections which do not vanish everywhere on
$X$. Therefore, $H^0(X,\pi^*(M_1\otimes
M_2)^\vee\otimes (E_1\otimes E_2))=H^0(C,
(M_1\otimes M_2)^\vee\otimes\pi_*(E_1\otimes
E_2))\neq 0$. So we have
$0\le\mu_{\max}((M_1\otimes
M_2)^\vee\otimes\pi_*(E_1\otimes E_2))$, and
hence $\mu_{\max}(\pi_*(E_1\otimes E_2))\ge
\deg M_1+\deg M_2\ge
\mu_{\max}(\pi_*(E_1))+\mu_{\max}(\pi_*(E_2))-2a(C)$.
\end{proof}

\begin{proposition}\label{Pro:convergence geometric de mu max et mu min}
Let $\pi:X\rightarrow C$ be a projective and
flat morphism from an algebraic variety $X$
to $C$ and $L$ be a line bundle on $X$ verifying the conditions of Theorem \ref{Thm:convergence of polygone geometric}. Then
the sequence $(\frac 1n\mu_{\max}(\pi_*(L^{\otimes
n})))_{n\ge 1}$ has a limit in $\mathbb R$.
\end{proposition}
\begin{proof}
Denote by $a_n=\mu_{\max}(\pi_*(L^{\otimes
n}))$ for any integer $n\ge 1$. After Lemma
\ref{Lem:majoration de mu max par une
croissance lineaire}, there exists a constant
$\varepsilon>0$ such that $a_n\le \varepsilon
n$ for sufficiently large $n$. On the other
hand, Lemma \ref{Lem:minoration de mumax de
produit tensoriel} shows that $a_{m+n}\ge
a_m+a_n-2a(C)$ for all integers $m$ and $n$.
After Corollary \ref{Cor:suite croissance
lineaire la version renforcee}, the sequence
$(a_n/n)_{n\ge 1}$ has a limit in $\mathbb
R$.
\end{proof}

\section{Convergence of Harder-Narasimhan polygons: arithmetic case}

\hskip\parindent In this section, we
establish the analogue of the results in the
previous section in Arakelov geometry. Let
$K$ be a number field and $\mathcal O_K$ be
its integer ring. We denote by $\Sigma_f$ the
set of all finite places of $K$, which
coincides with the set of all closed points
in $\Spec\mathcal O_K$. Let $\Sigma_\infty$
be the set of all embeddings of $K$ into
$\mathbb C$. Suppose that $E$ is a projective
$\mathcal O_K$-module of finite type, then
for any finite place $\mathfrak p$ of $K$,
the structure of $\mathcal O_K$-module on $E$
induces an ultranorm on $E_{K_{\mathfrak p}
}:=E\otimes_{\mathcal O_K}K_{\mathfrak p}$.

We have explained that to any non-zero
Hermitian vector bundle $\overline E$ on
$\Spec\mathcal O_K$, we can associated a flag
$0=E_0\subsetneq
E_1\subsetneq\cdots\subsetneq E_n=E$ of $E$
such that $\overline E_i/\overline E_{i-1}$
is semistable for any integer $1\le i\le n $
and that
\[\widehat{\mu}(\overline E_1/\overline E_0)>
\widehat{\mu}(\overline E_2/\overline
E_1)>\cdots>\widehat{\mu}(\overline
E_n/\overline E_{n-1}).\] If we write
$\mu_i=\widehat{\mu}(\overline E_i/\overline
E_{i-1} )$ for $1\le i\le n$, then the
sequence of real numbers $(\mu_i)_{1\le i\le
n }$ is strictly decreasing. Furthermore, the
flag of $E$ above induces a flag
$0=E_{0,K}\subsetneq E_{1,K}\subsetneq\cdots
\subsetneq E_{n,K}=E_K$ of the vector space
$E_K$. We obtain therefore a filtration
$\mathcal F^{\HN}$ of $E_K$ such that
\[\mathcal F_{s}^{\HN}E_K=\begin{cases}
0,&\text{if }s>\mu_1,\\
E_{i,K},&\text{if }\mu_{i+1}<s\le\mu_i,\quad
1\le i\le n,\\
E_K,&\text{if }s\le \mu_n,
\end{cases}\]
called the {\it Harder-Narasimhan filtration}
of $E_K$. Notice that the normalized
Harder-Narasimhan polygon of $\overline E$
coincides with the polygon associated to the
Harder-Narasimhan filtration of $E_K$.

Let $\overline F$ and $\overline G$ be two
non-zero Hermitian vector bundles on
$\Spec\mathcal O_K$ and
$\varphi:F_K\rightarrow G_K$ be a linear
mapping. For any place $\mathfrak
p\in\Sigma_f$, $\varphi$ induces a linear
mapping $\varphi_{\mathfrak
p}:F_{K_{\mathfrak p}}\rightarrow
G_{K_{\mathfrak p}}$. For any embedding
$\sigma\in \Sigma_\infty$, $\varphi$ induces
a linear mapping
$\varphi_\sigma:F_{\sigma}\rightarrow
G_{\sigma}$. We define the {\it height} of
$\varphi$ to be
\[h(\varphi)=\frac{1}{[K:\mathbb Q]}\Big(\sum_{\mathfrak p
\in\Sigma_f}\log\|\varphi_{\mathfrak p}\|+
\sum_{\sigma\in\Sigma_\infty}\log\|\varphi_\sigma\|\Big).\]
Notice that if $\varphi$ comes from an
$\mathcal O_K$-linear homomorphism $\phi$,
i.e. $\varphi=\phi_K$, then for any
$\mathfrak p\in\Sigma_f$,
$\log\|\varphi_{\mathfrak p}\|\le 0$. We
recall the slope inequalities:
\begin{enumerate}[1)]
\item if $\varphi$ is injective, then $\widehat{\mu}_{\max}
(\overline
F)\le\widehat{\mu}_{\max}(\overline
G)+h(\varphi)$;
\item if $\varphi$ is surjective, then $\widehat{\mu}_{\min}
(\overline
F)\le\widehat{\mu}_{\min}(\overline
G)+h(\varphi)$;
\item if $\varphi$ is non-zero, then $\widehat{\mu}_{\min}(
\overline F)\le\widehat{\mu}_{\max}(\overline
G)+h(\varphi)$.
\end{enumerate}
For the proof of the first inequality, one
can consult \cite{Bost2001}. The second
inequality is obtained by applying the first
one on $\varphi^\vee:G_K^\vee\rightarrow
F_K^\vee$. Finally if we apply the first two
inequalities on the two homomorphisms in the
decomposition $F_K\twoheadrightarrow
\varphi(F_K)\hookrightarrow G_K$
respectively, we obtain the third inequality.
Using the seconde slope inequality, we obtain
the following proposition.

\begin{proposition}\label{Pro:fonctoriality of arithmeticcase}
Let $\overline F$ and $\overline E$ be two
Hermitian vector bundles. If
$\varphi:\overline F_K\rightarrow\overline
E_K$ is a $K$-linear homomorphism, then for
any real number
$s\le\widehat{\mu}_{\min}(\overline
F)-h(\varphi)$, the image $\varphi(F_K)$ is
contained in $\mathcal F_s^{\HN}E_K$.
\end{proposition}
\begin{proof}
The case where $\varphi=0$ is trivial.
Suppose that $\varphi\neq 0$. Let $i$ be the
smallest index such that $\varphi(F_K)\subset
E_{i,K}$, which is always $\ge 1$ since
$\varphi\neq 0$. Consider the composed
homomorphism $\xymatrix{\psi:\relax
F_K\ar[r]^-\varphi&E_{i,K}\ar[r]&(E_i/E_{i-1}})_K$,
which is non-zero. By slope inequality, $s\le
\widehat{\mu}_{\min}(\overline
F)\le\widehat{\mu}_{\max}(\overline
E_i/\overline E_{i-1})+h(\psi)\le
\mu_i+h(\varphi)$, or equivalently
$s-h(\varphi)\le\mu_i$. Therefore
$\varphi_K(F_K)\subset E_{i,K}\subset\mathcal
F_{s-h(\varphi)}^{\HN}E_K$.
\end{proof}

Proposition \ref{Pro:fonctoriality of
arithmeticcase} implies that, for any
Hermitian subbundle $\overline F$ of
$\overline E$ such that
$\widehat{\mu}_{\min}(\overline F)\ge s$,
$F_K$ is contained in $\mathcal
F_{s}^{\HN}E_K$ (the height of the inclusion
mapping $F_K\rightarrow E_K$ is bounded from
above by $0$). Therefore we obtain the
relation
\[\mathcal F_s^{\HN}E_K=\sum_{\begin{subarray}{c}
0\neq F\subset E\\
\widehat{\mu}_{\min}(\overline F)\ge s
\end{subarray}}F_K.\]

\begin{corollary}\label{Cor:fonctorialtie arithe}
Let $\overline F$ and $\overline E$ be two
Hermitian vector bundles on $\Spec\mathcal
O_K$ and $\varphi:F_K\rightarrow E_K$ be a
$K$-linear mapping. Then for any real number
$s$, $\varphi$ sends $\mathcal F^{\HN}_sF_K$
into $\mathcal F^{\HN}_{s-h(\varphi)}E_K$.
\end{corollary}
\begin{proof}
Let $F_s$ be the saturated subbundle of $F$ such that
$F_{s,K}=\mathcal F_s^{\HN}F_K$. By the definition of
Harder-Narasimhan filtrations we know that
$\widehat{\mu}_{\min}(F_s)\ge s$ if $\mathcal
F_s^{\HN}F_K$ is non-zero. Therefore, the canonical
mapping from $\mathcal F_s^{\HN}F_K$ to $E_K$
factorizes through $\mathcal
F_{s-h(\varphi)}^{\HN}E_K$.
\end{proof}

In Corollary \ref{Cor:fonctorialtie arithe},
if the homomorphism $\varphi$ is an
isomorphism, then
$\tau_{h(\varphi)}\nu_{\mathcal
F^{\HN},E_K}\succ\nu_{\mathcal F^{\HN},F_K}$.
Therefore, for any $t\in [0,1]$,
$P_{\overline F}(t)\le P_{\overline
E}(t)+h(\varphi)t$. In particular, if $E$ is
a non-zero vector bundle on $\Spec\mathcal
O_K$ and if
$h=(\|\cdot\|_\sigma)_{\sigma\in\Sigma_\infty}$
and
$h'=(\|\cdot\|_{\sigma}')_{\sigma\in\Sigma_\infty}$
are two Hermitian structures on $E$, then for
any $t\in[0,1]$,
\begin{equation}\label{Equ:diffference of polygons}\big|P_{(E,h)}(t)-P_{(E,h')}(t)\big|\le\frac{t}{[K:\mathbb
Q]} \sum_{\sigma\in\Sigma_\infty}\sup_{0\neq
x\in E_{\sigma,\mathbb C}
}\Big|\log\|x\|_\sigma-\log\|x\|_\sigma'\Big|\end{equation}

Let $(\overline{\mathscr B}_n)_{n\ge 0}$ be a
collection of non-zero Hermitian vector bundles on
$\Spec\mathcal O_K$. For any integer $n\ge 0$ and any
$s\in\mathbb R$, we denote by $\mathscr B_{n,s}$ the
saturated subbundle of $\mathscr B_n$ such that
$B_{n,s}:=\mathscr B_{n,s,K}=\mathscr F^{\HN}_s\mathscr
B_{n,K}$. Suppose that $B=\bigoplus_{n\ge 0}\mathscr
B_{n,K}$ is equipped with a structure of commutative
$\mathbb Z_{\ge 0}$-graded algebra over $K$. For any
integer $r\ge 2$ and any element
$\mathbf{n}=(n_i)_{1\le i\le r}\in\mathbb N^r$, we have
a homomorphism $\varphi_{\mathbf{n}}$ from $
B_{n_1}\otimes\cdots\otimes B_{n_r}$ to
$B_{|\mathbf{n}|}$ defined by the structure of algebra.
After \cite{Chen_pm}, if $(s_i)_{1\le i\le r}$ is an
element in $\mathbb R^r$, we obtain, by using the
convention $\widehat{\mu}_{\min}(0)=+\infty$,
\[\widehat{\mu}_{\min}(\overline{\mathscr
B}_{n_1,s_1}\otimes\cdots\otimes\overline{\mathscr
B}_{n_r,s_r}) \ge\sum_{i=1}^r
\widehat{\mu}_{\min}(\overline{\mathscr B
}_{n_i,s_i})-\sum_{i=1}^r
\log(\rang(B_{n_i,s_i}))\ge\sum_{i=1}^r\Big(s_i-\log(\rang(
B_{n_i})) \Big).\] If $\overline E$ is a
Hermitian vector subbundle of
$\overline{\mathscr B}_{|\mathbf{n}|}$ such
that $E_K$ coincides with the image of
$B_{n_1,s_1}\otimes\cdots\otimes B_{n_r,s_r}$
in $B_{|\mathbf{n}|}$, after the slope
inequality, we have
\[\widehat{\mu}_{\min}
(\overline E)\ge
\widehat{\mu}_{\min}(\overline{\mathscr
B}_{n_1,s_1}\otimes\cdots\otimes
\overline{\mathscr
B}_{n_r,s_r})-h(\varphi_{\mathbf n})\ge
\sum_{i=1}^r\Big(s_i-\log(\rang(
B_{n_i}))\Big)-h(\varphi_{\mathbf{n}}).
\]
Suppose that $g:\mathbb Z_{\ge
0}\rightarrow\mathbb R_{\ge 0}$ is a function
such that $h(\varphi_{\mathbf{n}})\le
g(n_1)+\cdots+g(n_r)$ for $n_i$ sufficiently
large. For any integer $n\ge 1$, let
$f(n)=g(n)+\log(\rang( B_n))$. Then $B$ is an
$f$-quasi-filtered graded $K$-algebra.

\begin{theorem}\label{Thm:convergence de polygone
arithmetique} For any integer $n\ge 0$,
denote by $P_n$ the Harder-Narasimhan polygon
of $\overline{\mathscr B}_n$. Suppose that
$\displaystyle\lim_{n\rightarrow+\infty}
f(n)/n=0$ and that the sequence
$(\widehat{\mu}_{\max}(\overline{\mathscr B
}_n)/n)_{n\ge 1}$ is bounded. If $B$ is an
integral $K$-algebra of finite type and if
$B_n\neq 0$ for sufficiently large $n$, then the sequence $(\widehat{\mu}_{\min}(\overline{\mathscr B}_n)/n)_{n\ge 1}$ has a limit in $\mathbb R$ and
the function sequence $(P_n/n)_{n\ge 1}$
converges uniformly on $[0,1]$.
\end{theorem}
\begin{proof}
In fact, $P_n$ coincides with the polygon
associated to the filtered space $B_n$. The
theorem results therefore from Theorem \ref{Thm:convergence vague de algebre gradue filtree} (see also Remark \ref{Rem:convergence vague de
algebre gradue filtree}) and Corollary
\ref{Cor:convergence de polygone cas
general}.
\end{proof}

In the following, we shall establish the analogue of
Theorem \ref{Thm:convergence of polygone geometric} in
Arakelov geometry. Let $\pi:\mathscr
X\rightarrow\Spec\mathcal O_K$ be a scheme of finite
type and flat over $\Spec\mathcal O_K$ such that
$\mathscr X_K$ is proper. Let $\overline{\mathscr L}$
be a Hermitian line bundle on $\mathscr X$. For any
integer $D\ge 0$, let $E_D$ be the projective $\mathcal
O_K$-module $\pi_*(\mathscr L^{\otimes D})$. Suppose
that $E_D\neq 0$ for sufficiently large $D$ and that
the algebra $B:=\bigoplus_{D\ge 0}E_{D,K}$ is of finite
type over $K$. Clearly $B$ is integral. We denote by
$\|\cdot\|_{\sigma,\sup}$ the norm on $E_{D,\sigma}$
such that $\|s\|_{\sigma,\sup}=\sup_{x\in \mathscr
X_\sigma(\mathbb C)}\|s_x\|_{\sigma}$ for any $s\in
E_{D,\sigma}=H^0(\mathscr X_{\sigma,\mathbb C},\mathscr
L _{\sigma,\mathbb C}^{\otimes D})$. In general this is
not a Hermitian norm. For any integer $D\ge 0$ and any
$\sigma\in\Sigma_\infty$, we choose a Hermitian norm
$\|\cdot\|_\sigma$ on $E_{D,\sigma}$ such that
\begin{equation}\label{Equ:distance between norm sup and norm herm}
\sup_{0\neq s\in E_{D,\sigma}
}\Big|\log\|s\|_{\sigma}-\log\|s\|_{\sigma,\sup}\Big|=O(\log
D ) \qquad(D\rightarrow+\infty).\end{equation} This is
always possible by Gromov's inequality in smooth metric
case (see \cite{Gillet-Soule} Lemma 30), or by John's
or L\"owner's ellipsoid argument in general case (see
\cite{Gaudron07}, \cite{Thompson96}). Suppose in
addition that the collection
$h_D=(\|\cdot\|_{\sigma})_{\sigma\in\Sigma_\infty}$ is
invariant by the complex conjugation. Then $\overline
E_D=(E_D,h_D)$ becomes a Hermitian vector bundle on
$\Spec\mathcal O_K$. For any integer $r\ge 2$ and any
element $\mathbf{n}=(n_i)_{1\le i\le r}\in\mathbb N^r$,
let $\varphi_{\mathbf{n}}$ be the canonical
homomorphism from $ E_{n_1,K}\otimes\cdots\otimes
E_{n_r,K}$ to $E_{|\mathbf{n}|,K}$. For any integer
$D\ge 1$ and any $\sigma\in\Sigma_D$, we denote by
\[A_{D,\sigma}=\displaystyle\sup_{0\neq s\in
E_{D,\sigma}
}\Big|\log\|s\|_{\sigma}-\log\|s\|_{\sigma,\sup}\Big|.\]
From \eqref{Equ:distance between norm sup and norm
herm} we know that there exists an integer $n_0\ge 2$
and a real number $\varepsilon >0$ such that
$A_{D,\sigma}\le\varepsilon\log D$ for any $D\ge n_0$
and any $\sigma\in\Sigma_\infty$.

\begin{lemma}\label{Lem:log croissance de la norem}
We have the following inequality
\begin{equation}\label{Equ:Majoration de la
norme de produit tensoriel}h(\varphi_{\mathbf{n}})\le
\frac{1}{[K:\mathbb Q]} \sum_{\sigma\in\Sigma_\infty}
A_{|\mathbf{n}|,\sigma}+\sum_{i=1}^r\Big(
A_{n_i,\sigma}+\frac12\log(\rang(E_{n_i}))\Big).\end{equation}
\end{lemma}
\begin{proof}
Since $\varphi_{\mathbf{n}}$ comes from a
homomorphism of $\mathcal O_K$-modules,
$\|\varphi_{\mathbf{n}}\|_{\mathfrak p}\le 1$
for any finite place $\mathfrak p$ of $K$.
Consider now an embedding
$\sigma\in\Sigma_\infty$. If $(s_i)_{1\le
i\le r}\in E_{n_1,\sigma}\times \cdots\times
E_{n_r,\sigma}$, then
\[\begin{split}&\quad\;\log\|s_1\cdots s_r\|_{\sigma}\le
\log\|s_1\cdots
s_r\|_{\sigma,\sup}+A_{|\mathbf{n}|,\sigma}
\le\sum_{i=1}^r\log\|s_i\|_{\sigma,\sup}
+A_{|\mathbf{n}|,\sigma}\\&\le\sum_{i=1}^r\Big(\log\|s_i\|_\sigma
+A_{n_i,\sigma}\Big)+A_{|\mathbf{n}|,\sigma}
=\log\|s_1\otimes\cdots\otimes s_r
\|_\sigma+A_{|\mathbf{n}|,\sigma}+\sum_{i=1}^r
A_{n_i,\sigma}.
\end{split}
\]
Since $E_{n_1,\sigma}\otimes\cdots\otimes
E_{n_r,\sigma}$ contains an orthogonal base
which consists of
$\rang(E_{n_1})\cdots\rang(E_{n_r})$ elements
of the forme $s_1\otimes\cdots\otimes s_r$,
using Cauchy-Schwarz inequality, we obtain
\[\log\|\varphi_{\mathbf{n}}\|_{\sigma}
\le A_{|\mathbf{n}|,\sigma}+\sum_{i=1}^r\Big(
A_{n_i,\sigma}+\frac12\log(\rang(E_{n_i}))\Big).\]
Therefore, \eqref{Equ:Majoration de la norme de produit
tensoriel} holds.
\end{proof}

\begin{remark}\label{Rem:fquaisifiltrae arith}
Lemma \ref{Lem:log croissance de la norem} implies that
\[\displaystyle
h(\varphi_{\mathbf{n}})\le\sum_{i=1}^r\Big(
2\varepsilon\log n_i+\frac
12\log(\rang(E_{n_i}))\Big),\quad \text{for any }
\mathbf{n}\in\mathbb N_{\ge n_0}^r.\] Therefore, if we
define $f(n)=2\varepsilon\log n+\displaystyle\frac
32\log(\rang(E_n))$, then the graded algebra $B$
equipped with Harder-Narasimhan filtrations is
$f$-quasi-filtered. Notice that the function $f$
satisfies
$\displaystyle\lim_{n\rightarrow\infty}f(n)/n=0$.
\end{remark}

We recall a result in \cite{Bost_Kunnemann},
which is a reformulation of Minkowski's first
theorem in Arakelov geometry.
\begin{proposition}[\cite{Bost_Kunnemann}]
\label{Pro:Bost_Kunnemann} Let $\overline
E=(E,(\|\cdot\|_\sigma)_{\sigma\in\Sigma_\infty})$ be a
non-zero Hermitian vector bundle on $\Spec\mathcal
O_K$. The following inequality holds:
\begin{equation}
\widehat{\mu}_{\max}(\overline
E)-\frac12\log([K:\mathbb Q]\rang
E)-\frac{\log|\Delta_K|}{2[K:\mathbb
Q]}\le-\frac12\log\sup_{0\neq s\in
E}\Big(\sum_{\sigma\in\Sigma_\infty}\|s\|_\sigma^2\Big)
\le\widehat{\mu}_{\max}(\overline E)-\frac
12\log[K:\mathbb Q] .
\end{equation}
\end{proposition}

\begin{lemma}\label{Lem:bornee de widehat mu max D}
There exists a constant $C$ such that
$\widehat{\mu}_{\max}(\overline E_D)\le CD$
for any sufficiently large integer $D$.
\end{lemma}
\begin{proof}
Let $\overline L$ be a Hermitian line bundle
on $\mathscr X$ which is arithmetically ample
and such that $c_1(\overline L)>0$. Suppose
that $s$ is a section of ${ \mathscr
L}^{\otimes D}$ on $\mathscr X$, then $\Div
s$ is an effective divisor of $\mathscr X$.
Therefore, we have
\[h_{\overline L}(\Div s)=\widehat{c}_1(\overline L)^d\cdot
\widehat{c}_1(\overline{\mathscr L}^{\otimes
D })+\int_{\mathscr X(\mathbb
C)}\log\|s\|c_1(\overline L)^d\ge 0.\] On the
other hand, since $c_1(\overline L)>0$, we
obtain
\[\int_{\mathscr X(\mathbb C)}\log\|s\|c_1(\overline L)^d
\le
\max_{\sigma\in\Sigma_\infty}\log\|s\|_{\sigma,\sup}
\int_{\mathscr X(\mathbb C)}c_1(\overline
L)^d.\] Therefore, by defining $\displaystyle
C_1=\widehat{c}_1(\overline
L)^d\cdot\widehat{c}_1(\overline{\mathscr L
})\left(\int_{\mathscr X(\mathbb
C)}c_1(\overline L)^d\right)^{-1}$, we have
$-\max_{\sigma}\log\|s\|_{\sigma,\sup}\le
C_1D$, which implies
$-\log\|s\|_\sigma\le-\log\|s\|_{\sigma,\sup}+A_{D,\sigma}
\le C_1D+\varepsilon \log D$ for any
$\sigma\in\Sigma_\infty$. We then obtain
after Proposition \ref{Pro:Bost_Kunnemann}
that
\[\begin{split}\widehat{\mu}_{\max}(\overline E_D)
&\le -\sup_{0\neq s\in E_D}\frac 12\log\Big(
\sum_{\sigma\in\Sigma_\infty}\|s\|_\sigma^2\Big)
+\frac{1}{2}\log([K:\mathbb Q]\rang
E_D)+\frac{\log|\Delta_K|}{2[K:\mathbb
Q]}\\&\le C_1D+\varepsilon\log D+\frac
12\log(\rang
E_D)+\frac{\log|\Delta_K|}{2[K:\mathbb
Q]}=O(D).
\end{split}
\]
\end{proof}

\begin{theorem}\label{Thm:convergence arithmetice de polygon}  For any sufficiently large integer $D$,
we denote by $P_D$ the normalized
Harder-Narasimhan polygon of $\overline E_D$.
Then the sequence $(\widehat{\mu}_{\min}(\overline E_D)/D)_{D\ge 1}$ has a limit in $\mathbb R$ and the sequence of polygons $(P_D/D)_{D\ge
1}$ converges uniformly to a concave function
on $[0,1]$.
\end{theorem}
\begin{proof}
Notice that $P_D$ coincides with the polygon
associated to the Harder-Narasimhan
filtration of $E_{D,K}$. Therefore, the
theorem follows from Theorem \ref{Thm:convergence de polygone
arithmetique}.
\end{proof}

\begin{remark}
The limit of polygons in Theorem \ref{Thm:convergence
arithmetice de polygon} does not depend on the choice
of Hermitian metrics $\|\cdot\|_\sigma$. Suppose that
for any integer $D\ge 0$ and any
$\sigma\in\Sigma_\infty$, we choose another Hermitian
metric $\|\cdot\|^*_\sigma$ on $E_{D,\sigma}$ such that
the collection
$h_D^*:=(\|\cdot\|^*_\sigma)_{\sigma\in\Sigma_\infty}$
is invariant under complex conjugation and such that
\[A_{D,\sigma}^*:=\sup_{0\neq s\in E_{D,\sigma}}
\Big|\log\|s\|_\sigma^*-\log\|s\|_{\sigma,\sup}\Big|=O(\log
D ).
\]
We denote by $P_D^*$ the normalized
Harder-Narasimhan polygon of $(E_D,h_D^*)$.
After \eqref{Equ:diffference of polygons}, we
have
\[|P_D^*(t)-P_D(t)|\le\frac{1}{[K:\mathbb Q]}
\sum_{\sigma\in\Sigma_\infty}\Big(A_{D,\sigma}+A_{D,\sigma}^*\Big).\]
Since
$\displaystyle\lim_{D\rightarrow+\infty}A_{D,\sigma}/D=
\lim_{D\rightarrow+\infty}A_{D,\sigma}^*/D=0$,
we know that the two sequences
$(P_D^*/D)_{D\ge 1}$ and $(P_D/D)_{D\ge 1}$
converge to the same limit. Similarly the slope inequality implies that
\[\lim_{D\rightarrow+\infty}\frac 1D\Big|
\widehat{\mu}_{\min}( E_D,h_D)-\widehat{\mu}_{\min}(E_D, h_D^*)\Big|=0.\]
\end{remark}

We establish now (Proposition \ref{Pro:pentes
max asymp}) the analogue of Proposition
\ref{Pro:convergence geometric de mu max et
mu min} in Arakelov geometry.

\begin{lemma}\label{Lem:suradditivite de mumax}
For any integer $r\ge 2$ and any element
$\mathbf{n}=(n_i)_{1\le i\le r}\in\mathbb
N_{\ge n_0}^r$, we have
\begin{equation}\widehat{\mu}_{\max}(\overline
E_{|\mathbf{n}|})\ge\sum_{i=1}^r\widehat{\mu}_{\max}(\overline
E_{n_i})-2\varepsilon\sum_{i=1}^r\log n_i-
\sum_{i=1}^r\Big(\frac 12\log([K:\mathbb Q
]\rang
E_{n_i})+\frac{\log|\Delta_K|}{2[K:\mathbb Q
]}\Big).\end{equation}
\end{lemma}
\begin{proof}
Suppose that for any integer $1\le i\le r$,
$s_i$ is a non-zero element in $E_{n_i}$.
Then for any $\sigma\in\Sigma_\infty$,
\[\begin{split}\|s_1\cdots s_r\|_{\sigma}
&\le \|s_1\cdots
s_r\|_{\sigma,\sup}\exp(A_{|\mathbf{n}|,\sigma})
\le\prod_{i=1}^r\|s_i\|_{\sigma,\sup}\exp(A_{|\mathbf{n}|,\sigma})\\
&\le\exp(A_{|\mathbf{n}|,\sigma})\prod_{i=1}^r\Big(\|s_i\|_{\sigma}\exp(A_{n_i,\sigma})\Big)
\le|\mathbf{n}|^{\varepsilon}n_1^{\varepsilon
}\cdots
n_r^{\varepsilon}\prod_{i=1}^r\|s_i\|_{\sigma}.
\end{split}\]
Therefore,
\begin{gather*}\sum_{\sigma\in\Sigma_\infty}\|s_1\cdots
s_r\|_{\sigma}^2 \le
|\mathbf{n}|^{2\varepsilon}\Big(\prod_{i=1}^r
n_i^{2\varepsilon}\Big)\sum_{\sigma\in\Sigma_\infty}\sum_{j=1}^r\|s_j\|_\sigma^2
\le \Big(\prod_{i=1}^r
n_i^{4\varepsilon}\Big)\Big(\prod_{j=1}^r\sum_{\sigma\in\Sigma_\infty}\|s_j\|_\sigma^2\Big),\\
\text{and}\quad-\frac
12\log\Big(\sum_{\sigma\in\Sigma_\infty}\|s_1\cdots
s_r\|_{\sigma}^2\Big)\ge
-2\varepsilon\sum_{i=1}^r\log
n_i-\sum_{i=1}^r\frac
12\log\Big(\sum_{\sigma\in\Sigma_\infty}\|s_i\|_\sigma^2\Big).
\end{gather*}
After Proposition \ref{Pro:Bost_Kunnemann},
we obtain
\[\begin{split}
\widehat{\mu}_{\max}(\overline
E_{|\mathbf{n}|})&\ge-\frac
12\log\Big(\sum_{\sigma\in\Sigma_\infty}\|s_1\cdots
s_r\|_{\sigma}^2\Big) \ge
-2\varepsilon\sum_{i=1}^r\log
n_i-\sum_{i=1}^r\frac
12\log\Big(\sum_{\sigma\in\Sigma_\infty}\|s_i\|_\sigma^2\Big)\\
&\ge\sum_{i=1}^r\widehat{\mu}_{\max}(\overline
E_{n_i})-2\varepsilon\sum_{i=1}^r\log n_i-
\sum_{i=1}^r\Big(\frac 12\log([K:\mathbb Q
]\rang
E_{n_i})+\frac{\log|\Delta_K|}{2[K:\mathbb Q
]}\Big).
\end{split}\]
\end{proof}

\begin{proposition}\label{Pro:pentes max asymp}
The
sequence $(\frac 1D\widehat{\mu}_{\max}(\overline
E_D))_{D\ge 1}$ has
a limit in $\mathbb R$.
\end{proposition}
\begin{proof}
This is a direct consequence of Lemma
\ref{Lem:suradditivite de mumax} and Corollary
\ref{Cor:limite de an sur n forme forte}.
\end{proof}

From the slope inequality we know immediately
that the limits in Proposition
\ref{Pro:pentes max asymp} do not depend on
the choice of Hermitian metrics
$\|\cdot\|_\sigma$.

\section{Calculation of the limit of polygons for a bigraded algebra}

\hskip\parindent We present in this section
an explicit calculation of the limit of
polygons in the case where the bigraded
algebra associated to the quasi-filtered
graded algebra is of finite type. The method
used in this section is inspired by an
article of Faltings and W{\"u}stholz
\cite{Fal-Wusth}, which applies the theory of
Poincar\'e series in two variables.

\begin{definition} Let $A$ be a commutative
ring. We call {\it bigraded} $A$-{algebra}
any $\mathbb N^2$-graded commutative
$A$-algebra. If $B$ is a bigraded
$A$-algebra, we call {\it bigraded}
$B$-module any $B$-module $M$ equipped with a
$\mathbb Z^2$-graduation in $A$-modules such
that, for any $(n,d)\in\mathbb N^2$ and any
$(n',d')\in\mathbb Z^2$, we have
$B_{n,d}M_{n',d'}\subset M_{n+n',d+d'}$. We
call {\it homogeneous sub-$B$-module} of $M$
any sub-$B$-module $M'$ of $M$ such that
$\displaystyle M'=\bigoplus_{(n,d)\in\mathbb
Z^2}M'\cap M_{n,d}$. $M'$ is therefore
canonically equipped with a structure of
graded $B$-module. In particular, if $B$ is a
bigraded $A$-algebra, then $B$ is canonically
equipped with a structure of bigraded
$B$-module. The homogeneous sub-$B$-modules
of $B$ are called {\it homogeneous ideals} of
$B$.

If $B$ is a bigraded $A$-algebra and if $M$ is a
bigraded $B$-module, for any $(n,d)\in\mathbb Z^2$, we
denote by $M(n,d)$ the graded $B$-module such that
$M(n,d)_{n',d'}=M_{n+n',d+d'}$ for any
$(n',d')\in\mathbb Z^2$.
\end{definition}

Let $\mathbf{f}$ be a mapping from
$\{1,\cdots, n\}$ to $\mathbb N^2$. The ring
$A[T_1,\cdots,T_n]$ of polynomials is
canonically equipped with an $\mathbb
N^2$-graduation such that $T_i$ is
homogeneous of bidegree $\mathbf{f}(i)$. We
obtain hence a bigraded $A$-algebra, denote
by $A[\mathbf{f}]$.

If $B$ is a bigraded $A$-algebra of finite
type, then $B$ is generated by a finite
number of homogeneous elements
$x_1,\cdots,x_m$. We suppose that $x_i$ is of
bidegree $(n_i,d_i)$. Let
$\mathbf{f}:\{1,\cdots,m\}\rightarrow\mathbb
N^2$ be the function which sends $i$ to
$(n_i,d_i)$. Then the surjective homomorphism
of $A$-algebras from $A[\mathbf{f}]\cong
A[T_1,\cdots,T_m]$ to $B$ which sends $T_i$
to $x_i$ is compatible with $\mathbb
N^2$-graduations. It is therefore a
homomorphism of bigraded algebras. In this
case, any bigraded $B$-module $M$ can be
considered as a bigraded
$A[\mathbf{f}]$-module, which is of finite
type if $M$ is a $B$-module of finite type.

\begin{definition}\label{Def:serie de poincare}
Let $\mathbf{f}=(f_1,f_2)$ be a mapping from
$\{1,\cdots,m\}$ to $\mathbb N^2$ and $M$ be
a bigraded $A[\mathbf{f}]$-module of finite
type whose homogeneous component are all
$A$-modules of finite length. We call {\it
Poincar\'e series} of $M$ the element
$P_M\in\mathbb \mathbb Z\lbr
X,Y\rbr[X^{-1},Y^{-1}]$ defined by the
formula
$P_M=\displaystyle\sum_{(n,d)\in\mathbb
Z^2}\len_A(M_{n,d})X^nY^d$. We write
$\displaystyle
Q_M=P_M\prod_{i=1}^m(1-X^{f_1(i)}Y^{f_2(i)})$.
\end{definition}

\begin{proposition}\label{Pro:Qm est un polynome}
We have $Q_M\in\mathbb Z[X,Y,X^{-1},Y^{-1}]$.
\end{proposition}
\begin{proof} By replacing
$A$ with $A/\ann_A(M)$, we reduce the problem to the
case where $\ann_A(M)=0$. Since $M$ is an
$A[\mathbf{f}]$-module of finite type, there exist
integers $a<b$ such that $M$ is generated as
$A[\mathbf{f}]$-module by $\displaystyle
M'=\bigoplus_{(n,d)\in[a,b]^2\cap\mathbb Z^2}M_{n,d}$.
Since $M'$ is an $A$-module of finite length, and since
$\ann_A(M')=\ann_A(M)=0$, the ring $A$ is Artinian, so
is Notherian.

We deduce by induction on $m$. If $m=0$, then
$A[\mathbf{f}]=A$. Since $M$ is an $A$-module
of finite type, we have $P_M\in\mathbb
Z[X,Y,X^{-1},Y^{-1}]$. Suppose that the
proposition has been proved for
$1,\cdots,m-1$. Let $\mathbf{f}'$ be the
restriction of $\mathbf{f}$ on
$\{1,\cdots,m-1\}$. We write
$(n_m,d_m)=\mathbf{f}(m)$. The mapping
$T_m:M(-n_m,-d_m)\longrightarrow M$ is a
homomorphism of bigraded
$A[\mathbf{f}]$-modules. Let $N$ be its
kernel (considered as homogeneous
sub-$A[\mathbf{f}]$-module of $M$). We have
an exact sequence
\[\xymatrix{0\ar[r]&N(-n_m,-d_m)\ar[r]&
M(-n_m,-d_m)\ar[r]&M\ar[r]&M/T_mM\ar[r]&0}.\]
Therefore,
$P_M-X^{n_m}Y^{d_m}P_M=P_{M/T_mM}-X^{n_m}Y^{d_m}P_N$.
Since $M/T_mM$ and $N$ are $A[f']=A[f]/(T_m)$-modules
of finite type, by induction hypothesis, we obtain
\[Q_M=Q_{M/T_mM}-X^{n_m}Y^{d_m}Q_N\in\mathbb
Z[X,Y,X^{-1},Y^{-1}].\]
\end{proof}

\begin{remark} \label{Rem:algebre engendre par les degre 1}
Let
$\mathbf{f}=(f_1,f_2):\{1,\cdots,m\}\rightarrow\mathbb
N^2$ be a mapping such that $f_1\equiv 1$ and $M$ be a
bigraded $A[\mathbf{f}]$-module of finite type, whose
homogeneous components are $A$-modules of finite
length. The algebra $A[\mathbf{f}]$, equipped with the
first graduation, is the usually graded algebra of
polynomials in $m$ variables. We can also consider the
first graduation of $M$ for which the $n^{\text{th}}$
homogeneous component of $M$ is $\bigoplus_{d\in\mathbb
Z}M_{n,d}$. This homogeneous component is an $A$-module
of finite length since there exist only a finite number
of integers $d$ such that $M_{n,d}\neq 0$. With this
graduation, $M$ is a graded module of finite type over
the polynomial algebra $A[T_1,\cdots,T_m]$ (with the
usual grading). If we denote by $H_M$ the Poncar\'e
series associated to $M$ (for the first grading), we
have $H_M(X)=P_M(X,1)$. The notions $\dim M$ and $c(M)$
are hence defined, as in Section \ref{Sec:Convergence
of polygons of a quasi-filtered graded algebra}.
\end{remark}

The following theorem is an analogue in the two
variables case of the formula \eqref{Equ:serie de
Poincare} for Poincar\'e series.

\begin{theorem}\label{Thm:ecriture de serie de poincare a deux variables}
With the notations of Remark \ref{Rem:algebre
engendre par les degre 1}, the series $P_M$
is written in the forme
\[P_M(X,Y)=\sum_{r=0}^{h}\sum_{\begin{subarray}{c}
\alpha\subset\{1,\cdots,m\}\\
\#\alpha=r
\end{subarray}}I_{\alpha}(X,Y)\prod_{i\in\alpha}(1-XY^{f_2(i)})^{-1},\]
where
\begin{enumerate}[1)]
\item $I_{\alpha}\in\mathbb Z[X,Y,X^{-1},Y^{-1}]$,
\item if $\#\alpha=h$, the coefficients of $I_\alpha$
are positive,
\item if $M\neq 0$, there exists
at least an subset
$\alpha\subset\{1,\cdots,m\}$ of cardinal $h$
such that $I_\alpha\neq 0$.
\end{enumerate}
\end{theorem}

\begin{remark}
With the notations of Theorem
\ref{Thm:ecriture de serie de poincare a deux
variables}, we have
\[H_M(X)=\sum_{r=0}^h\Bigg(\sum_{\begin{subarray}{c}\alpha\subset\{1,\cdots,m\}\\
\#\alpha=r
\end{subarray}}I_\alpha(X,1)\Bigg)(1-X)^{-r}.\]
Therefore, if $M$ is non-zero, then $\dim
M=h$ and
$\displaystyle c(M)=\sum_{\begin{subarray}{c}\alpha\subset\{1,\cdots,m\}\\
\#\alpha=h
\end{subarray}}I_\alpha(1,1)$.
\end{remark}

To simplify the proof of Theorem
\ref{Thm:ecriture de serie de poincare a deux
variables}, we introduce the following
notation. If $M$ is a bigraded
$A[\mathbf{f}]$-module satisfying the
assertion of Theorem \ref{Thm:ecriture de
serie de poincare a deux variables}, we say
that $M$ verifies the the condition $\mathbb
P$, noted by $\mathbb P(M)$. The assertion of
Theorem \ref{Thm:ecriture de serie de
poincare a deux variables} then becomes:
\[\text{\it For any }A[\mathbf{f}]\text{\it -module }M,
\text{\it{ we have }} \mathbb P(M).\]

For any integer $m>0$, let $\Theta_m$ be the
set $\{(i,j)\in\mathbb Z^2\;|\;0\le i\le m,\;
j>0\}\cup\{(-\infty,0)\}$. We equip it with
the lexicographic relation ``$\le$'' as
follows:
\[(i,j)\le(i',j')\text{ if and only if }i<i'\text{ or if }i=i',\;j\le j'.\]
We verify easily that it is an order relation on
$\Theta_m$ and that the set $\Theta_m$ is totally
ordered for this relation. We use the expression
$(i,j)<(i',j')$ to present the condition
$(i,j)\le(i',j')\text{ but }(i,j)\neq(i',j')$.

\begin{lemma}\label{Lem:proprites de P}
Let
$\xymatrix{0\ar[r]&M'\ar[r]&M\ar[r]&M''\ar[r]&0}$
be a short exact sequence of bigraded
$A[\mathbf{f}]$-module. Suppose that
$M_{n,d}'$ and $M_{n,d}''$ are $A$-modules of
finite length  for any $(n,d)\in\mathbb Z^2$.
Then $M_{n,d}$ are $A$-modules of finite
length, and
\begin{enumerate}[1)]
\item $\dim M=\max(\dim M',\dim M'')$,
\item \[c(M)=\begin{cases}
c(M')+c(M''),&\dim M'=\dim M'',\\
c(M'),&\dim M'>\dim M'',\\
c(M''),&\dim M''>\dim M'.
\end{cases}\]
\item $\mathbb P(M')\text{ and }\mathbb P(M'')\Longrightarrow\mathbb P(M)$.
\end{enumerate}
\end{lemma}
\begin{proof}
In fact, we have $P_M=P_{M'}+P_{M''}$ and
$H_M=H_{M'}+H_{M''}$. By definition we know
that 1) and 2) are true. Finally, 3) is a
consequence of 1) and of the fact that
$P_M=P_{M'}+P_{M''}$.
\end{proof}

{\noindent{\it Proof of Theorem \ref{Thm:ecriture de
serie de poincare a deux variables}.\ \ \ }} By the
same argument as that for the proof of Proposition
\ref{Pro:Qm est un polynome}, we can suppose that $A$
is an Artinian ring. We shall prove the theorem by
induction on $m$. First we prove that the theorem is
true in the case where $\dim M\le 0$. If $M$ is of
dimension $\le 0$, then the Poincar\'e series
$H_M(X)=P_M(X,1)$ of $M$ is an element of $\mathbb
Z[X,X^{-1}]$, and $P_M\in\mathbb Z[X,Y,X^{-1},Y^{-1}]$.
Hence we have $\mathbb P(M)$. Since $\dim M\le m$, the
theorem is true when $m=0$. Suppose that the theorem is
true for bigraded modules of an $A$-algebra of
polynomials in $j$ variables ($0\le j<m$). Let
$\mathbf{f}=(f_1,f_2):\{1,\cdots,m\}\rightarrow\mathbb
N^2$ be a mapping such that $f_1\equiv 1$ and let $M$
be a bigraded $A[\mathbf{f}]=A[T_1,\cdots,T_m]$-module
of finite type such that $M_{(n,d)}$ is of finite
length over $A$ for any $(n,d)\in\mathbb Z^2$. Suppose
that $f_2(m)=d$.

We begin another procedure of induction on $(\dim M,c(
M))$. We have already proved $\mathbb P(M)$ for $\dim
M\le 0$. Suppose that we have proved $\mathbb P(M)$ for
$(\dim M,c( M))<(r,s)$, where $0<r\le m$, $s>0$. In the
following, we shall prove $\mathbb P(M)$ in the case
where $(\dim M,c(M))=(r,s)$. Consider the homothetic
transformation $T_m:M(-1,-d)\longrightarrow M$, which
is a homomorphism of bigraded $A[\mathbf{f}]$-modules.
We denote by $\mathbf{f}'$ the restriction of
$\mathbf{f}$ on $\{1,\cdots,m-1\}$. Let $N_1$ be the
kernel of $T_m$ (considered as homogeneous
sub-$A[\mathbf{f}]$-module). It is a bigraded
$A[\mathbf{f}']$-module of finite type. After the
induction hypothesis, we have $\mathbb P(N_1)$. Let
$M_1=M/N_1$. After Lemma \ref{Lem:proprites de P} 3),
to prove $\mathbb P(M)$, it suffices to prove $\mathbb
P(M_1)$. If $\dim N_1=\dim M$, then either $\dim
M_1<\dim M$, or $\dim M_1=\dim M$ and $c(M_1)=c( M)-c(
N_1)<c(M)$. So we always have $(\dim M_1,c(M_1))<(\dim
M,c(M))$. After the induction hypothesis, we have
$\mathbb P(M_1)$. Otherwise we have $\dim N_1<\dim M$
and $(\dim M_1,c(M_1))=(\dim M,c(M))$. If $\mathbb
P(M)$ is not true, by iterating the procedure above, we
obtain an increasing sequence of homogeneous submodules
\begin{equation}
\label{Equ:la suite croissante de
sous-modules avec dimension
diminue}N_1\subset N_2\subset\cdots
N_j\subset N_{j+1}\subset\cdots
\end{equation}
of $M$ such that (we define $M_0=M$)
\begin{enumerate}[i)]
\item $N_j=\Ker T_m^j$,
\item $\dim N_j<\dim M$,
\item $M_j:=M/N_j$ don't satisfy
the condition $\mathbb P$, and $(\dim
M_j,c(M_j))=(\dim M,c(M))$.
\end{enumerate}
Since $A[\mathbf{f}]$ is a Noetherian ring,
the sequence (\ref{Equ:la suite croissante de
sous-modules avec dimension diminue}) is
stationary. In other words, there exists
$j\in\mathbb N$ such that $M_j=M_{j+1}$.
Since $M_{j+1}$ identifies canonically with
the image of $M_j$ by the homothetic
transformation $T_m$,we have the exact
sequence
\[\xymatrix{0\ar[r]&M_j(-1,-d)\ar[r]^>>>>>{T_m}
&M_j\ar[r]&M_j/T_mM_j\ar[r]&0}.\] We write
$N'=M_j/T_mM_j$. It is actually an
$A[\mathbf{f}']$-module of finite type. After induction
hypothesis, we have $\mathbb P(N')$. Finally, since
$(1-XY^d)P_{M_j}(X,Y)=P_{N'}(X,Y)$, we have $\mathbb
P(M_j)$, which is absurd. Hence we have $\mathbb P(M)$.
\endzm

Let $P$ be the formal series in $\mathbb Z\lbr
X,Y\rbr[X^{-1},Y^{-1}]$ with positive coefficients.
Then $P$ is written in the forme $\displaystyle
P(X,Y)=\sum_{(n,d)\in\mathbb Z^2}a_{n,d}(P)X^nY^d$. For
any $n\in\mathbb N$, we write $S_n(P)=\sum_{d\in\mathbb
Z}a_{n,d}(P)$ and denote by $\nu_{n,P}$ the Borel
measure on $\mathbb R$ defined by
\[\displaystyle\nu_{n,P}=\sum_{d\in\mathbb
Z}\frac{a_{n,d}(P)}{S_n(P)}\delta_{d/n}.\] If
$S_n(P)=0$, then $\nu_{n,P}$ is by convention the zero
measure.

\begin{remark}
We keep the notations of Theorem
\ref{Thm:ecriture de serie de poincare a deux
variables} in supposing that $A$ is a field.
If for any integer $n$, we equip the space $
M_{n,\bullet}:=\bigoplus_{d\in\mathbb
Z}M_{n,d}$ with the $\mathbb R$-filtration
$\mathcal F$ defined by $\mathcal
F_{\lambda}M_{n,\bullet}=\bigoplus_{d\ge\lambda}M_{n,d}$,
then the measure $\nu_{n,P}$ identifies with
$T_{\frac 1n}\nu_{M_{n,\bullet}}$. This
observation is crucial because it enables us
to use the Poincar\'e series to study
measures of a bigraded algebra over a field.
\end{remark}

\begin{proposition}
If $P$ is a series in $\mathbb Z\lbr X,Y\rbr$
of the forme $\displaystyle
P(X,Y)=\prod_{i=1}^m(1-XY^{d_i})^{-1}$, then
\begin{enumerate}[1)]
\item the Borel measures $\nu_{n,P}$
converge vaguely to a Borel measure $\nu_{P}$ when
$n\rightarrow+\infty$;
\item the sequence of functions
$\displaystyle\Big(F_{n,P}:x\longmapsto 1-\int_{\mathbb
R}\indic_{]-\infty,x]}\mathrm{d}\nu_{n,P}\Big)_{n\ge
1}$ converges simply to $\displaystyle F_P:x\longmapsto
1-\int_{\mathbb R}\indic_{]-\infty,x]}\mathrm{d}\nu_P$.
\end{enumerate}
\end{proposition}
\begin{proof} 1) We have
\[\begin{split}P(X,Y)&=\prod_{i=1}^m
\Big(\sum_{n\ge
0}X^nY^{nd_i}\Big)=\sum_{(n,d)\in\mathbb N\times\mathbb
Z}\Bigg(\sum_{\begin{subarray}{c}
(u_1,\cdots,u_m)\in\mathbb N^m\\
u_1+\cdots+u_m=n,\\
u_1d_1+\cdots+u_md_m=d
\end{subarray}}1\Bigg)X^nY^d\\
&=1+\sum_{\begin{subarray}{c}(n,d)\in\mathbb
Z^2\\
n>0\end{subarray}}\Bigg(\sum_{\begin{subarray}{c}
(\mu_1,\cdots,\mu_m)\in \frac{1}{n}\mathbb N^m\\
\mu_1+\cdots+\mu_m=1,\\
\mu_1d_1+\cdots+\mu_md_m=d/n
\end{subarray}}1\Bigg)X^nY^d.\end{split}\]
On the other hand,
$\displaystyle S_n(P)=\sum_{\begin{subarray}{c}(\mu_1,\cdots,\mu_m)\in\frac{1}{n}\mathbb N^m\\
\mu_1+\cdots+\mu_m=1\end{subarray}}1$. Let $\Delta_m$
be the simplex $\{(\mu_1,\cdots,\mu_m)\in\mathbb
R_+^m\;|\;\mu_1+\cdots+\mu_m=1\}$,
$\varphi:\Delta_m\rightarrow\mathbb R$ be the mapping
which sends $(\mu_1,\cdots,\mu_m)$ to
$\mu_1d_1+\cdots+\mu_md_m$. For any integer $n>0$, let
$\eta_{n,P}$ be the measure on $\Delta_m$ defined by
$\displaystyle
\eta_{n,P}=\sum_{{\mu}\in\frac{1}{n}\mathbb N^m \cap
\Delta_m}\frac{1}{S_n(P)}\delta_{{\mu}}$. We observe
that $\nu_{n,P}$ is the direct image of $\eta_{n,P}$ by
$\varphi$. Therefore, $\nu_{n,P}$ is supported by
$\varphi(\Delta_m)$. Hence for any continuous function
$f:\mathbb R\rightarrow\mathbb R$, $f$ is integrable
with respect to the measure $\nu_{n,P}$. Furthermore,
we have $\displaystyle \int_{\mathbb
R}f\mathrm{d}\nu_{n,P}=\int_{\Delta_m} (f\circ\varphi)
\mathrm{d}\eta_{n,P}$, which is the $n^{\text{th}}$
Riemann sum of the function
$f\circ\varphi:\Delta_m\rightarrow\mathbb R$. So the
sequence $\displaystyle\Big(\int_{\mathbb
R}f\mathrm{d}\nu_{n,P}\Big)_{n\ge 1}$ converges to
$\displaystyle\int_{\Delta_m}f\circ\varphi
\mathrm{d}\eta=\int_{\mathbb
R}f\mathrm{d}\varphi_*\eta$ where $\eta$ is the
Lebesgue measure on $\Delta_m$. We then obtain that the
measures $\nu_{n,P}$ converge vaguely to the measure
$\nu_P=\varphi_*\eta$.

2) The mapping $\varphi$ can be extended to
an affine mapping $\Phi$ from
$\{(\mu_1,\cdots,\mu_m)\in\mathbb
R^m\;|\;\mu_1+\cdots+\mu_m=1\}$ to $\mathbb
R$ by simply defining
$\Phi(\mu_1,\cdots,\mu_m)=\mu_1d_1+\cdots\mu_md_m$.
If $d_1=d_2=\cdots=d_m=d$, then
$P(X,Y)=(1-XY^d)^{-m}$. Therefore, for any
$n\ge 1$, $\nu_{n,P}=\nu_P=\delta_d$. The
assertion is then evident. Otherwise the
image of $\Phi$ is the whole set $\mathbb R$
and for
any point $x\in\Image\varphi$, $\varphi^{-1}(x)$ is a negligible subset
of $\Delta_m$ for the Lebesgue measure. Therefore, the
one point set $\{x\}$ is negligible for the
measure $\lambda_P$. After \cite{Bourbaki65}
IV.5 Proposition 22, since $x$ is the only
discontinuous point of the function
$\indic_{]-\infty,x]}$, we obtain that the
sequence $\displaystyle\Big(\int_{\mathbb
R}\indic_{]-\infty,x]}\mathrm{d}\nu_{n,P}
\bigg)_{n\ge 1}$  converges to $\displaystyle
\int_{\mathbb
R}\indic_{]-\infty,x]}\mathrm{d}\nu_P$.
\end{proof}

\begin{proposition}\label{Cor:convergence de mesure pour les series monome}
Suppose that $Q$ is a non-zero series in
$\mathbb Z[X,Y,X^{-1},Y^{-1}]$ with positive
coefficients, and $P\in\mathbb Z\lbr
X,Y\rbr[X^{-1},Y^{-1}]$ is of the form
$\displaystyle
P(X,Y)=Q(X,Y)\prod_{i=1}^m(1-XY^{d_i})^{-1}$.
\begin{enumerate}[1)]
\item The Borel measures $\nu_{n,P}$
converge vaguely to a Borel measure $\nu_P$
when $n\rightarrow+\infty$.
\item Define the functions
\[\Big(F_{n,P}:x\longmapsto 1-\int_{\mathbb R}\indic_{]
-\infty,x]}\mathrm{d}\nu_{n,P}\Big)_{n\ge 1}
\qquad\text{and}\qquad F_P:x\longmapsto 1-\int_{\mathbb
R} \indic_{]-\infty,x]}\mathrm{d}\nu_P.\]
\begin{enumerate}[i)]
\item If $d_1=\cdots=d_m=d$, then for any
$x\neq d$, the sequence $(F_{n,P}(x))_{n\ge
1}$ converges to $F_P(x)$.
\item If $d_i$'s are not identical,
then the sequence of functions
$(F_{n,P})_{n\ge 1}$ converges simply to
$F_P$.
\end{enumerate}
\end{enumerate}
Furthermore, if we denote by $P'$ the series
$\displaystyle
P'(X,Y)=\prod_{i=1}^m(1-XY^{d_i})^{-1}$, then
we have $\nu_P=\nu_{P'}$, and hence
$F_P=F_{P'}$.
\end{proposition}
\begin{proof}
1) Suppose that $Q$ is of the form $\displaystyle
Q(X,Y)=\sum_{|n'|\le e }\sum_{|d'|\le r}
c_{n',d'}X^{n'}Y^{d'}$ where $c_{n',d'}\ge 0$. Since
$P=P'Q$, we obtain $\displaystyle
a_{n,d}(P)=\sum_{|n'|\le e }\sum_{|d'|\le
r}c_{n',d'}\,a_{n-n',d-d'}(P')$ and
\[\begin{split}S_n(P)&=\sum_{d\in\mathbb Z}a_{n,d}(P)
=\sum_{d\in\mathbb Z}\sum_{|n'|\le e
}\sum_{|d'|\le r}c_{n',d'}\,a_{n-n',d-d'}(P')\\
&=\sum_{|n'|\le e }\sum_{|d'|\le r}\sum_{d\in\mathbb
R}c_{n',d'}\,a_{n-n',d-d'}(P')=\sum_{|n'|\le e
}\sum_{|d'|\le r}c_{n',d'}\,S_{n-n'}(P').\end{split}\]
Denote by $\displaystyle C_{n'}=\sum_{|d'|\le
r}c_{n',d'}$, then we have $\displaystyle
S_n(P)=\sum_{|n'|\le e}C_{n'}S_{n-n'}(P')$. If
$g:\mathbb R\rightarrow\mathbb R$ is a continuous
function with compact support, then
\[\int_{\mathbb R}g\,\mathrm{d}\nu_{n,P}=\sum_{d\in\mathbb
Z}\frac{a_{n,d}(P)}{S_n(P)}\,g(d/n)\\
=\frac{1}{S_n(P)}\sum_{d\in\mathbb Z}\sum_{|n'|\le e
}\sum_{|d'|\le r}
c_{n',d'}\,a_{n-n',d-d'}(P')\,g(d/n).\] Notice that
\[\begin{split}&\quad\;\frac{1}{S_n(P)}
\sum_{|n'|\le e }\sum_{|d'|\le
r}\sum_{d\in\mathbb Z}
c_{n',d'}\,a_{n-n',d-d'}(P')\,g\left(\frac{d-d'}{n-n'}\right)\\
&=\frac{1}{S_n(P)}\sum_{|n'|\le e }\sum_{|d'|\le
r}c_{n',d'}S_{n-n'}(P')\int_{\mathbb R
}g\,\mathrm{d}\nu_{n-n',P'}=\frac{1}{S_n(P)}
\sum_{|n'|\le e}C_{n'}S_{n-n'}(P')\int_{\mathbb R
}g\,\mathrm{d}\nu_{n-n',P'}\end{split}\] converges to
$\int_{\mathbb R}g\,\mathrm{d}\nu_{P'}$ since
$\nu_{n,P'}$ converges vaguely to $\nu_{P'}$ when
$n\rightarrow\infty$. Finally, the function $g$ is
uniformly continuous on $\mathbb R$. For any number
$\delta>0$, there exists a number $\varepsilon>0$ such
that, for all $x,y\in\mathbb R$ such that
$|x-y|<\varepsilon$, we have $|g(x)-g(y)|<\delta$. On
the other hand, since $\displaystyle
P'=\prod_{i=1}^m(1-XY^{d_i})^{-1}$, if
$|d|>\displaystyle|n|\max_{1\le i\le m}|d_i|$, we have
$a_{n,d}(P')=0$. Hence for all integers $d,n$ such that
$|d|>\displaystyle\max_{1\le i\le m}|d_i|(|n|+e)+r$, we
have $a_{n-n',d-d'}(P')=0$ for any $|n'|\le e$ and any
$|d'|\le r$. Therefore, for all integers $n>e$,
$d\in\mathbb Z$, $|n'|\le e$ and $|d'|\le r$ such that
$a_{n-n',d-d'}(P')\neq 0$, we always have
\[\left
|\frac{d}{n}-\frac{d-d'}{n-n'}\right|=
\left|\frac{d'n-dn'}{n(n-n')}\right|
\le\frac{r}{n-n'}+\max_{1\le i\le m
}|d_i|\frac{e(n+e+r)}{n(n-n')}.
\]
Therefore, there exists an integer $N>0$ such that, for
all integers $n>N$, $d\in\mathbb Z$, $|n'|\le e$ and
$|d'|\in r$, we have either $a_{n-n',d-d'}(P')=0$, or
$\displaystyle\Big|\frac{d}{n}-\frac{d-d'}{n-n'}\Big|<\varepsilon$.
Hence we have
\[\begin{split}
&\quad\;\Bigg|\int_{\mathbb R}g\,\mathrm{d}\nu_{n,P}-
\frac{1}{S_n(P)}\sum_{|n'|\le e }\sum_{|d'|\le
r}\sum_{d\in\mathbb
Z}c_{n',d'}\,a_{n-n',d-d'}(P')\,g\left(\frac{d-d'}{n-n'}\right)\Bigg|\\
&\le\frac{1}{S_n(P)}\sum_{|n'|\le e }\sum_{|d'|\le
r}\sum_{d\in\mathbb
Z}c_{n',d'}\,a_{n-n',d-d'}(P')\left|g\left(\frac
dn\right)-g\left(\frac{d-d'}{n-n'}\right)\right|\\
&\le\frac{\delta}{S_n(P)}\sum_{|n'|\le e }\sum_{|d'|\le
r}\sum_{d\in\mathbb
Z}c_{n',d'}\,a_{n-n',d-d'}(P')=\delta.
\end{split}\]
We then deduce the vague convergence of
$\nu_{n,P}$ to $\nu_{P'}$.

2) If $d_1=\cdots=d_m=d$, then $\nu_{P}=\delta_{d}$. So
for any $x\neq d$, the set of discontinuous points of
$\indic_{]-\infty,x]}$, i.e., $\{x\}$, is negligible
for the measure $\nu_{P}$. Hence $\int_{\mathbb
R}\indic_{]-\infty,x]}\mathrm{d}\nu_{n,P}$ converges to
$\int_{\mathbb
R}\indic_{]-\infty,x]}\mathrm{d}\nu_{P}$. If $d_i$'s
are not identical, then any discrete subset of $\mathbb
R$ is negligible for the measure $\nu_{P}$, so the
sequence of functions $(F_{n,P})_{n\ge 1}$ converges
simply to the function $F_P$.
\end{proof}

\begin{remark}
With the notations of Proposition \ref{Cor:convergence
de mesure pour les series monome}, the limit measure
$\nu_P$ depends only on the vector
$(d_1,\cdots,d_m)\in\mathbb N^m$ (or simply the
equivalence class of $(d_1,\cdots,d_m)$ in $\mathbb
N^m/\mathfrak S_m$, the quotient of $\mathbb N^m$ by
the symmetric group $\mathfrak S_m$). In the following,
we denote by $\nu_{(d_1,\cdots,d_m)}$ this measure.
Actually, when $m>0$, it is a probability measure. When
$m=0$, $\nu_{\emptyset}$ is the zero measure.
\end{remark}

The following theorem is an immediate consequence of
Proposition \ref{Cor:convergence de mesure pour les
series monome}.

\begin{theorem}\label{Thm:mesure associe a un serie formel}
Let $(d_1,\cdots,d_m)\in\mathbb Z_+^m$ and
$\displaystyle P(X,Y)=\sum_{r=0}^h\sum_{\begin{subarray}{c}\alpha\subset\{1,\cdots,m\}\\
\card \alpha=r\end{subarray}}I_{\alpha}(X,Y)
\prod_{i\in\alpha}(1-XY^{d_i})^{-1}$ be a
series in $\mathbb Z\lbr
X,Y\rbr[X^{-1},Y^{-1}]$ where
\begin{enumerate}[a)]
\item the coefficients of $P$ are positive,
\item $I_\alpha\in\mathbb Z[X,Y,X^{-1},Y^{-1}]$,
\item for any $\alpha\subset\{1,
\cdots,m\}$ of cardinal $h$, the coefficients
of $I_\alpha$ are positive,
\item there exists at least
one $\alpha\subset\{1,\cdots,m\}$ of cardinal
$h$ such that $I_\alpha\neq 0$.
\end{enumerate}
Then
\begin{enumerate}[1)]
\item the Borel measures $\nu_{n,P}$
converge vaguely to a Borel measure $\nu_P$ when
$n\rightarrow+\infty$,
\item there exists a
finite subset $\Omega$ of $\mathbb R$ such
that the sequence of functions
\[\Big(F_{n,P}:x\longmapsto 1-
\int_{\mathbb R}\indic_{]-\infty,x]}\mathrm{d}\nu_{n,P}
\Big)_{n\ge 1}\] converges pointwise on $\mathbb
R\setminus\Omega$ to the function $\displaystyle
F_P:x\longmapsto 1-\int_{\mathbb R}
\indic_{]-\infty,x]}\mathrm{d}\nu_P$.
\end{enumerate}
Furthermore, if for any
$\alpha=\{i_1<\cdots<i_h\}$, we write
$d_\alpha=(d_{i_1},\cdots,d_{i_h})$, then the
limit measure $\nu_P$ equals to
$\displaystyle\sum_{\begin{subarray}{c}
\alpha\subset\{1,\cdots,m\}\\
\card\alpha=h
\end{subarray}}\frac{I_\alpha(1,1)}{S}\nu_{d_\alpha}$
where $\displaystyle
S=\sum_{\begin{subarray}{c}
\alpha\subset\{1,\cdots,m\}\\
\card\alpha=h
\end{subarray}} I_\alpha(1,1)$.
So $\nu_P$ is a probability measure when
$h>0$. If $h=0$, then $\nu_P$ is the zero
measure.
\end{theorem}

The results obtained above, notably Theorem
\ref{Thm:ecriture de serie de poincare a deux
variables} and Theorem \ref{Thm:mesure associe a un
serie formel}, imply immediately the following theorem.

\begin{theorem} Let $K$ be a
field,
$\mathbf{f}=(f_1,f_2):\{1,\cdots,m\}\rightarrow\mathbb
N^2$ be a mapping such that $f_1\equiv 1$ and $M$ be a
finite generated bigraded $K[\mathbf{f}]$-module. If
for any integer $n\ge 1$, we denote by $\nu_n$ the
Borel measure associated to the vector space
$M_{n,\bullet}:=\bigoplus_{d\in\mathbb Z }M_{n,d}$
which is equipped with the filtration induced by the second
grading, then the sequence of Borel measures $T_{\frac
1n}\nu_n$ converges vaguely to a Borel measure $\nu$ on
$\mathbb R$. If furthermore $M_{n,\bullet}$ is non-zero
for sufficiently large $n$, then the limit measure
$\nu$ is a probability measure, and the polygons
associated to $T_{\frac 1n}\nu_n$ converge uniformly to
a concave curve defined on $[0,1]$.
\end{theorem}

\begin{remark}
Let $K$ be a field and  $B$ be an $\mathbb N$-filtered
graded $K$-algebra (that is to say, the jumping set is
contained in $\mathbb N$) which is of finite type over
$K$ and is generated as $K$-algebra by $B_1$. We can
introduce a bigraded $K$-algebra $\widetilde B$ by
defining $\widetilde B_{n,d}=\mathcal F_{d}B_n/\mathcal
F_{d+1}B_n$. Notice that the filtered vector spaces
$\widetilde B_{n,\bullet}$ (whose filtration is induced
by the second grading) and $B_n$ have the same associated
measure. Therefore, if $\widetilde B$ is an algebra of
finite type over $K$ which is generated by $\widetilde
B_{1,\bullet}$, then the previous theorem shows that
the normalized polygons of $B_n$ converge uniformly.
However, this condition is not satisfied in general. We
can for example consider the algebra $B=K[X]$ of
polynomials, equipped with the usual graduation and the
filtration such that $\lambda(X^n)=n-1$ for any $n\ge
1$. Then $B$ is a filtered graded algebra since
$\lambda(X^{n+m})=n+m-1>n-1+m-1=\lambda(X^n)+\lambda(X^m)$.
On the other hand, the bigraded algebra $\widetilde B$
identifies with the algebra $K[T_1,\cdots,T_n,\cdots]$,
where the bidegree of $T_n$ is $(n,n-1)$, modulo the
homogeneous ideal generated by all elements of the form
$T_nT_m$. This is not an algebra of finite type over
$K$.
\end{remark}

Finally, we shall give an example of the limit of
normalized Harder-Narasimhan polygons in relative
geometric framework. Let $k$ be a field, and $C$ be a
smooth projective curve over $k$. We denote by $K$ the
field of rational functions on $C$. Let $(E_i)_{1\le
i\le m}$ be a finite family of locally free $\mathcal
O_C$-modules of finite type which are semistable. We
suppose in addition that for any family $(n_i)_{1\le
i\le m}$ of positive integers, the $\mathcal
O_C$-module $S^{n_1}E_1\otimes\cdots\otimes S^{n_m}E_m$
is semistable. This condition is satisfied notably when
one of the following conditions is satisfied:
\begin{enumerate}[1)]
\item the $\mathcal O_C$-modules $E_1,\cdots,E_m$ are all of rank $1$;
\item $C$ is the projective space
 $\mathbb P^1$;
\item $C$ is an elliptic curve over $k$ (see
\cite{Atiyah57});
\item $k$ is of characteristic $0$.
\end{enumerate}

Let $E$ be the direct sum
$E=E_1\oplus\cdots\oplus E_m$. Let $\mathscr
B$ be the symmetric algebra of $E$, which is
a graded $\mathcal O_C$-algebra. For any
integer $n\ge 1$, we have
\[\mathscr B_n=S^nE=\bigoplus_{\begin{subarray}{c}
(d_i)_{1\le i\le m}\in\mathbb N^m\\
d_1+\cdots+d_m=n
\end{subarray}}\Big(S^{d_1}E_1\otimes\cdots
\otimes S^{d_m}E_m\Big)^{\oplus\frac{n!}{d_1!\cdots
d_m!}}.\] Denote by $B$ the graded algebra over $K$
such that $B_n=\mathscr B_{n,K}$. For any integer $1\le
i\le m$, we denote by $r_i$ the rank of $E_i$ and by
$\mu_i$ the slope of $E_i$, and we choose a base
$\mathbf{u}_i=(u_{i,j} )_{1\le i\le r_j}$ of $E_{i,K}$.
We write
$\mathbf{u}=\mathbf{u}_1\amalg\cdots\amalg\mathbf{u}_m$
and $r=r_1+\cdots+r_m$ the rank of $E$. The algebra $B$
identifies hence with the algebra of polynomials
$K[\mathbf{u}]$. If
$\alpha:\mathbf{u}\rightarrow\mathbb R$ is a mapping,
denote by $|\alpha|$ the sum
$\sum_{i=1}^m\sum_{j=1}^{r_i}\alpha(u_{ij})$. For any
integer $n\ge 1$, we denote by $\nu_n=T_{\frac
1n}\nu_{\mathscr B_n}$ and we have
\[\begin{split}\nu_{\mathscr B_n}&=\sum_{\begin{subarray}{c}(d_i)_{1\le i\le m}\in\mathbb N^m\\
d_1+\cdots+d_m=n\end{subarray}}\frac{n!}{d_1!\cdots d_m!}
\frac{\rang(S^{d_1}E_1\otimes\cdots\otimes S^{d_m}E_m)}
{\rang(S^nE)}\,\delta_{d_1\mu_1+\cdots+d_m\mu_m}\\
&=\sum_{\begin{subarray}{c}
\alpha:\mathbf{u}\rightarrow\mathbb N\\
|\alpha|=n
\end{subarray}}\frac{1}{\rang(S^nE)}\,\delta_{\sum_{i=1}^m\mu_i\sum_{j=1}^{r_i} \alpha(u_{ij}).
}
\end{split}
\]
Therefore,
$\displaystyle\nu_n=\sum_{\begin{subarray}{c}
\beta:\mathbf{u}\rightarrow \frac{1}{n}\mathbb N\\
|\beta|=1
\end{subarray}}\frac{1}{\rang(S^nE)}
\,\delta_{\sum_{i=1}^m \mu_i
\sum_{j=1}^{r_i}\beta(u_{ij}). }$ Denote by $\Delta$
the simplex of dimension $r-1$ in $\mathbb R^r$
(considered as the function space of $\mathbf{u}$ in
$\mathbb R$) defined by the relation
\[\Delta:=\{x:\mathbf{u}\rightarrow\mathbb R_{\ge 0}\;|\;|x|=1\}.\]
and by $\Phi:\Delta\rightarrow\mathbb R$ the mapping
which sends $(x:\mathbf{u}\rightarrow\mathbb
R)\in\mathbb R^r$ to
$\sum_{i=1}^m\mu_i\sum_{j=1}^{r_i}x(u_{ij})$. This is a
continuous function. For any integer $n\ge 1$, let
$\Delta^{(n)}$ be the subset of $\Delta$ of functions
valued in $n^{-1}\mathbb N$. Then $\nu_n$ is the direct
image by $\Phi|_{\Delta^{(n)}}$ of the equidistributed
probability measure $w_n$ on $\Delta^{(n)}$. By abuse
of language, we still use the expression $w_n$ to
denote the direct image of $w_n$ by the inclusion
mapping from $\Delta^{(n)}$ in $\Delta$. Then
$\nu_n=\Phi_*(w_n)$. Since the sequence of measures
$(w_n)_{n\ge 1}$ converges vaguely to the uniform
measure on $\Delta$, the limit $\nu$ of the measure
sequence $(\nu_n)_{n\ge 1}$ exists and equals to the
direct image of the uniform measure on $\Delta$ by the
mapping $\Phi$. Therefore the uniform limit of polygons
associated to $\nu_n$ exists and equals to the
``polygon'' (it is in fact a concave function)
associated to the limit measure $\nu$.

\begin{example}
Let $E$ be the direct sum of two invertible modules
$L_1$ and $L_2$. We write $\mu_1=\deg(L_1)$ and
$\mu_2=\deg(L_2)$, and we suppose that $\mu_1<\mu_2$.
In this case, $\Delta=\{(x,1-x)\;|\;0\le x\le
1\}\subset\mathbb R^2$ is parametered by $[0,1]$. The
mapping $\Phi:\Delta\rightarrow\mathbb R$ sends
$(x,1-x)$ to $\mu_1x+\mu_2(1-x)$. Therefore, the limit
measure $\nu$ is the equidistributed probability
measure on $[\mu_1,\mu_2]$. Let $f$ be the function
defined by $f(t)=\esp^{\nu}[\indic_{\{x>t\}}]$. Then we
have $\displaystyle
f(x)=\frac{1}{\mu_2-\mu_1}\Big((\mu_2-x)_+-(\mu_1-x)_+\Big)$.
The quasi-inverse of $f$ is therefore
$f^*(t)=\mu_1t+\mu_2(1-t)$. Finally, the limit of
normalized Harder-Narasimhan polygons of $S^nE$ is
given by the quadratic curve
\[\mu_2x-\frac{\mu_2-\mu_1}{2}x^2,\]
which is non-trivial in general.
\end{example}
\bibliography{chen}
\bibliographystyle{alpha}

\appendix

\section{Pseudo-filtered graded algebra}
\label{App:Pseudo-filtered graded algebra}
\hskip\parindent In this section, we propose
another generalization of filtered graded
algebras which is weaker than the notion of
$f$-quasi-filtered graded algebras. By
imposing a condition on $f$ which is stronger
than
$\displaystyle\lim_{n\rightarrow+\infty}f(n)/n=
0$, we also obtain the vague convergence of
measures associated to filtrations and hence
the uniform convergence of polygons.

\begin{definition}\label{Def:algebre graduee pseudofiltree}
Let $B=\bigoplus_{n\ge 0}B_n$ be a graded $K$-algebra
and $f:\mathbb Z_{> 0 }\rightarrow\mathbb R_{\ge 0}$ be
a function. We say that $B$ is an $f$-{\it
pseudo-filtered graded} $K$-algebra if each $B_n$ is
equipped with a decreasing $\mathbb R$-filtration such
that, for all sufficiently large integers $n,m$, we
have
\[B_{n,s}B_{m,t}\subset B_{n+m,s+t-f(n)-f(m)}.\]
If $B$ is an $f$-pseudo-filtered graded
$K$-algebra, we say that a graded $B$-module
$M=\bigoplus_{n\in\mathbb Z}M_n$ is {\it
$f$-pseudo-filtered} if for any integer $n$,
$M_n$ is equipped with a decreasing $\mathbb
R$-filtration such that, for all sufficiently
large integers $n,m$, we have
\[B_{n,s}M_{m,t}\subset M_{n+m,s+t-f(n)-f(m)}.\]
\end{definition}

\begin{remark}
Note that $B$ is an $f$-pseudo-filtered
graded $B$-module. If $f\equiv 0$, then $B$
is a filtered graded $K$-algebra and $M$ is a
filtered graded $B$-module. If $g$ is another
function dominating $f$, then $B$ is a
$g$-pseudo-filtered graded $K$-algebra and
$M$ is a $g$-pseudo-filtered graded
$B$-module.
\end{remark}

Some results which are analogues to those in
Section \ref{Sec:Quasi-filtered graded
algebras} can be stated and verified without
difficulty for pseudo-filtered graded
algebras and for pseudo-filtered graded
modules, notably the corollaries
\ref{Cor:sous algebre algebre quotient},
\ref{Cor:sousmodule module quotient} and
\ref{Cor:fonctorialite de filtrations qotiet
etc} where we only need to replace
``quasi-filtered'' by ``pseudo-filtered'' in
the statement of the results.

Let $f:\mathbb Z_{\ge 0}\rightarrow\mathbb R_{\ge 0}$
be a decreasing function, $V$ be a vector space of rank
$0<d<+\infty$ over $K$, and $B$ be the symmetric
algebra generated by $V$, equipped with the usual
graduation. We suppose that for any positive integer
$n$, $B_n$ is equipped with an $\mathbb R$-filtration
which is separated, exhaustive and left continuous such
that $B$ is an $f$-pseudo-filtered graded $K$-algebra.
Let $g$ be an increasing function which is concave and
$c$-Lipschitz. For any integer $n\ge 0$, let
$I_n=\int_{\mathbb R}g\,\mathrm{d}T_{\frac
1n}\nu_{B_n}$. Then for all sufficiently large integers
$m$ and $n$, we have
\begin{equation}\begin{split}I_{m+n}&\ge\int_\mathbb
Rg\,\mathrm{d}\Big(T_{\frac{1}{m+n}}\nu_{\mathbf{u}^{(m,n)}}\Big)
=\int_{\Delta_{m+n}^{(d)}}g
\left(\frac{1}{m+n}\lambda(u_{\gamma}^{(m,n)})\right)
\mathrm{d}\xi_{m+n}(\gamma)\\
&=\int_{\Delta_{m}^{(d)}\times\Delta_{n}^{(d)}}g
\left(\frac{1}{m+n}\lambda(u_{\alpha+\beta}^{(m,n)})\right)
\mathrm{d}\rho_{(m,n)}(\alpha,\beta)\\
&\ge\int_{\Delta_{m}^{(d)}\times\Delta_{n}^{(d)}}
g\left(\frac{1}{m+n}\lambda(u_{\alpha}u_\beta)
\right)
\mathrm{d}\rho_{(m,n)}(\alpha,\beta)\\
&\ge\int_{\Delta_{m}^{(d)}\times\Delta_{n}^{(d)}}
g\left(\frac{\lambda(u_{\alpha})+
\lambda(u_\beta)-f(n)-f(m)}{m+n}
\right) \mathrm{d}\rho_{(m,n)}(\alpha,\beta)\\
&\ge\int_{\Delta_{m}^{(d)}\times\Delta_{n}^{(d)}}
\left[g\left(\frac{\lambda(u_{\alpha})+
\lambda(u_\beta)}{m+n}
\right)-\frac{c(f(n)+f(m))}{m+n}\right]
\mathrm{d}\rho_{(m,n)}(\alpha,\beta)\\
&\ge\int_{\Delta_{m}^{(d)}\times
\Delta_{n}^{(d)}}\left[\frac{n}{m+n}
g\left(\frac{\lambda(u_{\alpha})}{n}\right)
+\frac{m}{n+m}g\left(\frac{\lambda(u_\beta)}{m}
\right)\right]\mathrm{d}\rho_{(m,n)}(\alpha,\beta)-\frac{c(f(n)+f(m))}{m+n}\\
&=\frac{n}{n+m}I_n+\frac{m}{n+m}I_m-c\,\frac{f(n)+f(m)}{m+n}.
\end{split}\end{equation}

If the sequence $(I_n)_{n\ge 0}$ is bounded from above
and if $\sum_{\alpha\ge
0}f(2^\alpha)/2^\alpha<+\infty$, then $(I_n)_{n\ge 0}$
converges, which implies that the sequence of measures
$(T_{\frac 1n}\nu_{B_n})_{n\ge 1}$ converges vaguely.
In other words, $B$ satisfies the vague convergence
condition. The convergence of $(I_n)_{n\ge 0}$ is based
on Corollary \ref{Cor:inegalite de pesudo-filtreee},
which we shall present as below.

\begin{lemma}\label{Lem:lemme technqiue sur suite}
If $f:\mathbb Z_{>0}\rightarrow\mathbb R_{\ge 0}$ is an
increasing function such that $\sum_{\alpha\ge
1}f(2^\alpha)/2^\alpha<+\infty$, then\vspace{-1mm}
\[\lim_{\alpha\rightarrow+\infty}2^{-\alpha}
\sum_{i=0}^\alpha f(2^i)=0.\]
\end{lemma}
\begin{proof}
For any integer $\alpha\ge 0$, let
$S_\alpha=\sum_{i\ge\alpha}f(2^i)/2^i$. By Abel's
summation formula,
\[\sum_{i=0}^\alpha f(2^i)=
\sum_{i=0}^\alpha
(S_i-S_{i+1})2^i=S_0-S_{\alpha+1}2^\alpha+\sum_{i=1}^\alpha
S_i2^{i-1}.
\]
Since
$\displaystyle\lim_{\alpha\rightarrow+\infty}S_\alpha=0$,
we have $2^{-\alpha}\sum_{i=1}^\alpha S_i2^{i-1}$
converges to $0$ when $\alpha\rightarrow+\infty$, which
implies the lemma.
\end{proof}

\begin{proposition}\label{Pro:convergence de suite a croissance logarithmique}
Let $(b_n)_{n\ge 1}$ be a sequence of positive real
numbers and $f:\mathbb Z_{>0}\rightarrow\mathbb R_{\ge
0}$ be an increasing function such that
$\sum_{\alpha\ge 1}f(2^\alpha)/2^\alpha<+\infty$. If
there exists an integer $n_0>0$ such that, for any pair
$(m,n)$ of integers $\ge n_0$, we have $b_{n+m}\le
b_m+b_n+f(m)+f(n)$, then the sequence $(b_n/n)_{n\ge
1}$ has a limit in $\mathbb R_{\ge 0}$.
\end{proposition}
\begin{proof} First let us treat the case where
$n_0=1$. Since $f$ is an increasing function
we obtain that for any $(m,n)\in\mathbb
Z_{>0}^2$,
\begin{equation}\label{Equ:sousadditivite augemente1}b_{m+n}\le
b_n+b_m+2f(m+n).\end{equation} For any
integer $\alpha\ge 0$, let
$S_\alpha=\displaystyle\sum_{i\ge\alpha}f(2^i)/2^i$.
then
$\displaystyle\lim_{\alpha\rightarrow+\infty}S_\alpha=0$.
If $2^{\beta}\le n<2^{\beta+1}$ is an
integer, we have, for any $\alpha\in\mathbb
N$,
\begin{equation}\label{Equ:sousadditivite augemente2}
b_{2^\alpha n}\le 2^\alpha
b_n+\sum_{i=1}^\alpha
2^{\alpha+1-i}f(2^{i-1}n)\le 2^\alpha
b_n+\sum_{i=1}^\alpha
2^{\alpha+1-i}f(2^{\beta+i}).
\end{equation}
Suppose that $p=\sum_{i=0}^k\epsilon_i2^i$, where
$\epsilon_i\in\{0,1\}$ for any $0\le i<k$ and
$\epsilon_k=1$. If $0\le r< n$ is another integer, we
have after \eqref{Equ:sousadditivite augemente1} the
following inequality:
\begin{equation}
b_{np+r}\le
b_{np}+b_r+2f(np+r)\le\sum_{i=0}^k\epsilon_ib_{2^in}+b_r
+2\sum_{i=0}^k\epsilon_if\Big(\sum_{j=0}^i\epsilon_j2^jn\Big)
+2f(np+r)
\end{equation}
After \eqref{Equ:sousadditivite augemente2},
we have
\begin{equation}
b_{np+r}\le\sum_{i=0}^k\epsilon_i2^ib_n+b_r+\sum_{i=1}^k\epsilon_i
\sum_{j=1}^i2^{i+1-j}f(2^{\beta+j})
+2\sum_{i=0}^k\epsilon_if(2^{i+\beta+2})+2f(2^{k+\beta+2}).
\end{equation}
Therefore,
\begin{equation*}
\begin{split}
\frac{b_{np+r}}{np+r}&\le\frac{pb_n}{np+r}+\frac{b_r}{np+r}
+2^{-k-\beta}\sum_{i=1}^k\sum_{j=1}^i2^{i+1-j}f(2^{\beta+j})\\
&\qquad+2^{-k-\beta+1}\sum_{i=0}^kf(2^{i+\beta+2})
+2^{-k-\beta+1}f(2^{k+\beta+2}).
\end{split}
\end{equation*}
Since
\[2^{-k-\beta}\sum_{i=1}^k\sum_{j=1}^i2^{i+1-j}f(2^{\beta+j})=
2^{-k-\beta}\sum_{j=1}^k\sum_{i=j}^kf(2^{\beta+j})2^{i+1-j}
\le\sum_{j=1}^kf(2^{\beta+j})2^{2-j-\beta}=
4S_{\beta+1},\] we obtain that
\begin{equation*}
\frac{b_{np+r}}{np+r}\le\frac{pb_n}{np+r}+
\frac{b_r}{np+r} +4S_{\beta+1}
+2^{-k-\beta+1}\sum_{i=0}^{k+\beta+3}f(2^{i}).
\end{equation*}
After Lemma \ref{Lem:lemme technqiue sur
suite}, we have
\[\limsup_{m\rightarrow+\infty}\frac{b_m}{m}
\le\liminf_{n\rightarrow+\infty}\Big(\frac{b_n}{n}
+4S_{\lfloor\log_2n\rfloor+1}\Big)=\liminf_{n\rightarrow+\infty}\frac{b_n}{n}.\]
Therefore, the sequence $(b_n/n)_{n\ge 1}$
converges.

For the general case, by applying the above
result on the subsequence $(b_{n_0k})_{k\ge
1}$ and the function $g(k)=f(n_0k)$, we
obtain that the sequence $(b_{n_0k}/k)_{k\ge
1}$ has a limit in $\mathbb R_{\ge 0}$. On
the other hand, if $n_0\le l<2n_0$ is an
integer, then for any integer $k\ge 1$, we
have the inequality
\begin{equation}\label{Equ:encadrement de b}b_{n_0(k+2)}-b_{2n_0-l}-f(n_0k+l)
-f(2n_0-l)\le b_{n_0k+l}\le
b_{n_0k}+b_l+f(n_0k)+f(l).\end{equation} If
we divide \eqref{Equ:encadrement de b} by
$n_0k+l$,  we obtain, by passing to the limit
$k\rightarrow+\infty$,
\[\lim_{k\rightarrow+\infty}\frac{b_{n_0k+l}}{
n_0k+l}=\lim_{k\rightarrow+\infty}\frac{b_{n_0k}}{n_0k}.\]
Since $l$ is arbitrary, the proposition is
proved.
\end{proof}

\begin{corollary}\label{Cor:inegalite de pesudo-filtreee}
Let $(a_n)_{n\ge 1}$ be a sequence of real
numbers, $f:\mathbb Z_{>0}\rightarrow\mathbb
R_{\ge 0}$ be an increasing function and
$c>0$ be a constant. Suppose that
\begin{enumerate}[1)]
\item for sufficiently large integers $n,m$,
$a_{n+m}\ge a_n+a_m-f(n)-f(m)$,
\item $a_n\le cn$ for any integer $n\ge 1$,
\item $\sum_{\alpha\ge 0}f(2^\alpha)/
{2^\alpha}<+\infty$.
\end{enumerate}
Then the sequence $(a_n/n)_{n\ge 1}$ has a
limit in $\mathbb R$.
\end{corollary}
\begin{proof}
Consider the sequence $(b_n=cn-a_n)_{n\ge 1
}$ of positive real numbers. If $n$ and $m$
are two sufficiently large integers, we have
\[b_{n+m}=c(n+m)-a_{n+m}\le cn+cm-a_n-a_m+f(n)+f(m)
=b_n+b_m+f(n)+f(m).\] After Proposition
\ref{Pro:convergence de suite a croissance
logarithmique}, the sequence $(b_n/n)_{n\ge
1}$ has a limit in $\mathbb R$. Since
$a_n/n=c-b_n/n$, the sequence $(a_n/n)_{n\ge
1}$ also has a limit in $\mathbb R$.
\end{proof}

We establish finally the vague convergence
for normalized measures associated to a
pseudo-filtered graded algebra.
\begin{theorem} \label{Thm:convergence vague de algebre gradue pseudo filtree}
Let $f:\mathbb Z_{\ge 0}\rightarrow\mathbb R_{\ge 0}$
be an increasing function such that $\sum_{\alpha\ge
0}{f(2^\alpha)}/{2^\alpha}<+\infty$, $B$ be an integral
graded $K$-algebra of finite type over $K$, which is
generated by $B_1$. Suppose that
\begin{enumerate}[i)]
\item $d=\dim B$ is strictly positive,
\item for any positive integer $n$, $B_n$
is equipped with an $\mathbb R$-filtration
$\mathcal F$ which is separated, exhaustive
and left continuous such that $B$ is an
$f$-pseudo-filtered graded $K$-algebra,
\item $\displaystyle\limsup_{n\rightarrow+\infty}\sup_{0\neq a\in B_n}\frac{\lambda(a)}{n}<+\infty$.
\end{enumerate}
If for any integer $n> 0$, we write
$\nu_n=T_{\frac 1n}\nu_{B_n}$, then the
supports of $\nu_n$ ($n\ge 1$) are uniformly
bounded and the sequence of measures
$(\nu_n)_{\ge 1}$ converges vaguely to a
Borel probability measure on $\mathbb R$.
\end{theorem}
\begin{proof}
We apply the proof of Theorem
\ref{Thm:convergence vague de algebre gradue
filtree} in making some modifications. First
we replace the inequality
(\ref{Equ:estimation de lambda min}) by
$\lambda_{m+n}^{\min}=\lambda_n^{\min}+\lambda_{m}^{\min}
-f(n)-f(m)$ for all sufficiently large
integers $m,n$. After Corollary
\ref{Cor:inegalite de pesudo-filtreee}, the
sequence $(\lambda_n^{\min}/n)_{n\ge 1}$
converges, so is bounded from below.

For the first step, since $\sum_{\alpha\ge
0}f(2^\alpha)/2^\alpha<+\infty$, we have
$\displaystyle\lim_{\alpha\rightarrow+\infty}f(2^\alpha)
/2^\alpha=0$. As $f$ is an increasing function,
$\displaystyle\lim_{n\rightarrow+\infty}f(n)/n=0$.
Therefore, the first step of the proof of Theorem
\ref{Thm:convergence vague de algebre gradue filtree}
remains valid. Moreover, the third step is a formal
argument for the vague convergence condition, and
therefore works without problem. It remains to verify
that for any homogeneous element $x$ of $B$, the graded
$A$-module $Ax$, equipped with the inverse image
filtration, satisfies the vague convergence condition.
This corresponds to the second step of the proof of
Theorem \ref{Thm:convergence vague de algebre gradue
filtree}. Finally, no modification to the second step
is necessary since in inequalities
(\ref{Equ:quasi-fitlree 1}) and (\ref{Equ:quasi-fitlree
2}), involves only the product of {\bf two} homogeneous
elements in $B$.
\end{proof}

\begin{corollary}
With the notations of Theorem
\ref{Thm:convergence vague de algebre gradue
pseudo filtree}, the polygons associated to
probability measures  $\nu_n$ converge
uniformly to a concave function on $[0,1]$.
\end{corollary}

\begin{remark}
Instead of supposing that $B$ is generated by
$B_1$, if we suppose that $B_n$ is non-zero
for sufficiently large $n$, Theorem
\ref{Thm:convergence vague de algebre gradue
pseudo filtree} remains true, we have also
the uniform convergence of polygons.
\end{remark}


\end{document}